\documentclass{amsart}
\usepackage{marktext}
\usepackage{gimac}
\usepackage{caption}
\usepackage{graphicx}

\newcommand{\GI}[2][]{\sidenote[colback=yellow!20]{\textbf{GI\xspace #1:} #2}}

\newcommand{\EX}{\E}
\newcommand{\SL}{\mathit{SL}}

\newcommand{\Prob}{\P}
\newcommand{\mix}{\mathrm{mix}}
\newcommand{\dis}{\mathrm{dis}}
\newcommand{\tmix}{t_{\mix}}
\newcommand{\tdis}{t_{\dis}}
\newcommand{\ndis}{N_{\dis}}
\newcommand{\nmix}{N_{\mix}}

\newcommand{\comp}{c}

\DeclarePairedDelimiter{\ip}{\langle}{\rangle}
\newcommand{\Pro}{\mathbb{P}^1}
\newcommand{\cL}{\mathcal{L}}
\newcommand{\cH}{\mathcal{H}}

\newcommand{\osc}{\mathrm{osc}}

\colorlet{darkgreen}{green!40!black} 
\begin{document}
	\title[Speeding up Langevin Dynamics by Mixing]{Speeding up Langevin Dynamics by Mixing}

	\author[Christie]{Alexander Christie}
	\address{%
		Department of Mathematics,
		Penn State University,
		State College, PA 16803.}
	\author[Feng]{Yuanyuan Feng}
	\email{acc5843@psu.edu}
	\address{School of  Mathematical Sciences, Shanghai Key Laboratory of PMMP, East China Normal University, Shanghai, 200241, P.R. China.
	}
	\email{yyfeng@math.ecnu.edu.cn}
	\author[Iyer]{Gautam Iyer}
	\address{%
		Department of Mathematical Sciences,
		Carnegie Mellon University,
		Pittsburgh, PA 15213.}
	\email{gautam@math.cmu.edu}
	\author[Novikov]{Alexei Novikov}
	\address{%
		Department of Mathematics,
		Penn State University,
		State College, PA 16803.}
	\email{novikov@psu.edu}
	\begin{abstract}
		We study an overdamped Langevin equation on the $d$-dimensional torus with stationary distribution proportional to~$p = e^{-U / \kappa}$.
		When~$U$ has multiple wells the mixing time of the associated process is exponentially large (of size~$e^{O(1/\kappa)}$).
		We add a drift to the Langevin dynamics (without changing the stationary distribution) and obtain quantitative estimates on the mixing time.
		We show that an exponentially mixing drift can be rescaled to make the mixing time of the Langevin system arbitrarily small.
		For numerical purposes, it is useful to keep the size of the imposed drift small, and we show that the smallest allowable rescaling ensures that the mixing time is $O( d/\kappa^2)$, which is an order of magnitude smaller than~$e^{O(1/\kappa)}$.

		We provide one construction of an exponentially mixing drift, although with rate constants whose~$\kappa$-dependence is unknown.
		Heuristics (from discrete time) suggest that $\kappa$-dependence of the mixing rate is such that the imposed drift is of size~$O(d / \kappa^3)$.
		The large amplitude of the imposed drift increases the numerical complexity, and thus we expect this method will be most useful in the initial phase of Monte Carlo methods to rapidly explore the state space.
	\end{abstract}
	\thanks{This work has been partially supported by the National Science Foundation under grants
		DMS-1813943,
	  DMS-2108080,
		the Science and Technology Commission of Shanghai Municipality (No. 22DZ2229014) under grant 23YF1410300,
	  and the Center for Nonlinear Analysis.
  }
	\maketitle

\section{Introduction}\label{s:intro}

Sampling from a given target distribution is a problem that arises in many modern applications,
such as molecular dynamics~\cite{TurqLantelmeEA77, QuigleyProbert04, LiPaquetEA15, MazzolaSorella17, ArnonRabaniEA20}, machine learning~\cite{AndrieuFreitasEA03}, field theory~\cite{BatrouniKatzEA85}, Bayesian Statistics and computational physics~\cite{GelmanCarlinEA13}.
A typical situation of interest is to draw samples from a probability distribution with density proportional to
\begin{equation}\label{e:pdef}
	p = e^{-U / \kappa} \,.
\end{equation}
Here~$U$ is a potential function that is usually regular and explicit, and~$\kappa > 0$ is a small parameter.

Even though~$U$ is known explicitly, sampling from the above distribution is a numerically challenging problem with a long history~\cite{MetropolisRosenbluthEA53, Hastings70, Neal96, RaginskyRakhlinEA17, Cheng2020, BlancaCaputoEA21, ChewiGerberEA22}.
To briefly explain the difficulties involved, note first that in order to convert~$p$ into a probability density function, we need to normalize it by setting
\begin{equation}\label{e:pdfintro}
	\rho_\infty
		=\frac{p}{Z}
		= \frac{e^{-U/\kappa}}{Z} \,,
	\quad\text{where}\quad
	Z = Z(\kappa) = \int_{\R^d} p \, dx\,.
\end{equation}
Unfortunately, the constant~$Z$ is not easy to compute explicitly as numerical integration via quadrature is too expensive in high dimensions (see~\cite{Novak16} and references therein).

Moreover, even if the normalization constant~$Z$ is known, the majority of the mass of~$\rho_\infty$ is typically 
concentrated in a region with volume~$O(\kappa^{d/2})$, where~$p$ is relatively larger.
In order to effectively sample from~$\rho_\infty$, we need to identify this region.
There is no obvious way to do this without evaluating~$p$ everywhere, a task that is computationally infeasible in high dimensions.

There are many numerical algorithms designed to address these issues.
The first such algorithm was the celebrated Metropolis--Hastings algorithm~\cite{MetropolisRosenbluthEA53, Hastings70, LevinPeres17} which was designed to sample from the density~$\rho_\infty$ without knowledge of the normalization constant~$Z$.
Subsequently, numerous methods were developed to improve the convergence rate and address deficiencies in the Metropolis--Hastings algorithm.
Some popular methods include Hamiltonian Monte Carlo (HMC), Langevin Monte Carlo, Metropolis Adjusted Langevin algorithm (MALA), 
and various stochastic gradient methods~\cite{AndrieuFreitasEA03, Diaconis09, Betancourt17, LevinPeres17, GlattHoltzKrometisEA21,  GaoGurbuzbalabanEA22}.

Of these, one that will be of particular interest to us is the Langevin Monte Carlo method.
This method hinges on the fact that~$\rho_\infty$ is the stationary distribution of an over-damped Langevin equation, and so one can sample from~$\rho_\infty$ by performing Monte Carlo simulations.
To elaborate, consider the over-damped Langevin equation
\begin{equation}\label{e:langevin}
	dX_t = -\grad U(X_t) \, dt + \sqrt{2\kappa} \, dW_t\,,
\end{equation}
where~$W$ is a standard~$d$ dimensional Brownian motion.
It is easy to see that the density $\rho_\infty$ is stationary for the process~$X$.
If~$X$ is mixing, then the density of~$X_t$ for large enough~$t$ will be close to~$\rho_\infty$, and so performing Monte Carlo simulations on~\eqref{e:langevin} will allow us to sample from~$\rho_\infty$.

It is well known that if~$U$ is strongly convex, then the process~$X$ is exponentially mixing~\cite{BakryEmery85}.
More generally, if the stationary distribution satisfies a Poincar\'e inequality or log-Sobolev inequality 
(see for instance~\cite[A.19]{Villani09} and~\cite{VempalaWibisono19}), 
then the process~$X$ is exponentially mixing, and one can sample from~$\rho_\infty$ by simulating~\eqref{e:langevin}.
This leads to many sampling results such as~\cite{DalalyanTsybakov12, Dalalyan17, Dalalyan17a, DurmusMoulines17, Chewi23}, with guaranteed bounds on the convergence rate.

If~$U$ is not convex, however, the convergence rate could be extremely slow.
Indeed, it is well known that near non-degenerate local minima of~$U$, the process~$X$ can get trapped for time~$e^{O(1/\kappa)}$, which is extremely long.
This phenomenon is known as \emph{metastability}, and has been extensively studied (see for instance~\cite{
	SchussMatkowsky79,
	Schuss80,
	BovierEckhoffEA04,
	BovierGayrardEA05,
	FreidlinWentzell12,
	MenzSchlichting14
}).
As a result, one has to wait an extremely long amount of time before~\eqref{e:langevin} generates good samples.

The main contribution of this paper is to show that metastable points can be completely avoided by adding a ``sufficiently mixing'' drift to~\eqref{e:langevin}.
We will show (Theorem~\ref{t:FastConvIntro}, below),
that this will guarantee that the distribution of~$X_t$ is~$L^1$ close to the stationary distribution in time which a polynomial in~$1/\kappa$.
The added drift, however, does come with an increased computational cost.
In order for our method to work, the drift must be ``sufficiently mixing'' which requires it to be large.

To focus on the issue of metastability, we will ignore issues at infinity by restricting our attention to the compact torus~$\T^d$.
We expect our results can be generalized to the situation where~$U$ is strongly convex outside  a compact region, and the added drift vanishes outside this region.
If the potential~$U$ has multiple wells, metastable points will force the \emph{mixing time}%
\footnote{
	Recall the mixing time (defined in~\eqref{e:tmixdef} below), measures the rate at which the distribution of~$X$ converges to the stationary distribution~\cite{MontenegroTetali06,LevinPeres17} in total variation.
}
of~$X$ to be~$e^{O(1/\kappa)}$.
For~$\kappa \ll 1$ this is too large to be practical.

Our main result (Theorem~\ref{t:FastConvIntro}, below) reduces the mixing time to a polynomial in~$1/\kappa$ by adding a sufficiently mixing drift.
Explicitly, the modification of~\eqref{e:langevin} we consider is
\begin{equation}\label{e:Aeq}
		dX_t = A v_{At}(X_t) \, dt -\grad U(X_t) \, dt + \sqrt{2 \kappa} \, dW_t\,,
		\quad\text{on the torus } \T^d  \,.
\end{equation}
Here~$A \gg 1$ is a large parameter, and~$v$ is a time dependent uniformly Lipschitz flow such that
\begin{equation}\label{e:measurePreserving}
	\kappa \dv v - \grad U \cdot v = 0\,.
\end{equation}
Note equation~\eqref{e:measurePreserving} is equivalent to the condition that~$\dv (\rho_\infty  v) = 0$, which implies 
that the stationary distribution of~\eqref{e:Aeq} is still~$\rho_\infty$. 

The over damped Langevin system~\eqref{e:Aeq} with a drift satisfying~\eqref{e:measurePreserving} has been studied before by several authors.
In certain ways this system always converges faster to equilibrium faster than~\eqref{e:langevin} (see for example~\cite{HwangHwangMaEA93, ReyBelletSpiliopoulos15, DuncanPavliotisEA17,  HuWangEA20}).
Prior to our work, the increased convergence rate was obtained by taking $v(x) = J \nabla U(x)$ 
for an antisymmetric matrix~$J$~\cite{LelievreNierEA13, DuncanLelievreEA16, GuillinMonmarche16, LuSpiliopoulos18}.
With this approach, however, the mixing time is still~$e^{O(1/\kappa)}$, but with smaller constants than the mixing time of~\eqref{e:langevin}.
Using a different approach, Damak \etal~\cite{DamakFrankeEA20} produce a sequence of time independent flows in~$\R^2$ which make the mixing time arbitrarily small.
Their construction relies on a strong oscillation of stream lines of the imposed drift, and only applies to the two dimensional case with a quadratic potential.
\smallskip

The first result in this paper provides a quantitative estimate of the mixing time of~$X$, denoted by~$\tmix = \tmix(\kappa)$, when the deterministic flow of~$v$ is \emph{exponentially mixing}.%
\footnote{
	Here the (deterministic) flow of~$v$ is assumed to be exponentially mixing, in the sense of dynamical systems.
	We recall this notion in Section~\ref{s:mixing}, below.
}
Our result estimates mixing time in terms of the \emph{dissipation time}, denoted by~$\tdis$, which measures the rate at which $X_t$ converges to the stationary distribution in the~$L^2$ norm (the precise definition is in Section~\ref{s:dtime}, below).

\begin{theorem}\label{t:FastConvIntro}
	Suppose the vector field~$v$ satisfies~\eqref{e:measurePreserving} and generates an exponentially mixing flow.
	Denote the mixing rate by
	\begin{equation}\label{e:expMixRate}
		h(t)=D e^{-\gamma t}\,,
	\end{equation}
	where~$D$ and~$\gamma$ are constants that may depend on~$\kappa$.
	There exists constants $C$ and $A_0 = A_0(\kappa) < \infty$ such that
	\begin{gather}
		\label{e:TDisIntro}
		\tdis
			\leq 
				\frac{C \norm{\grad v}_{L^\infty}}{\gamma^2 A}
				\left(1+\ln^2\frac{D\gamma^2A}{\kappa\norm{\nabla v}_{L^\infty}}\right)\,,
		\\
		\label{e:TMixIntro}
		\tmix
			\leq C d\paren[\Big]{ 1+\frac{\norm{U}_{\osc}}{\kappa}-\ln(\kappa\tdis)}\tdis\,,
	\end{gather}
	for all sufficiently small~$\kappa > 0$.
	Here $\norm{U}_\osc = \max U - \min U$,
	\begin{equation*}
		\norm{\grad v}_{L^\infty} =
			\norm[\Big]{  \sum_{i,j} \abs{\partial_i v_j}^2 }_{L^\infty}^{1/2}\,.
	\end{equation*}
\end{theorem}
\begin{remark}\label{r:A0choiceExp}
	We clarify that the constant~$C$ is independent of both the dimension and~$\kappa$.
	Clearly both~$\tdis$ and~$\tmix$ vanish as~$A \to \infty$.
	However, when~$A$ is large, solving~\eqref{e:Aeq} is computationally expensive, and thus we would like to choose~$A$ to be as small as possible.
	From the proof (see Remark~\ref{r:A0choice}, below) we will show that~$A_0$ can be chosen according to
	\begin{equation}\label{e:A0choiceExp}
		A_0 = \frac{C' \kappa \norm{\grad v}_{L^\infty} }{\gamma^2} \ln^2 \frac{ C' D}{\kappa}  \,,
	\end{equation}
	for some constant~$C'$ which is independent of~$\kappa$.
	Thus if we choose~$A = A_0$ in Theorem~\ref{t:FastConvIntro}, then the bounds~\eqref{e:TDisIntro} and~\eqref{e:TMixIntro} reduce to the polynomial bounds
	\begin{equation}\label{e:PolyMix}
		\tdis \leq \frac{C''}{\kappa}
		\quad\text{and}\quad
		\tmix \leq \frac{C'' d}{\kappa^2}\,,
	\end{equation}
	for some constant~$C''$ that is independent of~$\kappa$ and~$d$.
	This is an order of magnitude smaller than the mixing time of~\eqref{e:langevin} which is~$e^{O(1/\kappa)}$ for multi-modal distributions.
\end{remark}

We were unable to use the ``usual techniques'' (e.g.\ coupling, Cheeger bounds, etc.~\cite{LevinPeres17,MontenegroTetali06}) to obtain the mixing time bound~\eqref{e:TMixIntro}.
The proof of Theorem~\ref{t:FastConvIntro} instead uses a PDE based Fourier splitting method to obtain the dissipation time bound~\eqref{e:TDisIntro} (see Theorem~\ref{t:fconv}, below), and then estimates the mixing time in terms of the dissipation time (Proposition~\ref{p:tdisTmix}, below).
Postponing further discussion of these ideas to Section~\ref{s:FastConvProof}, we now explicitly construct 
velocity fields that are exponentially mixing so that Theorem~\ref{t:FastConvIntro} may be applied.

Notice first that a large family of velocity fields satisfying~\eqref{e:measurePreserving} can be easily constructed by taking skew gradients.
Indeed, if $\psi$ is a periodic stream function, and $i, j \in \set{1, \dots, d}$  with $i \neq j$, then any velocity field~$v$ defined by
\begin{equation}\label{e:vDef}
	v
	\defeq \frac{1}{p} \grad^\perp_{i,j} (p \psi)
	= \grad^\perp_{i,j} \psi - \frac{\psi \grad^\perp_{i,j} U}{\kappa}\,,
\end{equation}
satisfies the measure preserving condition~\eqref{e:measurePreserving}.
Here $\grad^\perp_{i,j}$ is the skew gradient in the $x_i$-$x_j$ plane, and is defined by
\begin{equation}\label{e:GradPerpij}
	\grad^\perp_{i,j} \psi = -\partial_j \psi \bm e_i + \partial_i \psi \bm e_j\,,
\end{equation}
where $\bm e_i, \bm e_j$ are the standard $i^\text{th}$ and $j^\text{th}$ basis vectors respectively.
If~$\psi$ is a function of only one coordinate, say $\psi(x) = F(x_i)$, and~$p$ is identically constant, then the velocity field~$v$ above is simply a shear flow with with magnitude $F'(x_i)$ directed along the $j^\text{th}$ coordinate axis.
If~$\psi(x) = F(x_i)$ as above, but~$p$ is not identically constant, then we will call~$v$ a ``modified shear''.
In this case we note that the velocity field lies in the $x_i$-$x_j$ plane, but may not necessarily be directed along~$\bm e_j$.
Moreover, the magnitude now depends on all coordinates, and not just $x_j$.

\begin{figure}[htb]
	\centering
	\begin{minipage}{.4\textwidth}
		\centering
		\includegraphics[width=1\linewidth]{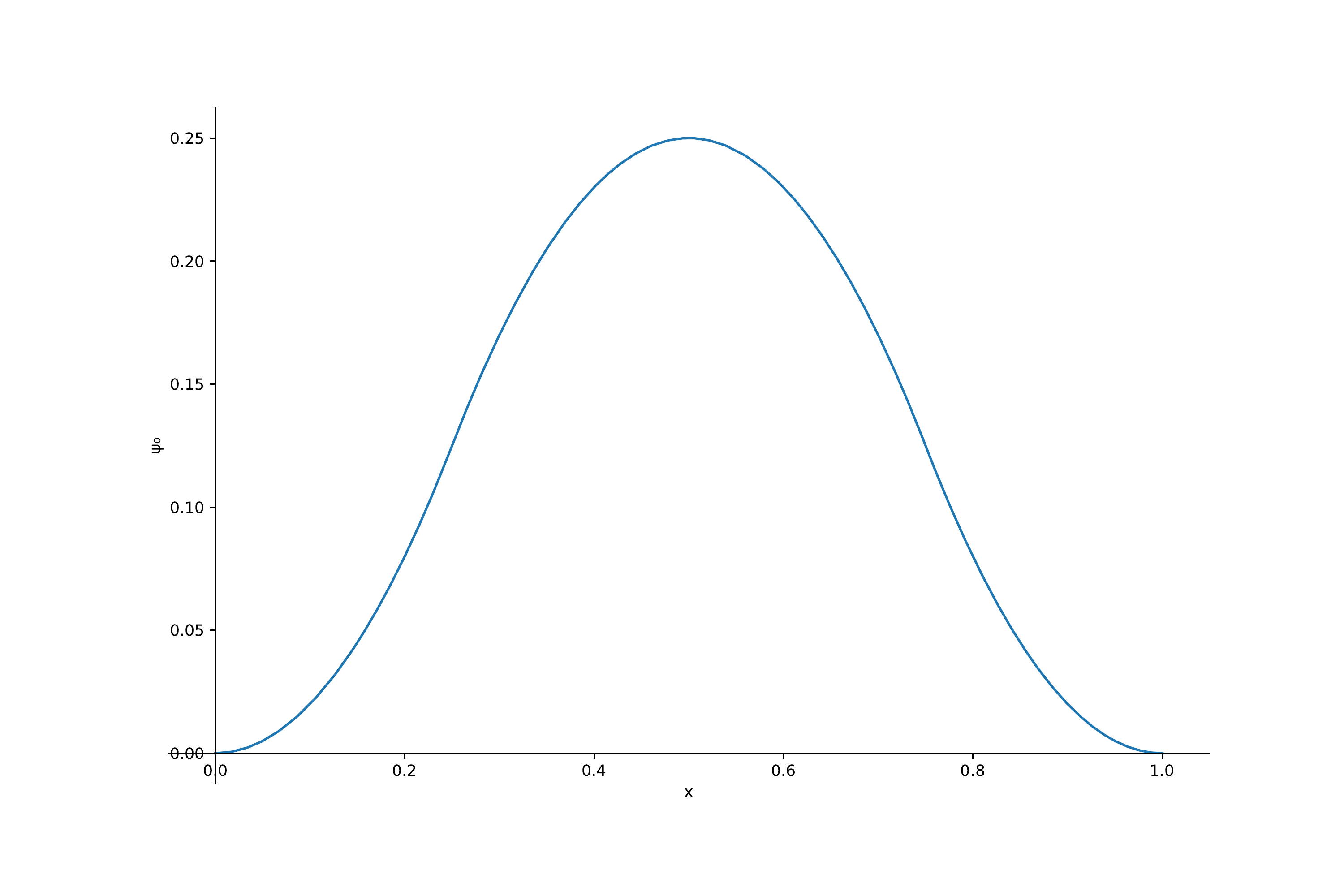}
	\end{minipage}%
	\quad
	\begin{minipage}{.4\textwidth}
		\centering
		\includegraphics[width=1\linewidth]{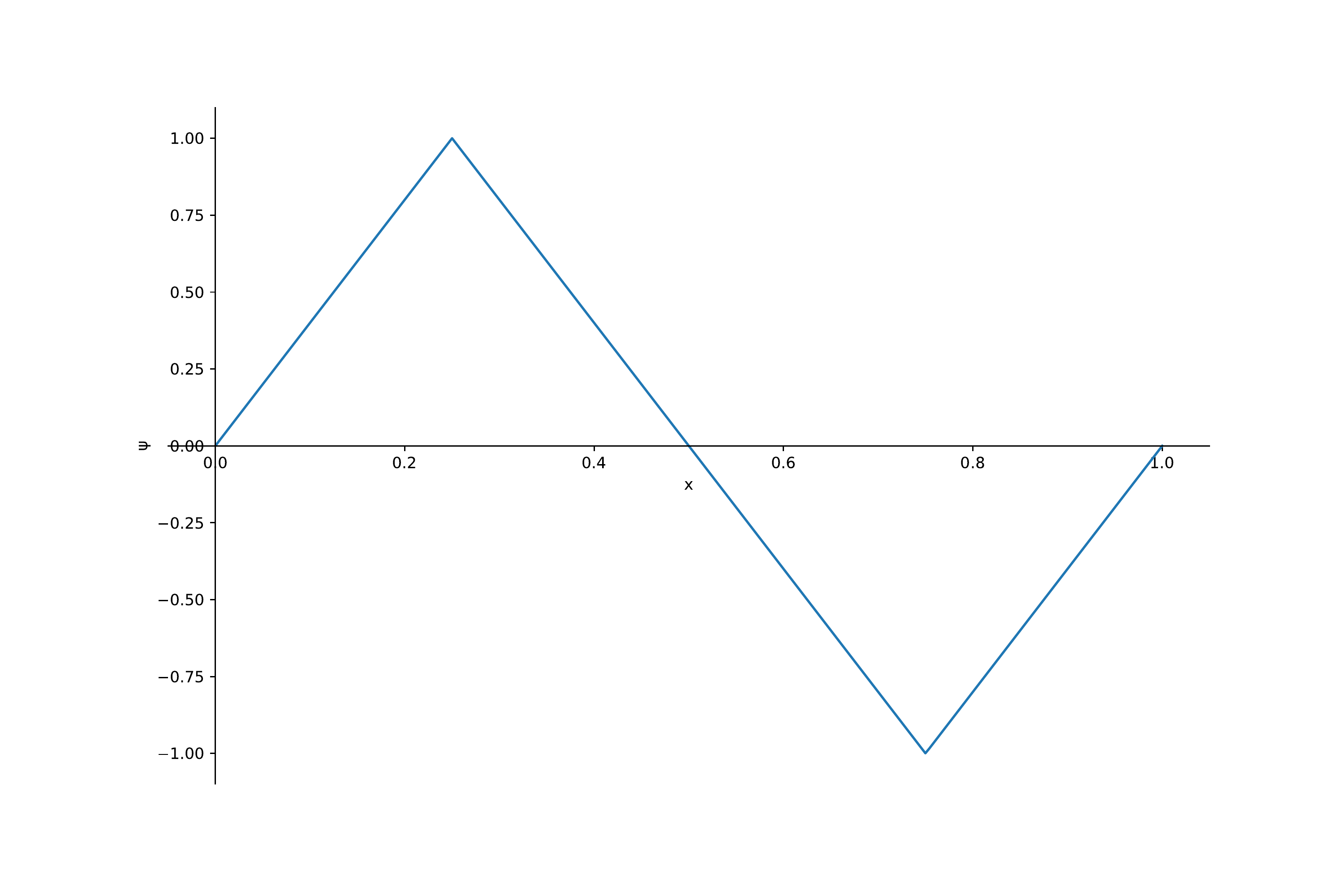}
	\end{minipage}%
	\caption{The function~$F$ (left) and its sawtooth shaped derivative~$F'$ (right).}
	\label{f:sawTooth}
\end{figure}
We will now construct exponentially mixing velocity fields using randomly shifted alternating modified shears, in the spirit of Pierrehumbert~\cite{Pierrehumbert94} who used a similar construction to study mixing in fluid dynamics.
Explicitly, choose $\alpha_n \in [0, 1]$, $\beta_n \in [0, 1]$ and $i_n, j_n \in \set{1, \dots, d}$ to be uniformly distributed, i.i.d.\ random variables such that $i_n \neq j_n$.
Given a periodic function~$F \colon \R \to \R$ define
\begin{equation}\label{e:vDefAltShear}
	v_t(x)
		\defeq \frac{\beta_n}{p(x)} \grad^\perp_{i_n, j_n} \paren[\big]{ p(x) F( x_{i_n} - \alpha_n ) }
	\,,
	\quad\text{when } t \in [n, n+1)\,.
\end{equation}
We will either choose~$F(x) = \sin(2\pi x)$ (as in~\cite{Pierrehumbert94}), or choose $F$ so that the derivative is a sawtooth shaped function 
(see Figure~\ref{f:sawTooth}, or the exact formula~\eqref{e:psi0} in Section~\ref{s:modifiedTentConstruction}, below).
We claim that that~$v$ is exponentially mixing with probability~$1$.

\begin{theorem}\label{t:SinMix}
	Suppose the potential~$U$ is $C^2$, and~$F$ is the function with sawtooth derivative shown in Figure~\ref{f:sawTooth}.
	If~$d \geq 3$, suppose further there exists a small ball $B(\hat x, \hat \epsilon) \subseteq \T^d$ such that
	\begin{equation}\label{e:highddegeneracycond}
		U(x)=\sum_{i=1}^d U_i(x_i)
		\quad\text{ for all } x \in B(\hat x, \hat \epsilon)\,.
	\end{equation}
	Then there exists a constant $\gamma = \gamma(\kappa , d) < \infty$, and a finite random variable~$D = D(\kappa, d) $ such that almost surely the velocity field~\eqref{e:vDefAltShear} is exponentially mixing with rate~\eqref{e:expMixRate}.

	If instead $F(x) = \sin( 2\pi x )$, then the same conclusions hold provided we also assume the critical points of~$U$ are isolated.
\end{theorem}

\begin{remark}
	While the cosine shears (corresponding to~$F(x) = \sin( 2\pi x)$) are more stable numerically, there are many common distributions (such as the Rosenbrock distribution~\cite{Rosenbrock60}) where the critical points of~$U$ are not isolated.
	In this case, we believe the velocity field~$v$ is still exponentially mixing for~$F(x) = \sin( 2 \pi x )$, however, certain technical aspects of our proof break down.
	When~$F'$ is a sawtooth-shaped function no assumption on critical points of~$U$ is required.

	The condition~\eqref{e:highddegeneracycond} is present only for technical reasons, and Theorem~\ref{t:SinMix} may still be true without it.
	We emphasize, however, that the condition is only required in an (arbitrarily) small ball $B(\hat x, \hat \epsilon)$.
	Any~$C^1$ potential can be modified slightly in a small region to ensure~\eqref{e:highddegeneracycond} holds.
\end{remark}

\begin{remark}\label{r:Rd}
	Before proceeding further we briefly comment on the situation when the state space is~$\R^d$, and not the compact torus.
	First even in the case that $U = \abs{x}^2$, the mixing time in~$\R^d$ is infinite.
	This is because an initial distribution concentrated a distance of~$R$ away from the origin will take time $O(\ln R)$ to mix.
	Thus in order to formulate Theorem~\ref{t:FastConvIntro} in~$\R^d$ one would have to either restrict to mixing times of initial distributions that have support in the same compact set, or only consider the dissipation time as in~\eqref{e:TDisIntro}.
	The bound on the dissipation time works with one additional assumption on~$U$ which we state in Remark~\ref{r:tDisRd}, below.

	In order to apply Theorem~\ref{t:FastConvIntro} in~$\R^d$, however, we would need to construct flows on~$\R^d$ that are exponentially mixing with respect to the density~$\rho_\infty$.
	This is not easy to do, and we are presently not aware of any such examples.
	Instead, a more useful approach, is to assume that~$U$ is strongly convex outside a compact region~$B$, and construct exponentially mixing flows on~$B$.
	We expect such flows can be constructed using the methods in Section~\ref{s:modifiedmixing} using a modified velocity profile, but goes beyond the scope of the present work.
	Once such flows are constructed, the structure of~$U$ will guarantee there are no metastable points outside~$B$, and inside~$B$ the mixing flow will eliminate the effects due to metastability.
\end{remark}

Unfortunately, $D$ and~$\gamma$ depend on~$\kappa$ and the dimension, and the proof of Theorem~\ref{t:SinMix} does not provide any information on the asymptotic behavior of~$D$ and~$\gamma$ as $\kappa \to 0$ and~$d \to \infty$.
We can, however, study a discrete time version of~\eqref{e:Aeq}, and produce exponentially mixing maps for which
\begin{equation}\label{e:Heusristics}
	D = O(\sqrt{d}) e^{O(1/\kappa)}\,,
	\quad
	\gamma = O(1)\,.
\end{equation}
(The precise construction is described in Section~\ref{s:heuristics}, below.)
Suppose, momentarily, that for one of the velocity fields from Theorem~\ref{t:SinMix} we still have~\eqref{e:Heusristics}.
For such velocity fields, we note that~$\norm{\grad v}_{L^\infty} = O(\sqrt{d}/\kappa)$.
Thus choosing~$A = A_0$, where $A_0$ is given by\eqref{e:A0choiceExp}, reduces to choosing
\begin{equation}
	A = O\paren[\Big]{
	\frac{\sqrt{d}}{\kappa^2}
	}\,,
\end{equation}
in order to obtain the polynomial in~$1/\kappa$ mixing time bounds stated in~\eqref{e:PolyMix}.
Note that with this choice, the drift term in~\eqref{e:Aeq} is of size~$O(d / \kappa^3)$.

\begin{figure}[htb]
		\centering
		\begin{minipage}{.3\textwidth}
			\centering
			\includegraphics[width=1\linewidth]{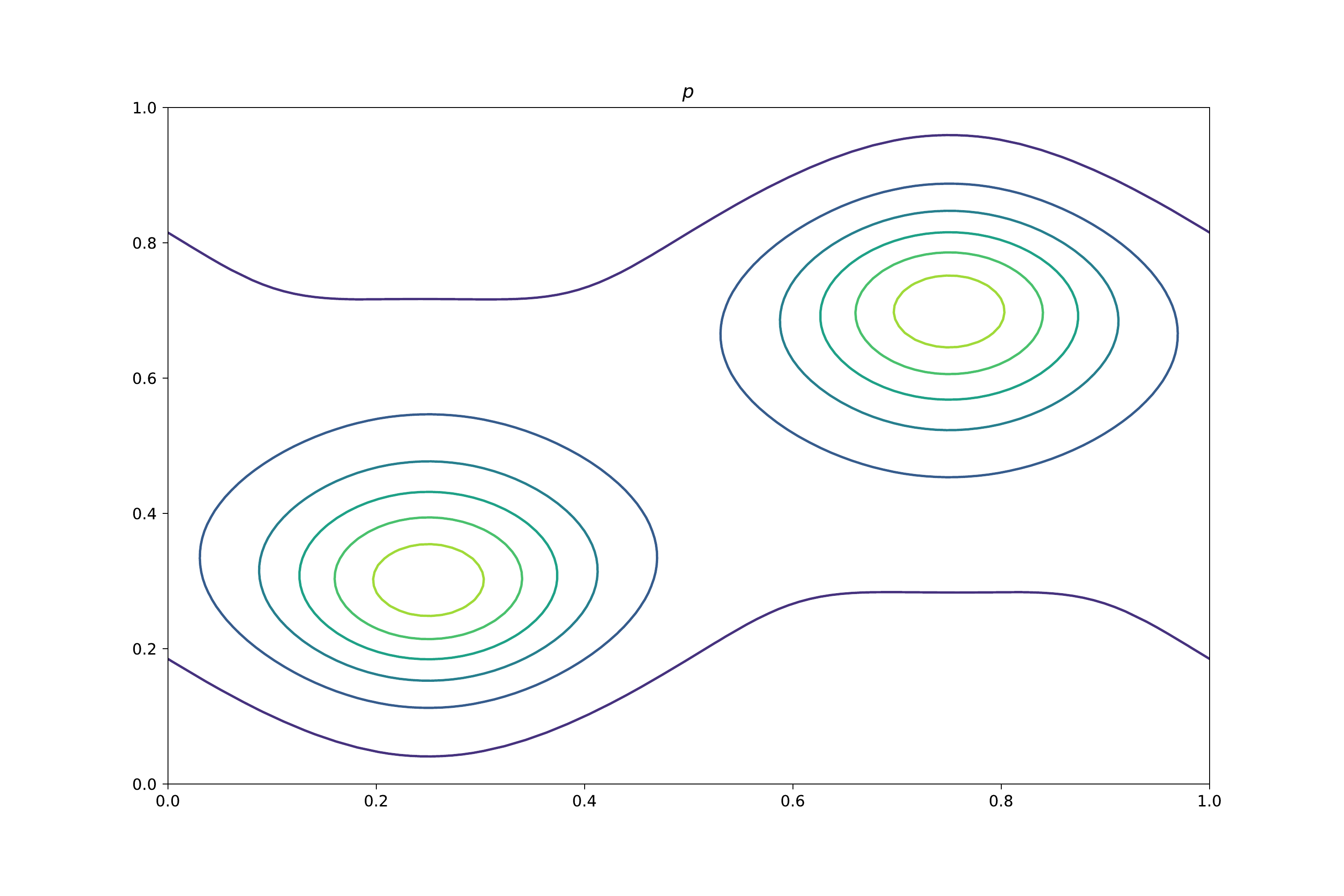}
		\end{minipage}%
		\begin{minipage}{.3\textwidth}
			\centering
			\includegraphics[width=1\linewidth]{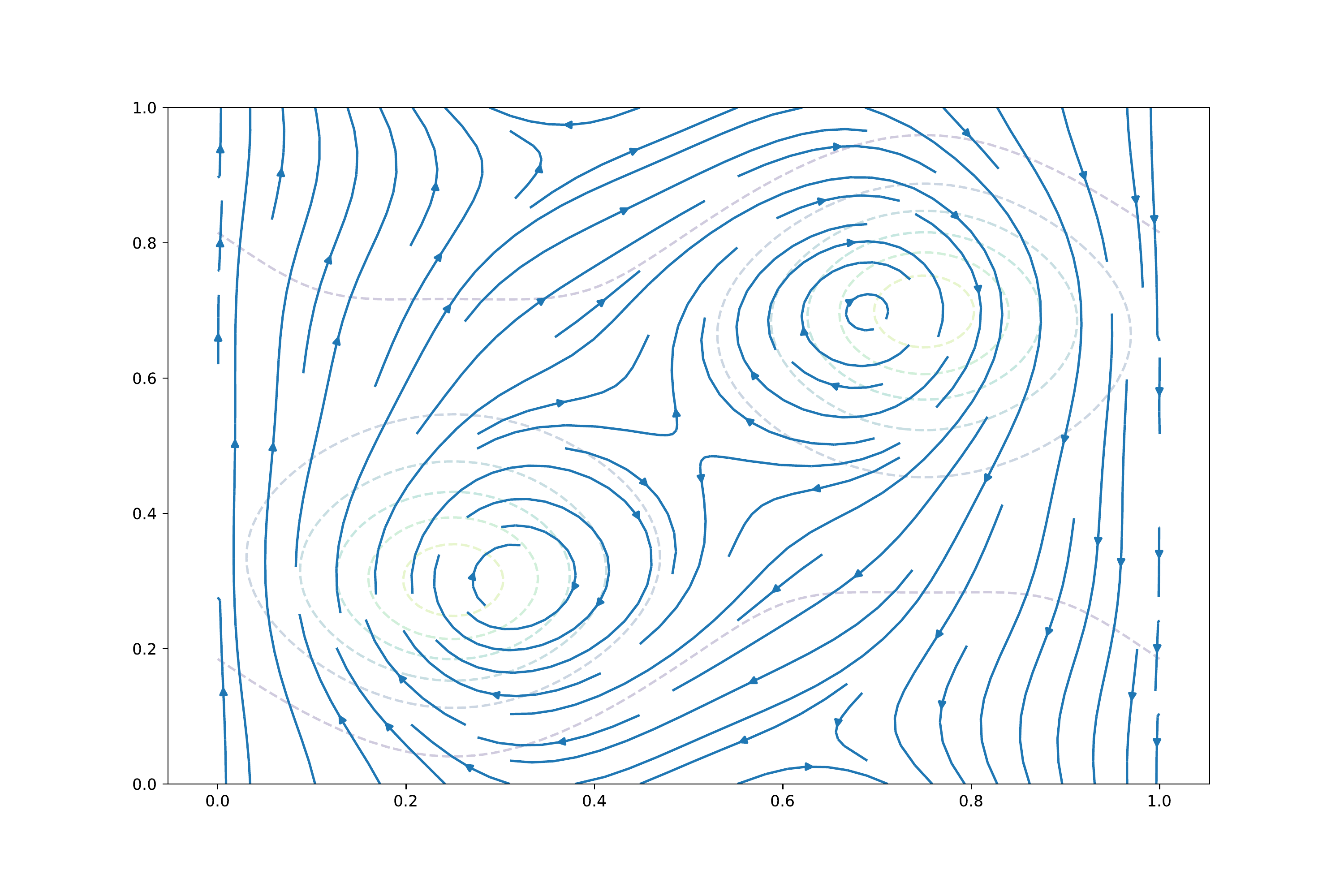}
		\end{minipage}%
		\begin{minipage}{.3\textwidth}
			\centering
			\includegraphics[width=1\linewidth]{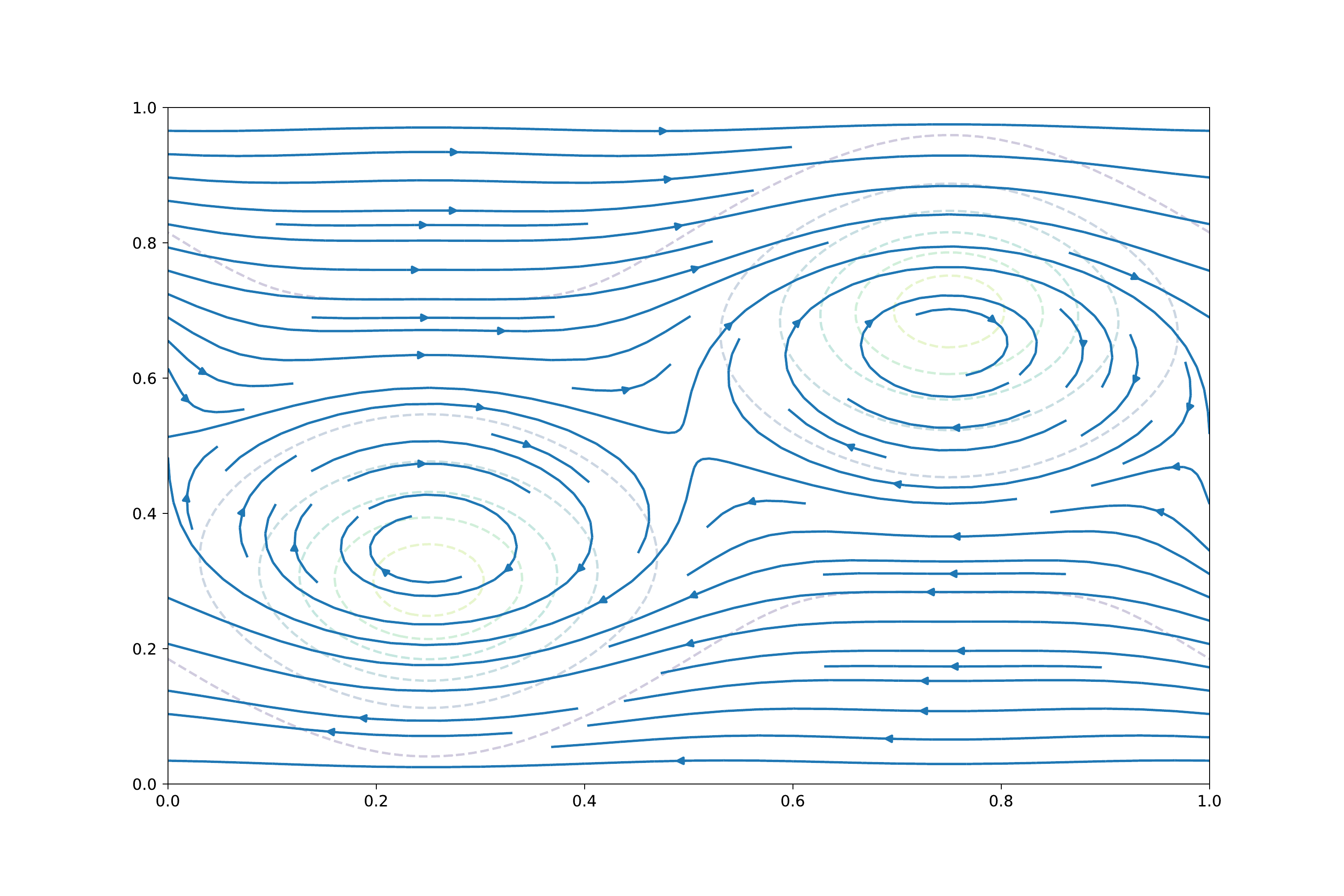}
		\end{minipage}
		\caption{Level sets of~$p$ (left), and stream plots of~$v$ (center, right) corresponding to a modified vertical and horizontal shear respectively.}
		\label{f:stream}
\end{figure}
To illustrate our results numerically, we choose a double well potential~$U$
\begin{equation}\label{e:Sinpotential}
	U \defeq (\sin^2(\pi(x_1-.75))+\sin^2(\pi(x_2-.7)))(\sin^2(\pi(x_1-.75))+\sin^2(\pi(x_2+.7))\,,
\end{equation}
which has two minima at the points $(0.25, 0.3)$ and~$(0.75, 0.7)$.
Rather than choosing~$v$ according to~\eqref{e:vDefAltShear}, it is numerically more convenient to choose
\begin{equation}\label{e:psiNumerics}
	\psi=M_t\left(\sin^2(2 \pi \omega t) \sin(2\pi (x_1-B_{1, t})+\cos^2(2\pi \omega t)\sin(2\pi(x_2-B_{2, t}))\right),
\end{equation}
and then define~$v$ according to~\eqref{e:vDef}.
Here~$\omega > 0$ is a parameter, $M$ is a mean reverting Ornstein--Uhlenbeck process, $B_{i}$ are Brownian motions, and $M$, $W$, and $B_i$ are mutually independent.
Note when $\omega t \in \pi \Z$, the stream function $\psi$ only depends on~$x_2$, and when $\omega t \in (\pi + \frac{1}{2}) \Z$, the stream function~$\psi$ only depends on~$x_1$.
Thus this is a time continuous way of choosing the velocity fields defined in~\eqref{e:vDef}, with~$\omega$ controlling the frequency at which the fields switch direction.
Level lines of the function~$p$ (equation~\eqref{e:pdef}), and a stream plot of~$v$ at times~$t = 0$ and~$t = (2\pi + 1) / (2\omega)$ are shown in Figure~\ref{f:stream}.
Of particular interest is the fact that~$v$ is not $0$ at the local minima of~$U$, and this is what allows solutions to~\eqref{e:Aeq} to quickly escape the metastable traps at critical points.

We now solve equation~\eqref{e:Aeq} numerically with $\kappa = 1/70$ and choose the initial distribution to be the delta measure located at $(0.75, 0.7)$ (one of the local minima of~$U$).
In a short amount of time solutions to both equations fill out a neighborhood of the local minimum that they start at.
However, since local minima of~$U$ are metastable traps for~\eqref{e:langevin}, very few realizations of solutions to~\eqref{e:langevin} are able to leave this neighborhood.
As a result, very few of these points are present in the other local minimum of~$U$ located at~$(0.25, .3)$ (see Figure~\ref{f:DistMixed}, left).
In contrast, solutions to~\eqref{e:Aeq} are not trapped at critical points for as long, many realizations of solutions to~\eqref{e:Aeq} quickly find their way to the other local minimum (see Figure~\ref{f:DistMixed}, center, right).
At the final time in our simulations ($T = 3.14$), the distribution of solutions to~\eqref{e:Aeq} were close to the stationary distribution, but the distribution of solutions to~\eqref{e:langevin} were not.
\begin{figure}[hbt]
	\centering
    \begin{minipage}{.3\textwidth}
		\centering
		\includegraphics[width=1\linewidth]{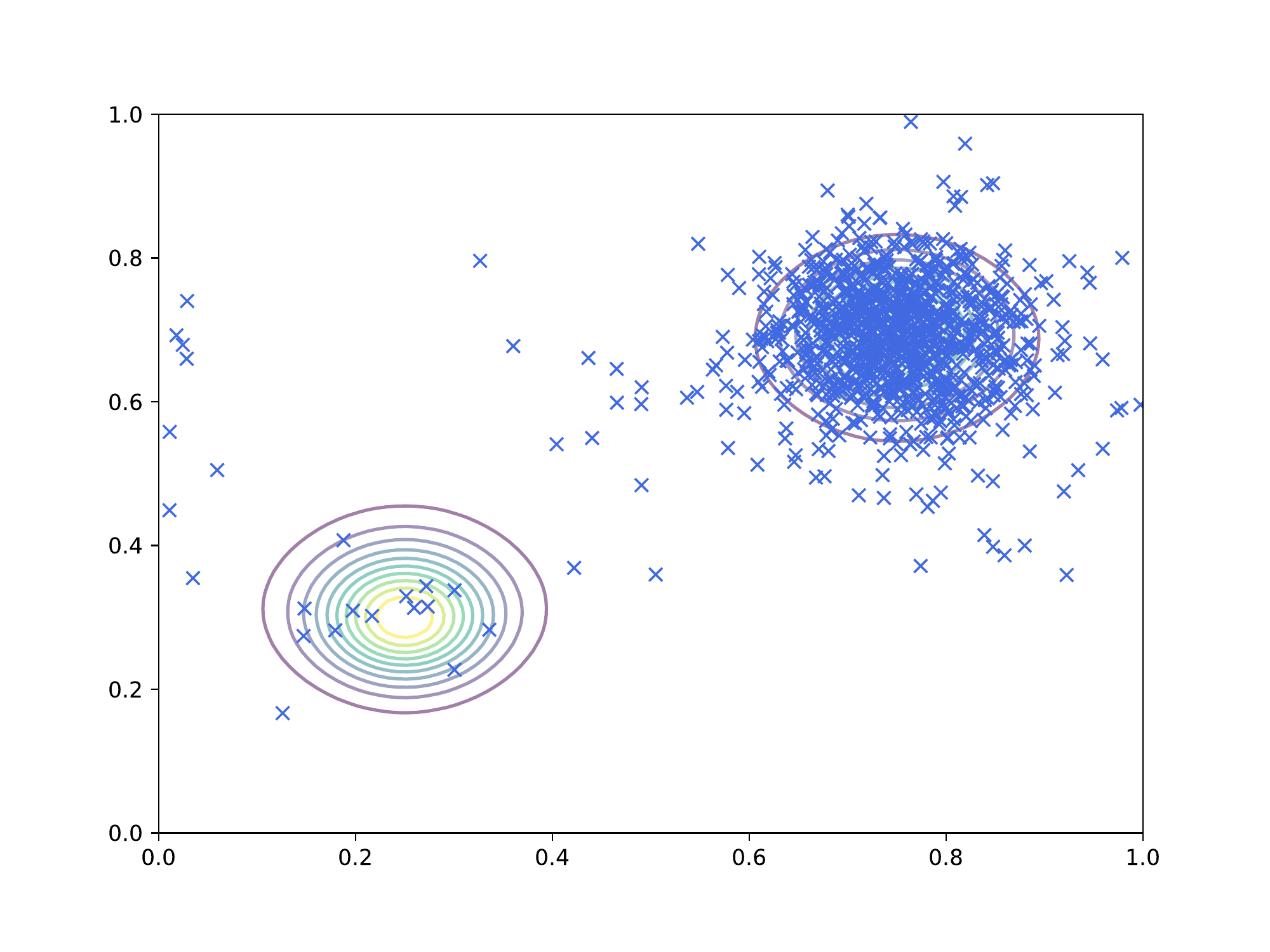}
		{\footnotesize Equation~\eqref{e:langevin} at time $T=3.14$}
	\end{minipage}%
	\begin{minipage}{.3\textwidth}
		\centering
		\includegraphics[width=1\linewidth]{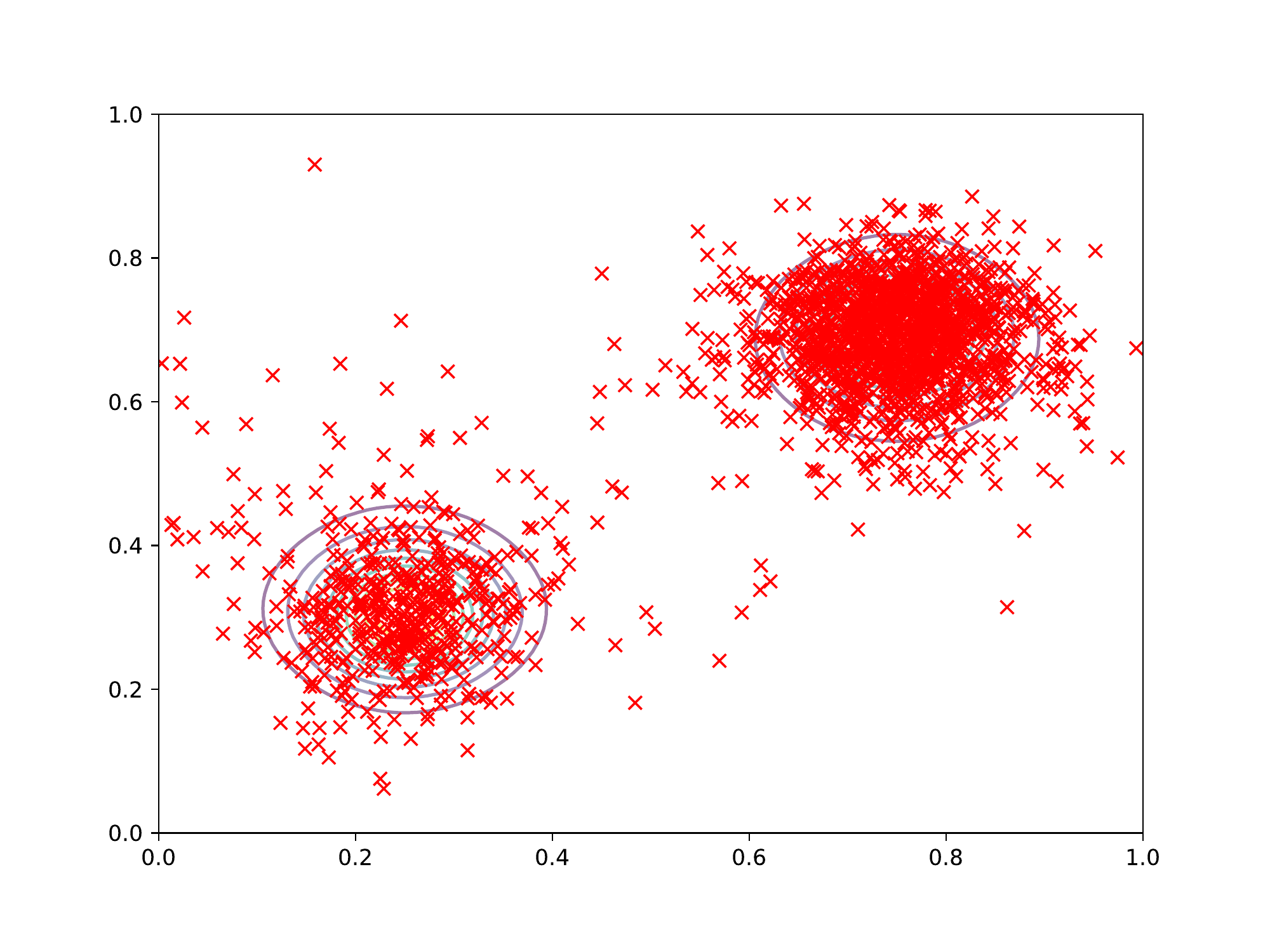}
		{\footnotesize Equation~\eqref{e:Aeq} at time $T=0.63$}
	\end{minipage}%
	\begin{minipage}{.30\textwidth}
    	\centering
    	\includegraphics[width=1\linewidth]{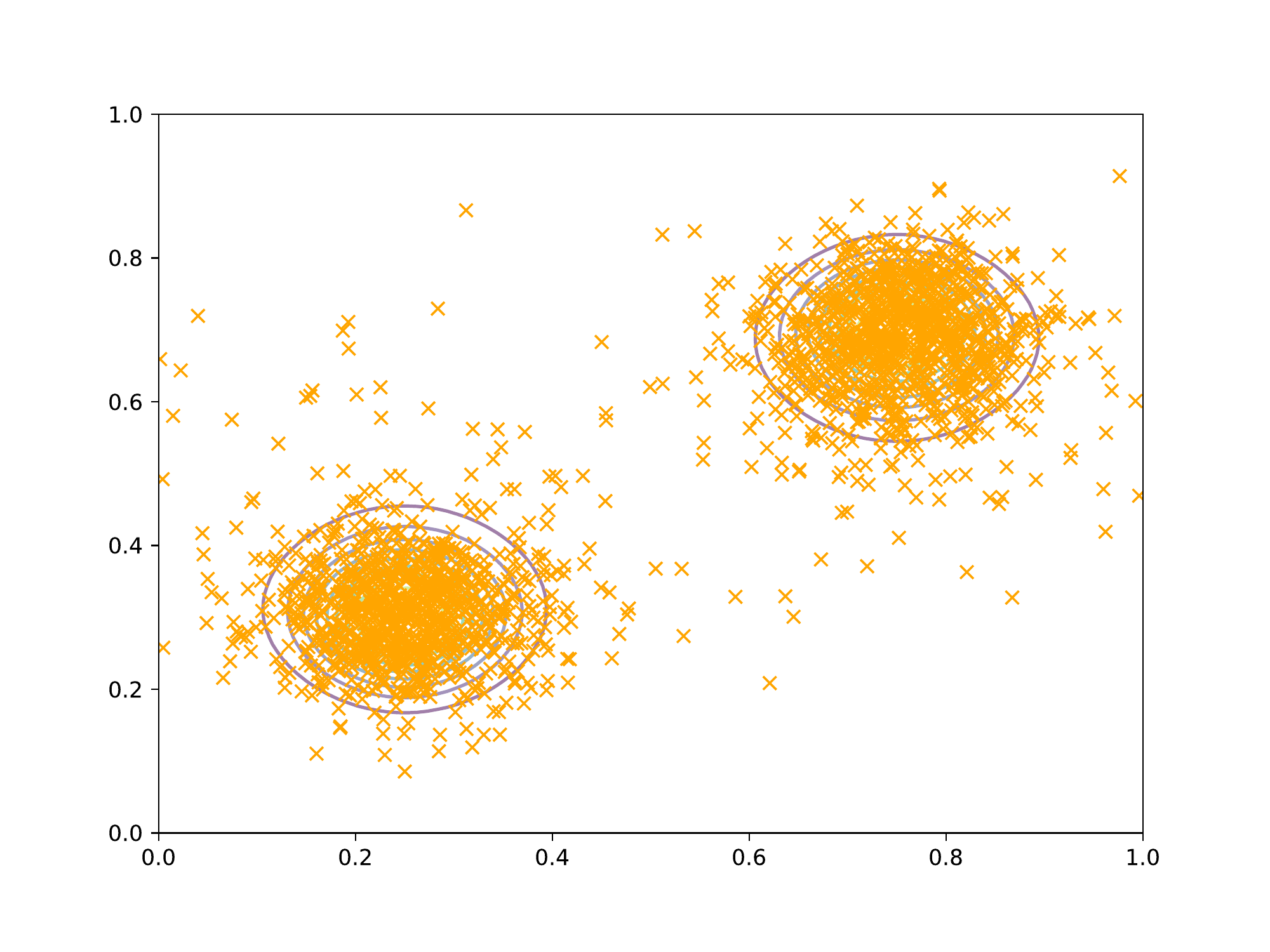}
			{\footnotesize Equation~\eqref{e:Aeq} at time $T=3.14$}
    \end{minipage}
    \caption{
			Distribution of $2000$ realizations of solutions to~\eqref{e:langevin} and~\eqref{e:Aeq} with a double well potential, $\kappa=1/70$, $A=3500$ and an adaptive time step method with~$dt \approx 10^{-5}$.
			Solutions to~\eqref{e:langevin} take a very long time to leave a neighborhood of a local minimum of~$U$, where as solutions to~\eqref{e:Aeq} leave it quickly and approach the stationary distribution much faster. 
		}
		\label{f:DistMixed}
\end{figure}
\smallskip

Despite the mixing time of~\eqref{e:Aeq} being an order of magnitude smaller than that of~\eqref{e:langevin}, there are a few issues that increase the complexity when solving~\eqref{e:Aeq} numerically.
Explicitly, solving equation~\eqref{e:langevin} is relatively easier as the largest term on the right is of size~$O(1)$.
Solving equation~\eqref{e:Aeq}, on the other hand, is relatively harder as the added drift has magnitude~$O(\sqrt{d} / \kappa)$.
The trade-off is that one only needs to solve~\eqref{e:Aeq} for time~$O(d / \kappa^2)$, where as, in order to obtain comparable results with~\eqref{e:langevin} one needs to solve it for time~$e^{O(1/\kappa)}$.
Thus we expect that algorithms using~\eqref{e:Aeq} will be useful in the initial exploratory phase of Monte Carlo methods, where accuracy is not as important.
After rapidly exploring the state space using~\eqref{e:Aeq}, it may be better to use other methods such as~\cite{
	RobertsTweedie96,
	MaChenEA15,
	BouchardCoteVollmerEA18,
	FearnheadBierkensEA18,
	BierkensFearnheadEA19,
	LuLuEA19,
	GlattHoltzKrometisEA21,
	LuSlepcevEA22
}.
\medskip

Finally, we conclude this section by stating a few questions that we are presently unable to address.
\begin{enumerate}[(1)]
	\item
		The most important question we are unable to address is to rigorously describe the asymptotic behavior of~$D$ and~$\gamma$ as~$\kappa \to 0$.
		The heuristics~\eqref{e:Heusristics} (which we can prove for a time discrete example) leads to the polynomial mixing time bounds~\eqref{e:PolyMix}, which is an order of magnitude smaller than the mixing time of~\eqref{e:langevin}.
		We are presently unaware of techniques that provide any rigorous bounds on~$D$ and~$\gamma$ as~$\kappa \to 0$.

	\item
		Since the numerical complexity increases with~$A$, it is useful to find velocity fields~$v$ where~$A_0$ can be chosen to be small.
		From~\eqref{e:A0choiceExp} this entails finding velocity fields for which~$\norm{\grad v}_{L^\infty}$ and $D$ are small, and~$\gamma$ is large.
		Are there velocity fields~$v$ for which one expects~$\gamma = O(1)$, and~$D = O(1/\kappa)$, as opposed to the bounds in~\eqref{e:Heusristics}?
  \item
		What happens when we work on $\R^d$ instead of~$\T^d$?
		The current proof techniques will still allow us to prove~\eqref{e:TDisIntro}.
		However, we are unable to prove~\eqref{e:TMixIntro}, or Theorem~\ref{t:SinMix} in~$\R^d$ as explained in Remark~\ref{r:Rd}.
	\item
		Even though we can prove that the velocity field in~\eqref{e:vDefAltShear} is almost surely exponentially mixing, it is computationally intensive as it involves changing the velocity field discontinuously at many points in time.
		It also requires choosing~$A$ large, which further increases the computational cost.
		Numerically we found the continuous time modification~\eqref{e:psiNumerics}, described above, yielded better results, presumably because of discretization artifacts.
		In the present work we make no mention of efficient discretizations of~\eqref{e:Aeq} for the (chaotic) velocity field~\eqref{e:vDefAltShear} (or the time continuous version~\eqref{e:psiNumerics}).
		Discretizations of~\eqref{e:langevin} and~\eqref{e:Aeq} have been extensively studied by many authors~\cite{
			JordanKinderlehrerEA98,
			DuncanPavliotisEA17,
			Wibisono18,
			Chewi23
		},
		and an important question is to choose velocity that minimize the computational cost of such schemes.
\end{enumerate}

\subsection*{Plan of this paper.}
	In Section~\ref{s:FastConvProof} we precisely define the notion of exponential mixing, and split the proof of Theorem~\ref{t:FastConvIntro} into two steps: Obtaining dissipation time bounds (Theorem~\ref{t:fconv}), and obtaining mixing time bounds (Proposition~\ref{p:tdisTmix}).
	We prove Theorem~\ref{t:fconv} in Section~\ref{s: mukapmix}, and prove Proposition~\ref{p:tdisTmix} in Section~\ref{s:tmixtdis}.
	In Section~\ref{s:heuristics} we study a time discrete version of~\eqref{e:Aeq} and produce exponentially mixing maps for which the $\kappa$-dependence of the mixing rate is known.
	(This is used to motivate the heuristics~\eqref{e:Heusristics}.)
	In Section~\ref{s:modifiedmixing} we prove Theorem~\ref{t:SinMix} and produce (random) velocity fields that are (almost surely) exponentially mixing.
	Finally, in Appendix~\ref{s:appendixmixing}, we show that a family of randomly shifted localized tent shaped shear flows almost surely generates an exponentially mixing flow for the Lebesgue measure.
	We provide this simpler example since the proof is simpler than the proofs done in Section~\ref{s:modifiedmixing}, and does not involve technical calculations checking the H\"ormander type conditions.

\subsection*{Acknowledgements}
The authors would like to thank
Nathan Glatt-Holtz,
Justin Krometis,
and
Dejan Slep\v cev,
for helpful comments and discussions.

\section{Mixing Time Bounds (Theorem \ref{t:FastConvIntro})}\label{s:FastConvProof}
The goal of this section is to fix our notation, precisely recall the notions of mixing rate, mixing time and dissipation time (used in Theorems~\ref{t:FastConvIntro} and~\ref{t:SinMix}), and to state two stronger results that immediately imply Theorem~\ref{t:FastConvIntro}.
The first result (Theorem~\ref{t:fconv}, below) obtains a dissipation time bound in terms of the mixing rate of the imposed velocity field~$v$.
When the mixing rate is exponential, this quickly reduces to the bound~\eqref{e:TDisIntro}.
The second result (Proposition~\ref{p:tdisTmix}, below) bounds the mixing time in terms of the dissipation time, allowing us to deduce~\eqref{e:TMixIntro} from~\eqref{e:TDisIntro}.
The heart of the matter lies in the proofs of Theorem~\ref{t:fconv} and Proposition~\ref{p:tdisTmix}, which we do in Sections~\ref{s: mukapmix} and~\ref{s:tmixtdis} respectively.

\subsection{Mixing rates}\label{s:mixing}

We begin by fixing our notation and precisely defining the notion of \emph{exponential mixing}, which was used in the statements of both Theorems~\ref{t:FastConvIntro} and~\ref{t:SinMix}.
Throughout this paper we will always assume the potential~$U$ is a $C^2$ periodic function, $p$ is defined by~\eqref{e:pdef}, $\rho_\infty$ is defined by~\eqref{e:pdfintro}, and~$\mu$ is the probability measure
\begin{equation}\label{e:mukappadef}
	\mu(dx)
		= \frac{p(x)}{Z} \, dx
		= \rho_\infty(x) \, dx\,.
\end{equation}
Define the $L^2(\mu)$ inner-product, $L^2(\mu)$ norm, and $\dot H^1(\mu)$ norms by
\begin{equation}
	\ip{f, g}_\mu = \int_{\T^d} f g \, d\mu\,,
	\quad
	\norm{f}_{L^2(\mu)}^2 = \ip{f, f}_\mu\,,
	\quad\text{and}\quad
	\norm{f}_{\dot H^1(\mu)} = \norm{\grad f}_{L^2(\mu)}\,.
\end{equation}
respectively, and the corresponding spaces are defined in the usual way.
Define
\begin{equation}\label{e:dotL2}
	\dot L^2(\mu) = \set[\Big]{ f \in L^2(\mu) \st \int_{\T^d} f \, d\mu = 0 }\,,
\end{equation}
to be the subspace of $\mu$-mean-zero functions.

Given a (time dependent) Lipschitz velocity field~$v$, define the flow to be the solution of the ODE
\begin{equation}\label{e:flowdef}
	\partial_t \Phi_{s, t} = v_t \circ \Phi_{s, t}\,,
	\quad \Phi_{s, s}(x) = x\,.
\end{equation}
It is easy to verify that~$\Phi$ preserves the measure~$\mu$ if and only if~$v$ satisfies~\eqref{e:measurePreserving}.
The notion of \emph{mixing} in dynamical systems, requires the flow to spread mass concentrated in a small region to the entire space as~$t \to \infty$ (see for instance~\cite{KatokHasselblatt95,SturmanOttinoEA06}).
More precisely, a flow is mixing if for every pair of test functions~$f, g \in L^2(\mu)$, we have
\begin{equation}\label{e:mixing}
	\lim_{t \to \infty} \ip{f \circ \Phi_{s, s+t}^{-1}, g}_\mu = \ip{f, 1}_\mu \ip{g, 1}_\mu\,.
\end{equation}
Since~$\Phi$ is invertible, one may equivalently replace~$\Phi_{s, s+t}^{-1}$ above with~$\Phi_{s, s+t}$.

The \emph{mixing rate} is the rate at which the convergence in~\eqref{e:mixing} happens for regular test functions~$f, g$.
The standard choice in dynamical systems is to choose~$f, g$ to be H\"older continuous, or~Lipschitz.
However, for our purposes, it is more convenient to use Sobolev regular functions instead.
For convenience, we will further assume the test functions are mean-zero so the right-hand side of~\eqref{e:mixing} vanishes.
We define the flow of~$v$ to be mixing with rate~$h$ if
\begin{equation}\label{e:mixingratedef}
	\sup_{f, g \in \dot H^1(\mu), s \in \R} \ip{ f\circ \Phi_{s, s+t}^{-1}, g }_{\mu} \leq h(t)\norm{f}_{\dot H^1(\mu)}\norm{g}_{\dot H^1(\mu)}\,.
\end{equation}
The function~$h \colon [0, \infty) \to (0, \infty)$ above is called the mixing rate, and it is always assumed to be a continuous decreasing function that vanishes at infinity.
The flow of~$v$ is said to be \emph{exponentially mixing} if the mixing rate~$h$ 
is in the form~\eqref{e:expMixRate} for some finite constants~$D, \gamma$, that may depend on $\kappa$.

Constructing exponentially mixing flows is not an easy task, and has been studied extensively in the dynamical systems literature~\cite{Anosov67,Pollicott85,Dolgopyat98,Liverani04,ButterleyWar20,TsujiiZhang23}.
Unfortunately, these results are all on manifolds other than the standard torus, which is not relevant to the scenario studied in the present paper.
Several authors have recently constructed (time dependent) examples of exponential mixing on the standard torus~\cite{DolgopyatKaloshinEA04,AlbertiCrippaEA19,ElgindiZlatos19,BianchiniZizza21,BedrossianBlumenthalEA22,BlumenthalCotiZelatiEA22}.
However, all these examples preserve the Lebesgue measure and not the measure~$\mu$ as we require.
We will shortly show (Section~\ref{s:modifiedmixing}, below) that the flow defined in~\eqref{e:vDefAltShear} is exponentially mixing (as stated in Theorem~\ref{t:SinMix}).
Postponing further discussion of this to Section~\ref{s:modifiedmixing}, we will now show how such flows can be used to improve the mixing time of~\eqref{e:Aeq}.

\subsection{Dissipation time bounds}\label{s:dtime}

We now recall the notion of~\emph{dissipation time}, and provide an upper bound on the dissipation time in terms of the mixing rate of the flow.
The results are similar to those in~\cite{FengIyer19}.
In our context, the added difficulty is that the measures~$\mu$ concentrate in regions with volume~$O(\kappa^{d/2})$, 
and so we need to track the dependence on both~$\kappa$ and the mixing rate.

Roughly speaking, the \emph{dissipation time} (see~\cite{FannjiangWoowski03,FengIyer19}) measures 
the rate at which the distribution of~$X_t$ approaches stationary distribution~$\mu$ in~$L^2(\mu)$ as~$t \to \infty$, when the initial distribution is also~$L^2(\mu)$.
Precisely, the \emph{dissipation time} of the process~$X$ is defined by
\begin{equation}\label{e:tdisdef}
	\tdis \defeq 
		\inf\set[\Big]{
			t \geq 0 \st
				\sup_{s \geq 0}
					\norm{ \theta_t }_{L^2(\mu)}
					\leq \frac{1}{2} \norm{ f }_{L^2(\mu)} \,,
				~\forall f \in \dot L^2(\mu)
		} \,,
\end{equation}
where
\begin{equation}\label{e:ThetaDef}
	\theta_t(x) = \E^{(x,s)} f(X_t) = \int_{\T^d} \rho(x, s; y, t) f(y) \, dy\,,
\end{equation}
and $\rho(x, s; y, t)$ denotes the transition density of the process~$X$.

The Poincar\'e inequality (Lemma~\ref{l:poincare}, below) quickly implies that
\begin{equation}\label{e:tDisPoincare}
	\tdis \leq \frac{1}{\kappa} \exp\paren[\Big]{ \frac{\norm{U}_\osc}{2 \kappa}}\,,
\end{equation}
which is too large to be of practical interest.
The first result we state is that if~$v$ is mixing, then it can be rescaled to ensure~$\tdis$ is much smaller than the right-hand side of~\eqref{e:tDisPoincare}.

\begin{theorem}\label{t:fconv}
	Let~$\tdis = \tdis(A,\kappa, v)$ be the dissipation time of the process~$X$ defined by~\eqref{e:Aeq}.
	For every sufficiently small $\kappa>0$, there exists $A_0 = A_0(\kappa) < \infty$, independent of the dimension, such that
	\begin{equation}\label{e:tdis}
		\tdis = \tdis(\kappa, A) \leq \frac{16 ( 1 + \ln 2 )}{H(A)}
		\qquad\text{for all } A> A_0(\kappa) \,.
	\end{equation}
	Here $H(A) = H(A, \kappa)$ is defined to be the unique solution of
	\begin{align}\label{e:H}
			\frac{\kappa}{4 H(A)}
				= h\paren[\Big]{ \frac{1}{16} \paren[\Big]{
					\frac{A}{H(A) \norm{\grad v}_{L^\infty}} }^{1/2} } \,.
	\end{align}
\end{theorem}
\begin{remark}\label{r:A0choice}
	During the course of the proof of Theorem~\ref{t:fconv} we will see that~$A_0$ should be chosen so that both
	\begin{equation}\label{e:A0Def}
		A_0 \geq 256 \Lambda \norm{\grad v}_{L^\infty} \paren[\Big]{ h^{-1}\paren[\Big]{ \frac{\kappa}{4 \Lambda} } }^2\,,
		\quad\text{and}\quad
		H(A_0) \geq \frac{\kappa}{4 h\left(\frac{1}{32 \sqrt{2}\norm{\nabla v}_{L^\infty}}\right)}\,.
	\end{equation}
	Here~$\Lambda = \Lambda(\kappa , d)$ is a constant that is chosen to ensure a growth condition on eigenvalues of the generator of~\eqref{e:langevin} (see~\eqref{e:LambdaDef}, below).
	By Weyl's law (specifically from~\eqref{e:Nlambda}) one can check that~$\Lambda = O(\kappa)$ as~$\kappa \to 0$.
	Thus, in the case~$h$ is given by~\eqref{e:expMixRate}, the bound~\eqref{e:A0Def} reduces to~\eqref{e:A0choiceExp} stated in Remark~\ref{r:A0choiceExp}.
\end{remark}

\begin{remark}\label{r:tDisRd}
	Theorem~\ref{t:fconv} still holds when the state space is~$\R^d$, provided the spectrum of the generator of~\eqref{e:langevin} is discrete and the eigenvalues grow according to Weyl's law (as stated in Lemma~\ref{l:spectral}, below).
	It is well known (see for instance Chapter~4 in~\cite{Pavliotis14}) that both these conditions hold provided the potential~$U$ satisfies
	\begin{equation}\label{a:Rd}
		\lim_{|x|\to +\infty} \Big( \frac{|\nabla U(x)|^2}{2}-\Delta U(x)\Big)=+\infty\,.
	\end{equation}
\end{remark}

The main idea behind the proof of Theorem~\ref{t:fconv} is to obtain bounds on the $L^2(\mu)$ decay of solutions to the associated PDE.
We do this by a spectral splitting method that is commonly used in the study of such equations.
Namely, we divide the analysis into two cases:
When~$\theta$ has most of its energy in large frequencies, the standard energy inequality shows that~$\norm{\theta}_{L^2(\mu)}$ decays fast.
On the other hand, when~$\theta$ has most of its energy in small frequencies, 
the mixing caused by the convection term $v \cdot \grad \theta$ generate high frequencies, which in turn forces~$\norm{\theta}_{L^2(\mu)}$ to decay fast.
When the underlying measure is the Lebesgue measure on~$\T^d$ a similar result was proved in~\cite{FengIyer19}, and our proof follows the same structure.

Once Theorem~\ref{t:fconv} is established, proving the upper bound~\eqref{e:TDisIntro} in Theorem~\ref{t:FastConvIntro} is simply a matter of choosing~$h$ to be the exponential~\eqref{e:expMixRate}, and simplifying~\eqref{e:H}.
We do this in Section~\ref{s:FastConvProof2}, below.

Notice that as~$A \to \infty$, we must have $H(A) \to \infty$ and hence so~$\tdis$ can be made arbitrarily small.
This, however, is not always computationally advantageous as solving~\eqref{e:Aeq} when~$A$ is large is very computationally intensive.
Moreover, making~$\tdis$ small is
not yet sufficient to guarantee solutions to~\eqref{e:Aeq} escape the metastable traps at local minima of~$U$.
Indeed, if $X_0$ is initially concentrated in a ball $B(x_0, \sqrt{\kappa} )$, after time~$t \geq \tdis$ the process~$X_t$ may still be concentrated in a ball of radius~$B( x_0, C \sqrt{\kappa})$.
Thus, we are not guaranteed~$X$ has explored the state space enough to escape metastable traps around local minima of~$U$.
We are however close: in the next section we will show that in an additional~$O(\tdis / \kappa)$ time, the process~$X$ will escape metastable traps be close to mixed.

\subsection{Mixing time bounds}

We will now study the relation between~$\tdis$ and~$\tmix$.
Recall~\cite{LevinPeres17} the \emph{mixing time} of a Markov process is the time taken for its distribution to become sufficiently close 
(in the total variation norm) to its stationary distribution.
In our context, the mixing time of~$X$ can be defined by
\begin{equation}\label{e:tmixdef}
	\tmix\defeq\inf \set[\Big]{t \geq 0 \st
		\sup_{x \in \T^d\,,\; s \geq 0} \, \int_{\T^d}  \abs[\big]{ \rho(s, x; s+t, y)-\rho_\infty(y)} \, dy \leq \frac{1}{2} }\,,
\end{equation}
where, as before, $\rho$ is the transition density of $X$.

The mixing time is a stronger notion than the dissipation time.
As mentioned earlier, waiting for time~$\tdis$ will not ensure~$X$ has escaped 
the metastable traps at local minima of~$U$; however waiting for time~$\tmix$ will certainly ensure this.

It is easy to see $\tmix$ always dominates~$\tdis$ (see for instance~\cite{IyerZhou22,IyerLuEA23}).
Our interest, however, is to control the mixing time by the dissipation time.
The advantage of this is that the dissipation time can be bounded using $L^2(\mu)$ based spectral methods, such as those used in the proof of Theorem~\ref{t:fconv}.
Thus controlling~$\tmix$ by~$\tdis$ will allow us to use Theorem~\ref{t:fconv} to obtain upper bounds on the mixing time.
Our next result provides upper (and lower) bounds on the mixing time in terms of the dissipation time.

\begin{proposition}\label{p:tdisTmix}
	There exists a universal (dimension independent) constant~$C$, that is independent of $U$, $v$, $\kappa$ and~$A$, such that
	\begin{equation}\label{e:tmixtdis}
		\frac{\tdis}{3}\leq \tmix \leq C d\paren[\Big]{ 1+\frac{\norm{U}_{\osc}}{\kappa}-\ln(\kappa\tdis)}\tdis\,.
	\end{equation}
\end{proposition}
\begin{remark}\label{r:rhsPositive}
The Poincar\'e Inequality (Lemma~\ref{l:poincare}, below), will guarantee that $\ln (\kappa \tdis) \leq \norm{U}_\osc / (2\kappa)$, and so the factor on the right of~\eqref{e:tmixtdis} is nonnegative.
\end{remark}

The reason for the large factor on the right of~\eqref{e:tmixtdis} is as follows.
Since the noise is regular, the density~$\rho(s, x; s+t, y)$  becomes~$L^2(\mu)$ for any~$t > 0$.
However, the~$L^2(\mu)$ norm is typically of order~$O(1/(\kappa t)^{d/2})$, which is large.
Waiting for a large multiple of~$\tdis$ will now make this small, which is what leads to the large factor on the right of~\eqref{e:tmixtdis}.

To carry out these details, we prove a stronger~$L^1(\mu) \to L^\infty$ bound on the transition density, with constants that are independent of~$v$.
The proof follows the structure of similar bounds on the torus with respect to the Lebesgue measure (see~\cite{ConstantinKiselevEA08}).
The key identity that allows us to make the proof work in our situation is that the ratio $\rho / \rho_\infty$ satisfies an equation that differs from the Kolmogorov backward equation by only a sign (see Lemma~\ref{l:eqverify}, below).

\begin{remark}\label{r:tdisTmixRd}
	The lower bound in~\eqref{e:tmixtdis} works in a general setting and, in particular, works when the state space is~$\R^d$.
	The upper bound need not hold in general as~$\tmix$ may be infinite.
\end{remark}

Of course, Proposition~\ref{p:tdisTmix} and the dissipation time bound~\eqref{e:TDisIntro} immediately imply the mixing time bound~\eqref{e:TMixIntro} stated in Theorem~\ref{t:FastConvIntro}.
We prove Theorem~\ref{t:FastConvIntro} in the next section, and postpone the proofs of Theorem~\ref{t:SinMix} and Proposition~\ref{p:tdisTmix} to Sections~\ref{s: mukapmix} and~\ref{s:tmixtdis} respectively.

\subsection{Proof of Theorem \ref{t:FastConvIntro}}\label{s:FastConvProof2}

The proof of Theorem~\ref{t:FastConvIntro} now follows by simplifying~\eqref{e:H} when~$h$ is given by~\eqref{e:expMixRate}, and using Proposition~\ref{p:tdisTmix}.
We carry out the details here.
\begin{proof}[Proof of Theorem~\ref{t:FastConvIntro}]
	By Theorem~\ref{t:fconv} we know $\tdis$ is bounded by~\eqref{e:tdis} where~$H(A)$ is defined by~\eqref{e:H}.
	In order to bound~$H(A)$, we make the following simple observation:
	If~$T$ satisfies
	\begin{equation}\label{e:T}
		T = a e^{- b \sqrt{T}} \,,
	\end{equation}
	for some constants~$a, b > 0$, then we must have
	\begin{equation}\label{e:TBound}
		T \leq \frac{1 +   \ln^2( a b^2 )}{b^2 }\,.
	\end{equation}
	To see this, note that if $T \leq 1/b^2$ there is nothing to prove.
	If~$T \geq 1/b^2$, then~\eqref{e:T} implies
	\begin{equation}
		T = \frac{1}{b^2} \ln^2 \paren[\Big]{ \frac{a}{T} }
			\leq \frac{\ln^2 (ab^2)}{b^2}\,,
	\end{equation}
	which proves~\eqref{e:TBound}.
	Choosing
	\begin{equation}
		T = \frac{1}{H(A)}\,,
		\quad
		a = \frac{2 D}{\kappa}\,,
		\quad\text{and}\quad
		b = \frac{ \gamma \sqrt{A} }{16 \sqrt{\norm{\grad v}_{L^\infty}} }
	\end{equation}
	and using~\eqref{e:TBound} in~\eqref{e:H} immediately implies~\eqref{e:TDisIntro} as desired.

	The mixing time bound~\eqref{e:TMixIntro} follows immediately from~\eqref{e:TDisIntro} and Proposition~\ref{p:tdisTmix}, concluding the proof.
\end{proof}

\section{Dissipation time bound (Theorem~\ref{t:fconv})}\label{s: mukapmix}

In this section we prove Theorem~\ref{t:fconv}.
The proof is entirely based on PDE techniques.
Indeed, the function~$\theta$ defined by~\eqref{e:ThetaDef} is the solution to the Kolmogorov backward equation
\begin{equation}\label{e:kBackward}
	\partial_t \theta
		= A v' \cdot \grad \theta + \mathcal L_\kappa \theta\,,
	\quad
	\text{with initial data }
	\theta_s = f\,.
\end{equation}
Here~$v'$ is the time changed velocity field
\begin{equation}\label{e:vPrimeDef}
	v'_t = v_{At}\,,
\end{equation}
and~$\mathcal L_\kappa$, defined by
\begin{equation}\label{e:lkappa}
	\mathcal L_\kappa f = - \grad U \cdot \grad f + \kappa \lap f \,,
\end{equation} 
is the generator of~\eqref{e:langevin}.

The main idea behind the proof is to split the analysis into two cases.
When~$\theta$ has most of its energy in large frequencies, the operator~$\mathcal L_\kappa$ will provide a strong damping effect, and~$\norm{\theta}_{L^2(\mu)}$ decays fast.
On the other hand, when~$\theta$ has most of its energy in small frequencies, the mixing caused by the convection term $v' \cdot \grad \theta$ generate high frequencies, which in turn forces~$\norm{\theta}_{L^2(\mu)}$ to decay fast.

To carry out the details we need a few spectral properties of the operator~$\mathcal L_\kappa$, which we collect here for easy reference.

\begin{lemma}\label{l:spectral}
	\begin{enumerate}[(1)]\reqnomode
	  \item The operator $-\mathcal L_\kappa$ is self-adjoint and nonnegative with respect to the inner-product~$\ip{\cdot, \cdot}_\mu$.
		\item For all $f, g \in H^1(\T^d, \mu)$ we have
			\begin{align}\label{e:H1norm}
				-\ip{ \mathcal{L}_\kappa f, g }_\mu=\kappa \ip{ \nabla f, \nabla g}_\mu\,.
			\end{align}
		\item
			The spectrum of~$-\mathcal L_\kappa$ is discrete, and the smallest eigenvalue on~$\dot L^2(\mu)$ is strictly positive.
			Moreover, if~$0 < \lambda_0 \leq \lambda_1 \leq \lambda_2 \dots$ are the eigenvalues of~$-\mathcal L_\kappa$ in ascending order, then
			\begin{equation}\label{e:Weyl}
				\lambda_n \xrightarrow{n \to \infty} \infty
				\qquad\text{and}\qquad
				\frac{\lambda_{n+1}}{\lambda_n} \xrightarrow{n \to \infty} 1\,,
			\end{equation}
	\end{enumerate}
\end{lemma}

Lemma~\ref{l:spectral} directly follows from Weyl's law~\cite{HuangSogge21}, and is presented later in this section.
We now state two lemmas which show fast energy decay both when~$\theta$ has mainly high frequencies, and when it does not.

\begin{lemma}\label{l:h1Large}
	Solutions to~\eqref{e:kBackward} satisfy the energy inequality
	\begin{equation}\label{e:EE}
		\partial_t\norm{\theta}_{L^2(\mu)}^2=-2\ip{\cL_\kappa \theta, \theta}_{\mu}=-2\kappa\norm{\nabla \theta}_{L^2(\mu)}^2\,.
	\end{equation}
  Consequently, if for all $t \in [0, T]$ we have
	\begin{equation*}
		\kappa \norm{\grad \theta_t}_{L^2(\mu)}^2 \geq \lambda \norm{\theta_t}_{L^2(\mu)}^2\,,
	\end{equation*}
	for some constant~$\lambda > 0$.
	Then for all $t \in [0, T]$ we must have
	\begin{equation}\label{e:h1EE}
		\norm{\theta_t}_{L^2(\mu)}^2 \leq
			\exp\paren[\big]{ -2 \lambda t } \norm{\theta_0}_{L^2(\mu)}\,.
	\end{equation}
\end{lemma}
\begin{lemma}\label{l:H1small}
	Let $\lambda_N$ be the largest eigenvalue of~$\mathcal L_\kappa$ such that $\lambda_N\leq  H(A)$, where we recall $H(A)$ is defined in~\eqref{e:H}.
	If 
	\begin{align}\label{e:H1small}
		\kappa \norm{\nabla\theta_s}_{L^2(\mu)}^2\leq \lambda_N \norm{\theta_s}_{L^2(\mu)}^2\,,
	\end{align}
	then
	\begin{align}\label{e:decayaftert0}
		\norm{\theta_{s + t_0}}_{L^2(\mu)}^2\leq \exp\paren[\Big]{ -\frac{ H(A) t_0}{8} }\norm{\theta_s}_{L^2(\mu)}^2\,,
	\end{align}
	where $t_0$ is given by
	\begin{equation}\label{e:t0choice}
		t_0\defeq \frac{2}{A}\,h^{-1}\paren[\Big]{ \frac{\kappa}{4\lambda_N} } \,,
	\end{equation}
	where $h(t)$ is defined in~\eqref{e:expMixRate}.
\end{lemma}

Since the proof of Lemma~\ref{l:h1Large} is short, we present it first.

\begin{proof}[Proof of Lemma~\ref{l:h1Large}]
	Equation~\eqref{e:EE} follows by multiplying~\eqref{e:kBackward} by~$\theta$, integrating by parts, and using~\eqref{e:measurePreserving} and~\eqref{e:H1norm}.
	Equation~\eqref{e:h1EE} follows immediately from~\eqref{e:EE} and Gr\"{o}nwall's lemma.
\end{proof}

The proof of Lemma~\ref{l:H1small} is more involved and relies on the mixing properties of~$v$.
The main idea is that when the spectrum of~$\theta_s$ is concentrated in low frequencies, then it is close to the solution of the transport equation,
\[
	\partial_t \phi = A  v' \cdot \nabla \phi\,, \quad \phi(0)=\theta_0\,.
\]
The mixing assumption on~$v$ guarantees that the transport equation moves energy to high frequencies.
These high frequencies are then dissipated faster by~$\mathcal L_\kappa$, leading to the faster decay stated in Lemma~\ref{l:H1small}.
Postponing the proof of Lemma~\ref{l:H1small} to Section~\ref{s:lowFreq}, we now prove Theorem~\ref{t:fconv}.

\begin{proof}[Proof of Theorem~\ref{t:fconv}]
	We claim that any solution of~\eqref{e:kBackward} with~$\theta_0 \in \dot L^2(\mu)$ satisfies
	\begin{align}\label{e:globaldecay}
		\norm{\theta_{s + t}}_{L^2(\mu)}
			\leq \exp \paren[\Big]{ - \frac{H(A)t}{16} + 1}\norm{\theta_s}_{L^2(\mu)}\,,
	\end{align}
	which immediately implies~\eqref{e:tdis}.

	To prove~\eqref{e:globaldecay}, we may without loss of generality assume~$s = 0$.
	Choose $\lambda_N$ as in Lemma~\ref{l:H1small}, choose $\lambda = \lambda_N$, and repeatedly apply Lemmas~\ref{l:h1Large} and~\ref{l:H1small} to obtain an increasing sequence of times $t_k' \to \infty$ such that $t_{k+1}'-t_{k}'\leq t_0$ and
	\begin{equation}\label{e:2022-10-10-1}
		\norm{\theta_{t_{k+1}'}}_{L^2(\mu)}^2\leq \exp\Bigl(-(t_{k+1}' - t_k')\min\set[\Big]{ \lambda_N, \frac{ H(A)}{8} }\Bigr)\norm{\theta_{t_k'}}_{L^2(\mu)}^2\,,
	\end{equation}
	By Lemma~\ref{l:spectral} there exists~$\Lambda = \Lambda(\kappa)$ such that
	\begin{equation}
		\lambda_{n+1} \leq 8 \lambda_{n}
		\quad \text{whenever } \lambda_{n+1} \geq \Lambda\,.
	\end{equation}
	Let~$A_0$ be defined by~\eqref{e:A0Def}, and note that for~$A \geq A_0$ we have $H(A) \geq \Lambda$.
	This implies
	\begin{equation*}
		\frac{H(A) }{8} \leq \lambda_N \leq H(A)\,,
	\end{equation*}
	and hence~\eqref{e:2022-10-10-1} implies
	\begin{equation}\label{e:2022-10-10-2}
		\norm{\theta_{t_{k+1}'}}_{L^2(\mu)}^2\leq \exp\Bigl(-(t_{k+1}' - t_k') \frac{ H(A)}{8} \Bigr)\norm{\theta_{t_k'}}_{L^2(\mu)}^2\,.
	\end{equation}

	By construction of $t_k'$, for any $t \geq 0$ there exists $k \in \N$ such that $t - t_0 \leq t_k' \leq t$.
	Iterating~\eqref{e:2022-10-10-2} shows
	\begin{equation}\label{e:2022-10-10-3}
		\norm{\theta_t}_{L^2(\mu)}\leq
			\exp\Bigl(-
				\frac{ H(A) t_k'}{16}\Bigr)
			\norm{\theta_0}_{L^2(\mu)}
		\leq \exp\Big(-
		\frac{ H(A) (t-t_0)}{16}\Big)\norm{\theta_0}_{L^2(\mu)}\,.
	\end{equation}
	By choice of $t_0$ and $\lambda_N$, we note
	\begin{align*}
		H(A) t_0 &=\frac{2}{A}h^{-1}\Big(\frac{\kappa}{4\lambda_N}\Big) H(A)\leq \frac{ H(A)}{8\sqrt{A \lambda_N\norm{\nabla v}_{L^\infty}}}
		\\
		&\leq \paren[\Big]{ \frac{ H(A)}{8 A\norm{\nabla v}_{L^\infty}}}^{1/2}
		= \frac{1}{32 \sqrt{2}\norm{\nabla v}_{L^\infty} h^{-1}(\frac{\kappa}{4H(A)}) }
		\leq 1\,.
	\end{align*}
	Using this in~\eqref{e:2022-10-10-3} implies~\eqref{e:globaldecay} as desired.
\end{proof}

The remainder of this section is devoted to proving Lemmas~\ref{l:spectral} and~\ref{l:H1small}.

\subsection{Spectral bounds on \texorpdfstring{$\mathcal L_\kappa$}{Lk} (Lemma~\ref{l:spectral}).}

The operator~$\mathcal L_\kappa$ can be conjugated to a Schr\"odinger operator and well-known results spectral results (e.g.~\cite{HuangSogge21}) for Schr\"odinger operators will imply Lemma~\ref{l:spectral}.
\begin{proof}[Proof of Lemma~\ref{l:spectral}]
	The first two assertions of Lemma~\ref{l:spectral} are direct computations.
	Indeed,
	\begin{equation*}
		-\mathcal{L}_\kappa
			=-\kappa\Delta + \nabla U\cdot \nabla
			=- \kappa e^{U/\kappa}\nabla\cdot(e^{-U/\kappa}\nabla) \,,
	\end{equation*}
	and hence for all $f, g\in H^1(\T^d,\mu)$, we have
	\begin{align*}
		-\langle \mathcal{L}_\kappa f, g\rangle_\mu
			&=-\kappa\int_{\T^d}  e^{U/\kappa}\nabla\cdot(e^{-U/\kappa}\nabla f) g\, \rho_\infty \, dx
			= \frac{-\kappa}{Z} \int_{\T^d}  \nabla\cdot(e^{-U/\kappa}\nabla f) g\, dx
		\\
			&= \frac{\kappa}{Z} \int e^{-U/\kappa}\nabla f\cdot \nabla g\,dx
			=\kappa\langle \nabla f, \nabla g\rangle_{\mu}\,.
	\end{align*}
	This immediately implies the first two assertions.

	To study the spectrum, let $L^2 = L^2(\T^d)$ denote the space of all square-integrable functions with respect to the Lebesgue measure, and~$\ip{\cdot, \cdot}$ the associated inner-product.
	Define the operator~$\mathcal U \colon L^2(\mu) \to L^2$ by
	\begin{equation*}
		\mathcal U f = \frac{1}{\sqrt{Z}} e^{-U / 2\kappa} f\,.
	\end{equation*}
	Clearly $\ip{f, g}_\mu = \ip{f, g}$, and so $\mathcal U$ is an isometry.
	Define the operator~$\mathcal H \colon L^2 \to L^2$ by $\mathcal H \defeq \mathcal U \mathcal L_\kappa \mathcal U^{-1}$.
	We compute
	\begin{equation*}
		-\cH f =-\kappa\Delta f +\paren[\Big]{ \frac{1}{4} |\nabla U|^2-\frac{1}{2}\Delta U} f \,.
	\end{equation*}
	Thus~$\mathcal L_\kappa$ is unitarily equivalent to the operator~$\mathcal H$, and hence the operators~$\mathcal L$ and~$\mathcal U$ have the same spectrum.

	The operator~$\mathcal H$ is a Schr\"odinger operator and has been extensively studied.
	In particular, the eigenvalues of $\cH$ satisfy Weyl's law~\cite{HuangSogge21} (see also \cite[Theorem VI]{Ray54}), 
	which states that
	\begin{align}
		\label{e:Nlambda}
		N(\lambda)  &\defeq \sum_{\lambda_n < \lambda} 1
			= \frac{\omega_d}{(2\pi)^d} \paren[\Big]{ \frac{\lambda}{\kappa} }^{\frac{d}{2}}
				+ O\paren[\Big]{ \frac{\lambda}{\kappa}  }^{\frac{d-1}{2}}\,,
	\end{align}
	asymptotically, as $\lambda \to \infty$.
	Here $\omega_d$ is the volume of the unit ball in $\R^d$.
	This immediately implies the third assertion in Lemma~\ref{l:spectral}, finishing the proof.
\end{proof}

\subsection{Low frequency energy decay (Lemma~\ref{l:H1small})}\label{s:lowFreq}

To prove Lemma~\ref{l:H1small}, we will first show (Lemma~\ref{l:diff}, below) that when~\eqref{e:H1small} holds, $\theta$ is sufficiently close to solutions to the transport equation~\eqref{e:drift}.
By the mixing assumption on~$v$ this will move energy to high frequencies, which will then be dissipated faster by the diffusion operator~$\mathcal L_\kappa$.

\begin{lemma}\label{l:diff}
	Let $\theta$ be the solution of~\eqref{e:kBackward} with initial data~$\theta_0 \in \dot L^2(\mu)$, and let $\phi$ be a solution of the transport equation
\begin{equation}\label{e:drift}
	\partial_t \phi = A  v' \cdot \nabla \phi\,, \quad \phi(0)=\theta_0\,,
\end{equation}
	with the same initial data.
	For any $t\ge 0$ we have
	\begin{align}
		\|\theta-\phi\|_{L^2(\mu)}^2 \leq   \sqrt{2\kappa t}\norm{\theta_0}_{L^2(\mu)}\Big(2A\norm{\nabla v}_{L^\infty}\int_0^t\norm{\nabla\theta}_{L^2(\mu)}^2\,ds+\norm{\nabla\theta_0}_{L^2(\mu)}^2\Big)^{1/2}\,.
	\end{align}
\end{lemma}
\begin{proof}
	Multiplying~\eqref{e:kBackward} by $-\mathcal L_\kappa \theta$ and integrating over space gives
	\begin{align}\label{e:l1energy}
		\nonumber
		\frac{\kappa}{2}\partial_t \langle \nabla \theta, \nabla \theta\rangle_\mu
		&= - \frac{1}{2}\partial_t \langle\theta,\cL_\kappa \theta\rangle_\mu
		= - \langle\partial_t \theta, \cL_\kappa\theta\rangle_{\mu}
	\\
		&= - \|\cL_\kappa\theta\|_{L^2(\mu)}^2
			- A\langle v'\cdot \nabla \theta, \cL_\kappa\theta\rangle_{\mu}\,.
	\end{align}
	For the last term, we note
	\begin{align*}
			\langle v'\cdot \nabla \theta, \mathcal L_\kappa \theta\rangle_{\mu}
			&= \frac{\kappa}{Z}\int_{\T^d} v'\cdot\nabla \theta \nabla \cdot (e^{-U/\kappa}\nabla \theta) \,dx
	\\
			&=-\frac{\kappa}{Z}\int_{\T^d} \nabla (v'\cdot\nabla \theta)\cdot \nabla \theta e^{-U/\kappa}\,dx
	\\
			&=\frac{\kappa}{2Z}\int_{\T^d} v' e^{-U/\kappa}\cdot \nabla (|\nabla \theta|^2)\,dx
				+\kappa\int_{\T^d} (\nabla\theta \cdot \nabla) v' \cdot \nabla \theta\,d\mu\\
			&= \kappa\int_{\T^d} (\nabla\theta \cdot \nabla) v' \cdot \nabla \theta\,d\mu\,,
	\end{align*}
	since $v'$ satisfies~\eqref{e:measurePreserving}.
	Consequently,
	\begin{equation*}
		\abs[\big]{ \langle v'\cdot \nabla \theta, \mathcal L_\kappa \theta\rangle_{\mu} }
			\leq \kappa \int_{\T^d} \abs{\grad v'} \abs{\grad \theta}^2 \,d\mu
			\leq \kappa \norm{\grad v}_{L^\infty} \norm{\grad \theta}_{L^2(\mu)}^2 \,,
	\end{equation*}
	and substituting this in~\eqref{e:l1energy}, we have
	\begin{align}\label{e:l1energy2}
		\|\cL_\kappa\theta\|_{L^2(\mu)}^2\leq A\kappa\|\nabla v\|_{L^\infty}\|\nabla \theta\|_{L^2(\mu)}^2 -\frac{\kappa}{2}\partial_t \langle \nabla \theta, \nabla \theta\rangle_\mu \,.
	\end{align}

	On the other hand, $\norm{\theta - \phi}_{L^2(\mu)}$ satisfies
	\begin{align}\label{e:diff}
		\nonumber
		\partial_t \|\theta-\phi\|_{L^2(\mu)}^2
			&=2\langle \cL_\kappa\theta, \theta-\phi\rangle_\mu
			\leq -2\langle \cL_\kappa\theta, \phi\rangle_\mu
			\leq 2\|\cL_\kappa\theta\|_{L^2(\mu)}\|\phi\|_{L^2(\mu)}
		\\
		&\leq 2\|\cL_\kappa\theta\|_{L^2(\mu)}\|\theta_0\|_{L^2(\mu)}\,.
	\end{align}
	To obtain the last equality above we used the fact that $\norm{\phi_t}_{L^2(\mu)} = \norm{\phi_0}_{L^2(\mu)}$, which is true because~$v'$ satisfies~\eqref{e:measurePreserving}.
	Integrating~\eqref{e:diff} and applying~\eqref{e:l1energy2}, we have
	\begin{align*}
		\MoveEqLeft
		\|\theta-\phi\|_{L^2(\mu)}^2 \leq 2\norm{\theta_0}_{L^2(\mu)}\int_0^t\norm{\cL_\kappa\theta}_{L^2(\mu)}\,ds\\
		&\leq 2\sqrt{t}\norm{\theta_0}_{L^2(\mu)}\Big(\int_0^t A\kappa\|\nabla v\|_{L^\infty}\|\nabla \theta\|_{L^2(\mu)}^2 -\frac{\kappa}{2}\partial_t \langle \nabla \theta, \nabla \theta\rangle_\mu \,ds\Big)^{1/2}\\
		&\leq  \sqrt{2\kappa t}\norm{\theta_0}_{L^2(\mu)}\Big(2A\norm{\nabla v}_{L^\infty}\int_0^t\norm{\nabla\theta}_{L^2(\mu)}^2\,ds+\norm{\nabla\theta_0}_{L^2(\mu)}^2\Big)^{1/2}\,,
	\end{align*}
	concluding the proof.
\end{proof}

We now use Lemma~\ref{l:diff} to prove Lemma~\ref{l:H1small}.

\begin{proof}[Proof of Lemma~\ref{l:H1small}]
	For simplicity, and without loss of generality, we assume $s = 0$.
	We claim our choice of $\lambda_N$ and $t_0$ will guarantee 
	\begin{align}\label{e:intH1large}
		\kappa \int_0^{t_0}\norm{\nabla \theta}_{L^2(\mu)}^2\,ds \geq \frac{\lambda_N t_0}{8}\norm{\theta_0}_{L^2(\mu)}^2\,.
	\end{align}
	To prove this, assume, for sake of contradiction, that
	\begin{align}\label{e:converse}
		\kappa \int_0^{t_0}\norm{\nabla \theta}_{L^2(\mu)}^2\,ds < \frac{\lambda_N t_0}{8}\norm{\theta_0}_{L^2(\mu)}^2\,.
	\end{align}
	Let $P_N$ be the orthogonal projection from $\dot L^2(\mu)$ to the space spanned by the first~$N$ eigenfunctions $\{e_1,\dots, e_N\}$, recall~\eqref{e:H1norm} and notice
	\begin{align}\label{e:H1largetmp1}
		\MoveEqLeft
		\nonumber
		\int_0^{t_0}\norm{\nabla\theta}_{L^2(\mu)}^2\,ds\geq \frac{\lambda_N}{\kappa} \int_{\frac{t_0}{2}}^{t_0}\norm{(I-P_N)\theta}_{L^2(\mu)}^2\,ds\\
		\nonumber
		&\geq \frac{\lambda_N}{2\kappa}\int_{\frac{t_0}{2}}^{t_0}\norm{(I-P_N)\phi}_{L^2(\mu)}^2\,ds
		-\frac{\lambda_N}{\kappa}\int_{\frac{t_0}{2}}^{t_0}\norm{(I-P_N)(\theta-\phi)}_{L^2(\mu)}^2\,ds\\
		&\geq \frac{\lambda_N t_0}{4\kappa}\norm{\theta_0}_{L^2(\mu)}^2-\frac{\lambda_N}{2\kappa}\int_{\frac{t_0}{2}}^{t_0}\norm{P_N\phi}_{L^2(\mu)}^2\,ds
		-\frac{\lambda_N}{\kappa}\int_{\frac{t_0}{2}}^{t_0}\norm{\theta-\phi}_{L^2(\mu)}^2\,ds\,.
	\end{align}


	We will now bound each of the negative terms on the right of~\eqref{e:H1largetmp1}.
	For the second term in~\eqref{e:H1largetmp1}, we first note that the mixing rate of the rescaled velocity field $A v' = A v_{At}$ is $h(At)$.
	Thus, by~\eqref{e:mixingratedef},
	\begin{align}
		\norm{P_N \phi_s}_{L^2(\mu)}^2
			&= \ip{\phi_s, P_N \phi_s}_{L^2(\mu)}
			= \ip{\theta_0 \circ \Phi_{0, s}^{-1}, P_N \phi_s}_{L^2(\mu)}
		\\
			&\leq h( As ) \norm{\theta_0}_{\dot H^1(\mu)} \norm{ P_N \phi_s }_{\dot H^1(\mu)}
			\leq \sqrt{\lambda_N/\kappa} h(A s) \norm{\theta_0}_{\dot H^1(\mu)} \norm{\phi_s}_{L^2(\mu)}
		\\
			&= \sqrt{\lambda_N/\kappa} h(A s) \norm{\theta_0}_{L^2} \norm{\theta_0}_{\dot H^1(\mu)}\,.
	\end{align}
	Combining this with~\eqref{e:H1small} implies
	\begin{align}\label{e:H1largetmp2}
		\int_{\frac{t_0}{2}}^{t_0}\norm{P_N\phi}_{L^2(\mu)}^2\,ds
			&\leq \sqrt{\lambda_N/\kappa} \int_{\frac{t_0}{2}}^{t_0} h(As) \norm{\theta_0}_{\dot H^1(\mu)} \norm{\theta_0}_{L^2(\mu)} \,ds
		\\
			&\leq  \frac{\lambda_N t_0}{2 \kappa }h\paren[\Big]{ \frac{At_0}{2} } \norm{\theta_0}_{L^2(\mu)}^2\,.
	\end{align}

For the last term in~\eqref{e:H1largetmp1}, applying Lemma~\ref{l:diff} and inequality~\eqref{e:converse}, we note
	\begin{align*}
		\MoveEqLeft
		\int_{\frac{t_0}{2}}^{t_0}\norm{\theta-\phi}_{L^2(\mu)}^2\,ds\\
		&\leq \int_{\frac{t_0}{2}}^{t_0}\sqrt{2\kappa\tau}\norm{\theta_0}_{L^2(\mu)}
			\Big(2A\norm{\nabla v}_{L^\infty}\int_0^\tau\norm{\nabla \theta}_{L^2(\mu)}^2\,ds+\norm{\nabla \theta_0}_{L^2(\mu)}^2\Big)^{1/2}\,d\tau\\
		&\leq \frac{\sqrt{2\kappa}t_0^{3/2}\norm{\theta_0}_{L^2(\mu)}}{2}\Big(\frac{A\lambda_N t_0\norm{\nabla v}_{L^\infty}}{4\kappa} \norm{\theta_0}_{L^2(\mu)}^2+\norm{\nabla \theta_0}_{L^2(\mu)}^2\Big)^{1/2}\,.
	\end{align*}
	Using~\eqref{e:H1small} this gives
	\begin{align}\label{e:H1largetmp3}
		\nonumber
		\int_{\frac{t_0}{2}}^{t_0}\norm{\theta-\phi}_{L^2(\mu)}^2\,ds&\leq  \frac{\sqrt{2\kappa}t_0^{3/2}\norm{\theta_0}_{L^2(\mu)}^2}{2}\Big(\frac{A\lambda_N t_0\norm{\nabla v}_{L^\infty}}{4\kappa} +\frac{\lambda_N}{\kappa}\Big)^{1/2}\\
		&\leq \frac{\sqrt{2}t_0^{3/2}\norm{\theta_0}_{L^2(\mu)}^2}{2}\Big(\frac{A\lambda_N t_0\norm{\nabla v}_{L^\infty}}{2} \Big)^{1/2}\,.
	\end{align}
	To obtain the last inequality above, we used the fact $At_0 \norm{\grad v}_{L^\infty} \gg 1$, which is guaranteed by the choice of $t_0$ in~\eqref{e:t0choice}.

	Using~\eqref{e:H1largetmp2} and~\eqref{e:H1largetmp3} in~\eqref{e:H1largetmp1} gives
	\begin{align*}
		\kappa \int_{0}^{t_0}\norm{\nabla\theta}_{L^2(\mu)}^2\,ds
		&= \frac{\lambda_N t_0}{4} \norm{\theta_0}_{L^2(\mu)}^2
			\paren[\Big]{1
				- \frac{\lambda_N}{\kappa} h\paren[\Big]{ \frac{A t_0}{2} }
				- 2 t_0 \sqrt{A \lambda_N \norm{\grad v}_{L^\infty}} }
		\\
		&= \frac{\lambda_N t_0}{4} \norm{\theta_0}_{L^2(\mu)}^2
			\paren[\Big]{1
				- \frac{1}{4}
				- 4 h^{-1} \paren[\Big]{ \frac{\kappa}{4\lambda_N} } \sqrt{\frac{\lambda_N \norm{\grad v}_{L^\infty}}{A}} }\,,
	\end{align*}
	where we used~\eqref{e:t0choice} to obtain the last equality.
	Since the function
	\begin{equation*}
		\lambda \mapsto \sqrt{\lambda} h^{-1}\paren[\Big]{ \frac{\kappa}{4\lambda} }
	\end{equation*}
	is increasing, the definition of~$H(A)$ (equation~\eqref{e:H}) and the fact that $\lambda_N \leq H(A)$ imply
	\begin{equation*}
		\kappa \int_{0}^{t_0}\norm{\nabla\theta}_{L^2(\mu)}^2\,ds
		= \frac{\lambda_N t_0}{4} \norm{\theta_0}_{L^2(\mu)}^2
			\paren[\Big]{1
				- \frac{1}{4}
				- \frac{1}{4} }
		\geq \frac{\lambda_N t_0}{8} \norm{\theta_0}_{L^2(\mu)}^2\,.
	\end{equation*}
	This contradicts~\eqref{e:converse}, and hence concludes the proof of~\eqref{e:intH1large}.
	\smallskip

	Finally, to obtain~\eqref{e:decayaftert0} we use Lemma~\ref{l:spectral} to find~$\Lambda = \Lambda(\kappa)$ such that
	\begin{equation}\label{e:LambdaDef}
		\lambda_{n+1} \leq 2 \lambda_{n}
		\quad \text{whenever } \lambda_{n+1} \geq \Lambda\,.
	\end{equation}
	If~$A_0$ is chosen according to~\eqref{e:A0Def}, then we will have $H(A) \geq \Lambda$, and hence
	\begin{equation*}
	  \frac{ H(A)}{2}\leq  \lambda_N  \leq H(A) \,.
	\end{equation*}
	Combining this with~\eqref{e:intH1large} we obtain
	\begin{equation*}
		\kappa \int_0^{t_0}\norm{\nabla \theta}_{L^2(\mu)}^2\,ds \geq \frac{H(A) t_0}{16}\norm{\theta_0}_{L^2(\mu)}^2\,.
	\end{equation*}
	Now the energy equality~\eqref{e:EE} implies~\eqref{e:decayaftert0} as desired.
\end{proof}

\section{Bounding \texorpdfstring{$\tmix$}{tmix} in terms of \texorpdfstring{$\tdis$}{tdis} (Proposition \ref{p:tdisTmix})}\label{s:tmixtdis}
\subsection{The lower bound}

It is natural to expect that the mixing time controls the dissipation time in a general setting, and a similar result appeared recently in~\cite{IyerZhou22}.
Roughly speaking, to bound the mixing time, we need to start~$X$ with \emph{any} 
initial distribution and show that the distribution of $X$ becomes close to the invariant distribution in the total variation norm.
To bound the dissipation time, 
we only need to consider $L^2$ initial distribution and bound the distance to the invariant distribution in a weaker sense.
As a result, the lower bound in Proposition~\ref{p:tdisTmix} is true in a more general setting, and the proof we present doesn't rely on the specific structure of~\eqref{e:Aeq}.

\begin{proof}[Proof of the lower bound in Proposition~\ref{p:tdisTmix}]
	Let $f \in \dot L^2(\mu)$, and define
	\begin{equation*}
		\theta_t(x) = \E^{(x, s)} f(X_{s + t}) = \int_{\T^d} \rho(x, s; y, s+t) f(y) \, dy \,.
	\end{equation*}
	In order to prove the lower bound in~\eqref{e:tmixtdis}, we need to show that for any $t \geq 3\tmix$, we have
	\begin{equation}\label{e:2022-09-29-2}
		\norm{\theta_t}_{L^2(\mu)} \leq \frac{1}{2} \norm{f}_{L^2(\mu)}\,.
	\end{equation}
	To prove this, we may without loss of generality assume $s = 0$.
	For notational convenience, we write
	\begin{equation}\label{e:rhoAb}
		\rho_t(x, y) = \rho(x, 0; y, t)\,.
	\end{equation}

	Since $f \in \dot L^2(\mu)$ we note
	\begin{equation*}
		\theta_t(x)
			= \int_{\T^d}  \rho_t(x,y) f(y) \,dy
			= \int_{\T^d} \paren[\big]{ \rho_t(x,y)-\rho_\infty(y) } f(y) \,dy
	\end{equation*}
	and hence
	\begin{align*}
		\MoveEqLeft
		\norm{\theta_t}_{L^2(\mu)}^2
			=\int_{\T^d}\left(\int_{\T^d} \left(\rho_t(x,y)-\rho_\infty(y)\right)f(y)\,dy\right)^2\rho_\infty(x)\,dx
		\\
			&\leq\int_{\T^d}\left(\int_{\T^d}\abs{\rho_t(x,y)-\rho_\infty(y)}f(y)^2\,dy\int_{\T^d}\abs{\rho_t(x,y)-\rho_\infty(y)}\,dy\right)\rho_\infty(x)\,dx\,.
	\end{align*}
	Since $\tmix$ is the mixing time, for any $t \geq n\, \tmix$ we must have
	\begin{equation*}
		\sup_{x \in \T^d} \int_{\T^d} \abs{\rho_t(x, y) - \rho_\infty(y)} \, dy \leq \frac{1}{2^n} \,.
	\end{equation*}
	This implies
	\begin{align}
		\nonumber
		\norm{\theta_t}_{L^2(\mu)}^2&\leq \frac{1}{2^n}\int_{\T^d}\int_{\T^d}\abs{\rho_t(x,y)-\rho_\infty(y)}f(y)^2 \rho_\infty(x)\,dy \, dx\\
		\label{e:2022-09-29-1}
	&\leq \frac{1}{2^n}\int_{\T^d}\int_{\T^d}\left(\rho_t(x,y)+\rho_\infty(y)\right)f(y)^2\rho_\infty(x)\,dx \, dy\,.
	\end{align}
	Since $\rho_\infty$ is the density of the invariant measure, we know
	\begin{equation*}
		\int_{\T^d}\rho_t(x,y)\rho_\infty(x)\,dx=\rho_\infty(y)
		\qquad\text{and}\qquad
		\int_{\T^d} \rho_\infty(x) \, dx = 1\,.
	\end{equation*}
	Using this in~\eqref{e:2022-09-29-1} implies
	\begin{equation*}
		\norm{\theta_t}_{L^2(\mu)}^2
			\leq 2^{1-n}\norm{f}_{L^2(\mu)}^2\,,
	\end{equation*}
	and choosing $n = 3$ implies~\eqref{e:2022-09-29-2} as desired.
\end{proof}

\subsection{The upper bound}

To control the mixing time by the dissipation time we need to use the regularizing effects of the noise.
More precisely, given any initial distribution, the noise regularizes it and the density becomes square-integrable, but with a large $L^2$ norm.
Now waiting some multiple of the dissipation time will ensure mixing.

We implement the above idea by starting with an $L^1 \to L^\infty$ bound on the transition density~$\rho$.
This is the analog of the well-known drift independent $L^1 \to L^\infty$ estimates in~\cite{ConstantinKiselevEA08} in the case where the underlying measure is~$\mu$ instead of the Lebesgue measure.

\begin{lemma}\label{l:l1linfrho}
	When $d \geq 3$,
	for every $x \in \T^d$, $s > s' \geq 0$, and $t > 0$ we have
	\begin{equation}\label{e:l1linfrho}
		\norm[\Big]{ \frac{\rho^{(x, s')}_{s+t}}{\rho_\infty} - 1}_{L^\infty}
		 	\leq \frac{C_1(d)}{(\kappa t)^{d/2} }
		 		\exp\paren[\Big]{ \frac{2d\norm{U}_\osc }{\kappa} }
				\norm[\Big]{ \frac{\rho^{(x, s')}_s}{\rho_\infty} - 1 }_{L^1(\mu)}\,,
	\end{equation}
	where $\rho^{(x,s)}_t$ denotes the transition density $\rho(x, s; \cdot, t)$, and~$C_1$ is a dimensional constant that can be bounded by
	\begin{equation}\label{e:C1bound}
	  C_1(d) \leq C 2^d\,,
	\end{equation}
where $C$ is a universal constant independent of $d$.
	When $d = 2$, the inequality~\eqref{e:l1linf} needs to be replaced by
	\begin{equation}\label{e:lilinf2drho}
		\norm[\Big]{ \frac{\rho^{(x, s')}_{s+t}}{\rho_\infty} - 1}_{L^\infty}
		 	\leq \frac{C_1'(\epsilon)}{(\kappa t)^{1 + \epsilon} }
		 		\exp\paren[\Big]{ \frac{(4 + 4\epsilon) \norm{U}_\osc }{\kappa} }
				\norm[\Big]{ \frac{\rho^{(x, s')}_s}{\rho_\infty} - 1 }_{L^1(\mu)}\,,
	\end{equation}
	where $\epsilon > 0$, and $C_1'(\epsilon)$ is an~$\epsilon$-dependent constant.
\end{lemma}

Momentarily postponing the proof of Lemma~\ref{l:l1linfrho}, we now prove the upper bound in Proposition~\ref{p:tdisTmix} and control the mixing time in terms of the dissipation time.
\begin{proof}[Proof of the upper bound in Proposition~\ref{p:tdisTmix}]	
	For simplicity, and without loss of generality, we will again assume~$s = 0$, and abbreviate the transition density as in~\eqref{e:rhoAb}.
	We will also assume $d \geq 3$.
	The proof when $d = 2$ is similar, and follows by choosing~$\epsilon > 0$ and replacing our use of~\eqref{e:l1linfrho} with~\eqref{e:lilinf2drho}.

	When~$d \geq 3$, inequality~\eqref{e:l1linfrho} implies
	\begin{align*}
		\MoveEqLeft
		\norm{\rho_t(x,\cdot)-\rho_\infty}_{L^1}
			= \int_{\T^d} \abs{\rho_t(x,y)-\rho_\infty(y)} \, dy
			\leq \norm[\Big]{\frac{\rho_t(x,\cdot)}{\rho_\infty}-1}_{L^\infty}
		\\
			&\leq  
				C_1\paren[\Big]{ \frac{4}{\kappa t} }^{ \frac{d}{2} }
				\exp\paren[\Big]{ \frac{2d\norm{U}_\osc }{\kappa} }
				\norm[\Big]{\frac{\rho_{3t/4}(x,\cdot)}{\rho_\infty}-1}_{L^1(\mu)}
		\\
			&\leq  
				C_1 \paren[\Big]{ \frac{4}{\kappa t} }^{ \frac{d}{2} }
				\exp\paren[\Big]{ \frac{2d\norm{U}_\osc }{\kappa} }
				\norm[\Big]{\frac{\rho_{3t/4}(x,\cdot)}{\rho_\infty}-1}_{L^2(\mu)}\,.
	\end{align*}
	Here $C_1 = C_1(d)$ is the constant form Lemma~\ref{l:l1linfrho}.

	The above implies that for any $t \geq 4 n \tdis$, we have 
	\begin{align*}
		\MoveEqLeft
		\norm{\rho_t(x,\cdot)-\rho_\infty}_{L^1}
		\\
			&\leq  
				C_1\paren[\Big]{ \frac{4}{\kappa \tdis} }^{ \frac{d}{2}}
				\exp\paren[\Big]{ \frac{2d\norm{U}_\osc }{\kappa} - n }
				\norm[\Big]{\frac{\rho_{t/2}(x,\cdot)}{\rho_\infty}-1}_{L^2(\mu)}
		\\
			&\leq  
				C_1^2 \paren[\Big]{ \frac{4}{\kappa \tdis} }^{ d }
				\exp\paren[\Big]{ \frac{(4d\norm{U}_\osc }{\kappa}  - n}
				\norm[\Big]{\frac{\rho_{t/4}(x,\cdot)}{\rho_\infty}-1}_{L^1(\mu)}
		\\
			&\leq  
				2 C_1^2 \paren[\Big]{ \frac{4}{\kappa \tdis} }^{ d }
				\exp\paren[\Big]{ \frac{4d\norm{U}_\osc }{\kappa}  - n} \,.
	\end{align*}
Choosing 
\begin{align}\label{e:n}
n=\ceil[\Big]{\frac{4d\norm{U}_{\osc}}{\kappa}+(2d+2)\ln 2+2\ln C_1-d\ln(\kappa \tdis)}
\end{align}
yields
\begin{align}
	\norm{\rho_t(x,\cdot)-\rho_\infty}_{L^1}\leq \frac{1}{2}\,,\quad \text{at $t=4n\tdis$}\,.
\end{align}
Using~\eqref{e:C1bound} we note that~$n$ in~\eqref{e:n} can be bounded by
\begin{align}
	n\leq C d\paren[\Big]{ 1+\frac{\norm{U}_{\osc}}{\kappa}-\ln(\kappa\tdis)}\,,
\end{align}
for some dimension independent constant~$C$, and this implies the upper bound in~\eqref{e:tmixtdis} as desired.
\end{proof}

It remains to prove Lemma~\ref{l:l1linfrho}, which we do in the next sub-section.

\subsection{The \texorpdfstring{$L^1 \to L^\infty $}{L1 to Linfty} bound on the transition density (Lemma~\ref{l:l1linfrho}).}

To prove Lemma~\ref{l:l1linfrho}, we first compute an evolution equation for the ratio $\rho^{(x, s)}_t / \rho_\infty$.
Recall that in the variables~$y, t$, the transition density~$\rho$ is a solution to the forward equation
\begin{equation}\label{e:rhoForward}
	\partial_t \rho = - \dv_y (A v' \, \rho ) + \mathcal L^*_{\kappa, y} \rho\,,
\end{equation}
where~$v'$ is the time rescaled velocity field~\eqref{e:vPrimeDef}, $\mathcal L_\kappa^*$ is defined by
\begin{equation}\label{e:lKappaStar}
	\mathcal L_\kappa^* f  = \dv( f \grad U) + \kappa \lap f\,.
\end{equation}
We clarify that and the notation~$\mathcal L_{\kappa , y}$ refers to the fact that~$\mathcal L_\kappa$ uses derivatives with respect to variable~$y$ in~\eqref{e:rhoForward}.
While the operator~$\mathcal L_\kappa$ is self adjoint with respect to the~$L^2(\mu)$ inner-product, it is not self-adjoint with respect to the standard~$L^2$ inner-product.
Indeed, the adjoint of~$\mathcal L_\kappa$ with respect to the standard~$L^2$ inner-product is precisely~$\mathcal L_\kappa^*$.

\begin{lemma}\label{l:eqverify}
	Let~$\varphi$ be a solution of the forward equation
	\begin{equation}\label{e:kForwardSignV}
		\partial_t \varphi = -\dv (A v' \varphi) + \mathcal L_\kappa^* \varphi\,,
	\end{equation}
	where we recall~$\mathcal L_\kappa^*$ is defined in equation~\eqref{e:lKappaStar}.
	The function~$\theta$, defined by
	\begin{equation*}
		\theta_t(x) \defeq \frac{\varphi_t(x)}{\rho_\infty(x) }\,,
	\end{equation*}
	is a solution of the equation
	\begin{align}\label{e:cpde}
		\partial_t \theta+A v'\cdot \nabla \theta -\mathcal L_\kappa \theta =0\,.
	\end{align}
\end{lemma}
\begin{remark*}
	The equation~\eqref{e:cpde} differs from the backward equation~\eqref{e:kBackward} only in the sign of the convection term~$A v' \cdot \grad \theta$.
\end{remark*}
\begin{proof}
	The proof is a direct calculation.
	Substituting~$\varphi = \rho_\infty \theta$ in~\eqref{e:kForwardSignV} yields
	\begin{align*}
		\rho_\infty \partial_t \theta
			&= \theta \paren[\Big]{\mathcal L_\kappa^* \rho_\infty - \dv( A v' \rho_\infty ) }
		\\
			&\qquad
				\mathbin{+} \rho_\infty \paren[\Big]{ (\grad U + A v') \cdot \grad \theta + \kappa \lap \theta }
				+ 2 \kappa \grad \rho_\infty \cdot \grad \theta\,.
	\end{align*}
	Using~\eqref{e:measurePreserving}, and the fact that
	\begin{equation*}
		\mathcal L_\kappa^* \rho_\infty = 0\,,
		\qquad
		\kappa \grad \rho_\infty = - \rho_\infty \grad U\,,
	\end{equation*}
	we obtain
	\begin{equation*}
		\rho_\infty \partial_t \theta
			= \rho_\infty \paren[\big]{ -A v' \cdot \grad \theta + \mathcal L_\kappa \theta }\,.
	\end{equation*}
	Since $\rho_\infty > 0$, this implies~\eqref{e:cpde} concluding the proof.
\end{proof}

The next lemma we need is an $L^1\to L^\infty$ bound on the semigroup operator of~\eqref{e:kBackward}.
This is the analog of the results in~\cite{ConstantinKiselevEA08,Zlatos10,IyerXuEA21,FannjiangKiselevEA06} when the underlying measure is~$\mu$ and not the Lebesgue measure.
\begin{lemma}\label{l:L1Linf}
	When $d \geq 3$, every solution to~\eqref{e:cpde} with $\mu$-mean zero initial data satisfies
	\begin{align}\label{e:l1linf}
		\norm{\theta_{s + t}}_{L^\infty}\leq
			\frac{C_1(d)}{(\kappa t)^{d/2}}
			\exp\paren[\Big]{ \frac{2d\norm{U}_\osc }{\kappa} } \norm{\theta_s}_{L^1(\mu)}\,.
	\end{align}
	where $C_1(d)$ is as defined in~\eqref{e:C1bound}.


	When~$d = 2$, the inequality~\eqref{e:l1linf} needs to be replaced by
	\begin{align}\label{e:l1linf2d}
		\norm{\theta_{s + t}}_{L^\infty}\leq
			\frac{C_1'(\epsilon)}{(\kappa t)^{d/2+\epsilon}}
			\exp\paren[\Big]{ \frac{(2d+4\epsilon)\norm{U}_\osc }{\kappa} } \norm{\theta_s}_{L^1(\mu)}\,,
	\end{align}
	where $\epsilon > 0$ and $C_1'(\epsilon)$ is an~$\epsilon$-dependent constant.
\end{lemma}

Of course Lemma~\ref{l:eqverify} and~\ref{l:L1Linf} immediately imply Lemma~\ref{l:l1linfrho}.
\begin{proof}[Proof of Lemma~\ref{l:l1linfrho}]
	For any fixed $x \in \T^d$, $s' \geq 0$, we know that the transition density $\rho(x, s'; y, t)$ satisfies the forward equation~\eqref{e:kForwardSignV} in the variables~$y$, $t$.
	Thus, by Lemma~\ref{l:eqverify}, the function~$\theta$ defined by
	\begin{equation*}
		\theta_t(y) \defeq \frac{\rho(x, s'; y, t)}{\rho_\infty(y)} - 1\,,
	\end{equation*}
	satisfies equation~\eqref{e:cpde}.
	Clearly~$\theta$ has $\mu$-mean zero.
	Also, for any for any $s > s'$, $\theta_s \in L^1(\mu)$, and so Lemma~\ref{l:L1Linf} applies.
	The bounds~\eqref{e:l1linfrho} and~\eqref{e:lilinf2drho} follow immediately from~\eqref{e:l1linf} and~\eqref{e:l1linf2d} respectively.
\end{proof}


It remains to prove Lemma~\ref{l:L1Linf}.
For this we will need a Nash inequality with respect to the measure~$\mu$.
\begin{lemma}[Nash Inequality]\label{l:actuallyTheNashInequality}
  For $d \geq 3$ and any $\mu$-mean zero function~$f$ we have
	\begin{equation}\label{e:nashRd1}
		C_2 \norm{\grad f}_{L^2(\mu)}^2 \geq
			\frac{ \norm{f}_{L^2(\mu)}^{2 + \frac{4}{d} }}{\norm{f}_{L^1(\mu)}^{\frac{4}{d} } }\,,
	\end{equation}
	where~$C_2 = C_2(d, U, \kappa)$ is a dimensional constant that can be bounded by
	\begin{equation*}
	  C_2 \leq
			2^{2+\frac{8}{d}} C_d^2
			\exp\paren[\Big]{ \frac{4 \norm{U}_\osc}{\kappa } }
	\end{equation*}
	where
	\begin{align}\label{e:cd}
		C_d
			\defeq \frac{1}{ \sqrt{\pi d (d-2)} } \paren[\Big]{ \frac{ \Gamma(d) }{ \Gamma(\frac{d}{2}) } }^{1/d}\,.
	\end{align}
	When $d=2$, the inequality~\eqref{e:nashRd1} needs to be replaced by
	\begin{align}\label{e:Nash2d}
	C_2'\norm{\nabla f}_{L^2(\mu)}^2\geq \frac{\norm{f}_{L^2(\mu)}^{4-\delta}}{\norm{f}_{L^1(\mu)}^{2-\delta}}\,,
	\end{align}
	where~$\delta\in(0,2)$ is arbitrary, and $C_2' = C_2'(\delta, U, \kappa)$ can be bounded by
	\begin{equation*}
	  C_2' \leq C''_2(\delta) \exp\paren[\Big]{ \frac{4\norm{U}_\osc}{\kappa } }\,,
	\end{equation*}
	for some $\delta$-dependent constant~$C''_2$.
\end{lemma}

Momentarily postponing the proof of Lemma~\ref{l:actuallyTheNashInequality}, we finish the proof of Lemma~\ref{l:L1Linf}.

\begin{proof}[Proof of Lemma~\ref{l:L1Linf}]
	Multiplying equation~\eqref{e:cpde} by $\rho_\infty \theta$ and integrating gives
	\begin{equation}\label{e:tmpL21}
	\partial_t\norm{\theta}_{L^2(\mu)}^2
		=-2\kappa\norm{\nabla\theta}_{L^2(\mu)}^2
		\leq-\frac{2 \kappa}{C_2}
			\frac{ \norm{\theta}_{L^2(\mu)}^{\frac{2(d+2)}{d}} }{ \norm{\theta}_{L^1(\mu)}^{\frac{4}{d}} } \,.
	\end{equation}
	We claim $\norm{\theta_t}_{L^1(\mu)}\leq \norm{\theta_0}_{L^1(\mu)}$.
	To see this, let~$\theta_{t}^+$ and $\theta_t^-$ be solutions of~\eqref{e:cpde} with initial data~$\max\set{\theta_0, 0}$ and $-\min\set{\theta_0, 0}$ respectively.
	By the comparison principle, we know~$\theta_{t}^\pm \geq 0$, and by linearity $\theta_t = \theta_{t, +} - \theta_{t, -}$.
	This implies
	\begin{equation*}
		\norm{\theta_t}_{L^1(\mu)}
			\leq \int_{\T^d} \paren{ \theta_{t}^+ + \theta_{t}^- } \, d\mu
			= \int_{\T^d} \paren{ \theta_{0}^+ + \theta_{0}^- } \, d\mu
			= \norm{\theta_0}_{L^1(\mu)}\,.
	\end{equation*}
	Using this in~\eqref{e:tmpL21} yields
	\begin{equation*}
	\partial_t\norm{\theta_t}_{L^2(\mu)}^2
		\leq-\frac{2 \kappa}{C_2}
			\frac{ \norm{\theta_t}_{L^2(\mu)}^{\frac{2(d+2)}{d}} }{ \norm{\theta_0}_{L^1(\mu)}^{\frac{4}{d}} } \,.
	\end{equation*}
	This is a differential inequality for~$\norm{\theta_t}_{L^2(\mu)}^2$, which can be solved to give
	\begin{align}\label{e:l1l21}
		\norm{\theta_t}_{L^2(\mu)}
			\leq
			\frac{(d\, C_2)^{d/4} }{(4\kappa t)^{d/4} }
			\norm{\theta_0}_{L^1(\mu)} \,.
	\end{align}

	Now let $\mathcal{P}_{s, t}(v')$ denote the solution operator to~\eqref{e:cpde} (i.e.\ the function $\vartheta_t$ defined by $\vartheta_t \defeq \mathcal P_{s, t}(v')(f)$ solves~\eqref{e:cpde} with initial data $\vartheta_s = f$).
	From~\eqref{e:l1l21}, we see
	\begin{equation*}
	  \norm{\mathcal{P}_{s, s+t}(v')}_{\dot L^1(\mu) \to \dot L^2(\mu)}
			\leq \frac{(d\, C_2)^{d/4} }{(4\kappa t)^{d/4} }
	\end{equation*}
	Moreover, since~$v'$ satisfies~\eqref{e:measurePreserving} we see that
	\begin{equation*}
	  \left(\mathcal{P}_{s,t}(v')\right)^*=\mathcal{P}_{s, t}(-v') \,,
	\end{equation*}
	where $\mathcal P_t(v')^*$ denotes the adjoint of $\mathcal P_t v'$ with respect to the~$L^2(\mu)$ inner-product.
	Consequently,
	\begin{align*}
		\norm{\mathcal{P}_{s, s+2t}(v')}_{\dot L^1\to \dot L^\infty}
			&\leq \norm{\mathcal{P}_{s+t, s+2t}(v')}_{\dot L^1\to \dot L^2}\norm{\mathcal{P}_{s, s+t}(v')}_{\dot L^2 \to \dot L^\infty}
		\\
			&=\norm{\mathcal{P}_{s+t, s+2t}(v')}_{\dot L^1\to \dot L^2}\norm{\left(\mathcal{P}_{s, s+t}(v')\right)^*}_{\dot L^1 \to \dot L^2}
		\\
			&\leq \frac{(d\, C_2)^{d/2}}{(4\kappa t)^{d/2} }\,.
	\end{align*}
	This in turn implies
	\begin{align}
		\norm{\theta(t)}_{L^\infty}
			\leq \frac{(d\, C_2)^{d/2}}{(\kappa t)^{d/2} } \norm{\theta_0}_{L^1} \,,
	\end{align}
Recalling the definition of $C_2$, we actually have
\begin{align}
(d\, C_2)^{d/2}\leq C 2^d \exp\paren[\Big]{ \frac{2d\norm{U}_\osc }{\kappa} }\,,
\end{align}
where $C$ is some universal constant independent of $d$.	
	which concludes the proof when $d \geq 3$.

	The proof when~$d = 2$ is similar and only involves using~\eqref{e:Nash2d} instead of~\eqref{e:nashRd1}.
\end{proof}

\subsection{The Nash and Poincar\'e inequalities.}

We conclude this section by proving the Nash (Lemma~\ref{l:actuallyTheNashInequality}) and Poincar\'e inequalities.

When $d \geq 3$, recall the standard Nash inequality states
\begin{align}\label{e:nash}
	C_d^2\norm{\nabla f}_{L^2}^2\geq \frac{ \norm{f - f_0}_{L^2}^{\frac{2(d+2)}{d}} }{ \norm{f - f_0}_{L^1}^{\frac{4}{d}} }\,,
\end{align}
where $C_d$ ia as defined in~\eqref{e:cd}, and $f_0 = \int_{\T^d} f \, dx$.

The Nash inequality above can be deduced from the Sobolev inequality and interpolation.
Indeed, the Sobolev inequality (see for instance~\cite{Lieb83}), implies
\begin{equation}\label{e:sobolev}
C_d\norm{\nabla f}_{L^2}\geq \norm{f - f_0}_{L^{2^*}}\,,
\qquad\text{where}\qquad
	2^*=\frac{2d}{d-2}\,.
\end{equation}
Since $d \geq 3$ note $2^* > 2$, and hence the interpolation inequality gives
\begin{equation*}
	\norm{f - f_0}_{L^2} \leq \norm{f - f_0}_{L^{2^*}}^{\frac{d}{d+2}} \norm{f - f_0}_{L^1}^{\frac{2}{d + 2}} \,.
\end{equation*}
Combined with~\eqref{e:sobolev} this implies~\eqref{e:nash} as claimed.

To prove the Nash inequality~\eqref{e:nashRd1} with respect to the measure~$\mu$, we first need an elementary result controlling the $L^p(\mu)$ 
difference to the mean when the underlying measure~$\mu$ is changed.
\begin{lemma}\label{l:eqnorm}
	Let $\tilde \rho_\infty$ be a probability density function on $\T^d$, and let~$\tilde \mu$ be the probability measure such that $d\tilde \mu = \tilde \rho_\infty \, dx$.
	Suppose there exists a constants~$B_1, B_2$ such that
	\begin{equation*}
		\frac{1}{B_1} \tilde \rho_\infty(x) \leq \rho_\infty(x) \leq B_2 \tilde \rho_\infty(x) 
		\quad\text{for all } x \in \T^d\,.
	\end{equation*}
	Then, for any $p \in [1, \infty)$, and any~$f \in L^p(\mu)$, we have
	\begin{equation*}
		\frac{1}{2 B_1^{1/p} } \norm{ f - \tilde f}_{L^p(\tilde \mu)}
			\leq \norm{f - \bar f}_{L^p(\mu)}
			\leq 2 B_2^{1/p} \norm{f - \bar f}_{L^p(\tilde \mu)} \,.
	\end{equation*}
	Here
	\begin{equation*}
		\bar f = \int_{\T^d} f \, d \mu
		\qquad\text{and}\qquad
		\tilde f = \int_{\T^d} f \, d\tilde \mu\,,
	\end{equation*}
	are the means of~$f$ with respect to the measures~$\mu$ and~$\tilde \mu$ respectively.
\end{lemma}
\begin{proof}
	From the triangle inequality, we note
	\begin{align*}
		\norm{f- \bar f}_{L^p(\mu)}
			&\leq \norm{f- \tilde f}_{L^p(\mu)}+\abs{\bar f - \tilde f}
			\leq  B_2^{1/p}\norm{f-\tilde f}_{L^p(\tilde \mu)}
				+ \abs[\Big]{\int_{\T^d}  (f- \tilde f) \, d\mu }
			\\
			&\leq  B_2^{1/p}\norm{f-\tilde f}_{L^p(\tilde \mu)}
				+ \norm{f - \tilde f}_{L^1(\mu)}
			\leq  2 B_2^{1/p}\norm{f-\tilde f}_{L^p(\tilde \mu)}\,.
	\end{align*}
	The proof of the lower bound is similar.
\end{proof}

We now prove Lemma~\ref{l:actuallyTheNashInequality}.
\begin{proof}[Proof of Lemma~\ref{l:actuallyTheNashInequality}]
To prove~\eqref{e:nashRd1} we note Lemma~\ref{l:eqnorm} implies
\begin{equation*}
  C_d^2 \norm{\grad f}_{L^2(\mu)}^2
	\geq C_d^2 \min(\rho_\infty) \norm{\grad f}_{L^2}^2
	\geq
		\frac{ \norm{f - \bar f}_{L^2(\mu)}^{\frac{2(d+2)}{d}} }{ 2^{2 + \frac{8}{d} } B \norm{f - \bar f}_{L^1(\mu)}^{\frac{4}{d}} }\,,
\end{equation*}
where~$\bar f = \int_{\T^d} f \, d\mu$ and
\begin{equation*}
  B= B(U, \kappa) = \frac{\max(\rho_\infty)^{1+\frac{2}{d} } }{\min(\rho_\infty)^{1 + \frac{4}{d} }}
		\leq \paren[\Big]{ \frac{\max(\rho_\infty)}{\min(\rho_\infty)} }^{1 + \frac{4}{d} }
		\leq \exp\paren[\Big]{ \paren[\Big]{1 + \frac{4}{d} } \frac{\norm{U}_\osc}{\kappa} }\,.
\end{equation*}
This finishes the proof when $d \geq 3$.

When $d=2$, the Sobolev inequality in~\eqref{e:sobolev} becomes
\begin{align}
C'\norm{\nabla f}_{L^2}\geq \norm{f-f_0}_{L^p}\,,
\end{align}
for any $2<p<\infty$. And the interpolation inequality gives
\begin{align}
\norm{f-f_0}_{L^2}\leq \norm{f-f_0}_{L^p}^{\frac{p}{2p-2}}\norm{f-f_0}_{L^1}^{\frac{p-2}{2p-2}}\,,
\end{align}
which further yields 
\begin{align}
C'^2\norm{\nabla f}_{L^2}^2\geq \frac{\norm{f-f_0}_{L^2}^{4-\frac{4}{p}}}{\norm{f-f_0}_{L^1}^{2-\frac{4}{p}}}\,.
\end{align}
We then apply Lemma~\ref{l:eqnorm} and get the desired result.
\end{proof}

Finally, for completeness, we conclude this section by stating the Poincar\'e inequality for the measure~$\mu$.
We do not use this in the proof of our main result, but only use it in Remark~\ref{r:rhsPositive} to comment that right hand side of~\eqref{e:tmixtdis} is nonnegative.
\begin{lemma}[Poincar\'e Inequality]\label{l:poincare}
	Let $\lambda_0$ be the smallest eigenvalue of $- \mathcal L_\kappa$ on $\dot L^2(\mu)$.
	Then~$\lambda_0$ is bounded below by
	\begin{equation}\label{e:lambda0Bound}
		\lambda_0 \geq 2\pi \exp\paren[\Big]{ \frac{-\norm{U}_\osc}{ 2\kappa}  }\,.
	\end{equation}
	Moreover, for all $f\in H^1\cap \dot L^2(\mu)$, we have
	\begin{equation}\label{e:poincareconst}
		\lambda_0 \| f\|_{L^2(\mu)}^2 \leq \|\nabla f \|_{L^2(\mu)}^2\,.
	\end{equation}
\end{lemma}
\begin{proof}
	Using~\eqref{e:H1norm} and standard spectral theory, we know
	\begin{equation*}
		\lambda_0 = \inf_{f \in \dot L^2(\mu) - 0} \frac{\norm{\grad f}_{L^2(\mu)} }{\norm{f}_{L^2(\mu)} }\,,
	\end{equation*}
	which immediately implies~\eqref{e:poincareconst}.
	To obtain~\eqref{e:lambda0Bound}, let $L^2_\mathrm{nc}$ denote the set of all non-constant $L^2$ functions, and note that the above implies
	\begin{align*}
		\lambda_0
			&= \inf_{f \in L^2_\mathrm{nc}(\mu)} \sup_{c \in \R}
					\frac{\norm{\grad f}_{L^2(\mu)} }{\norm{f - c}_{L^2(\mu)} }
			\geq \exp\paren[\Big]{ \frac{-\norm{U}_\osc}{2 \kappa}  }
				\inf_{f \in L^2_\mathrm{nc}} \sup_{c \in \R}
					\frac{\norm{\grad f}_{L^2} }{\norm{f - c}_{L^2} }
		\\
			&= \exp\paren[\Big]{ \frac{-\norm{U}_\osc}{2 \kappa}  }
				\inf_{f \in \dot L^2 - 0} 
					\frac{\norm{\grad f}_{L^2} }{\norm{f}_{L^2} }
			= 2 \pi \exp\paren[\Big]{ \frac{-\norm{U}_\osc}{2 \kappa}  } \,.
			\qedhere
	\end{align*}
\end{proof}

\section{Explicit asymptotics in discrete time}\label{s:heuristics}

In this this section we consider a discrete time version of~\eqref{e:Aeq}.
Namely, we will run equation~\eqref{e:langevin} (without the drift) for time~$1/A$; and then we will run the flow $A v'$ (without noise) for time~$1/A$.
Running the flow (without noise) corresponds to applying the $\mu$-measure preserving diffeomorphism~$\Phi_{s, s+1}$, defined in~\eqref{e:flowdef}.
If instead of applying~$\Phi_{s, s+1}$ (the flow map of a velocity field), we apply an arbitrary~$\mu$-measure preserving diffeomorphism, then we provide an example which is exponentially mixing with rate~\eqref{e:expMixRate}, where the behavior of both~$D$ and~$\gamma$ is known as~$\kappa \to 0$.
This is what leads to the heuristics~\eqref{e:Heusristics} described in Section~\ref{s:intro}.

Explicitly, suppose $\varPhi \colon \T^d \to \T^d$ is a~$\mu$-measure preserving diffeomorphism, and~$\rho^X_t(x, y)$ is the transition density of the solution to~\eqref{e:langevin}.
Let~$Y$ be the Markov process with transition probability
\begin{equation}
	\P( Y_{n+1} \in dy \given Y_n = x ) = \rho^X_{1/A}(x, \varPhi^{-1}(y) )\,.
\end{equation}
Alternately, one can (equivalently) define~$Y_{n+1}$ by letting~$Z$ be the solution of~\eqref{e:langevin} with initial data~$Z_0 = Y_n$, and then defining
\begin{equation}
	Y_{n+1} = \varPhi(Z_{1/A})\,.
\end{equation}
With this notation~$Y_n$ above serves as a proxy for~$X_t$ (the solution to~\eqref{e:Aeq}) where~$n$ and~$t$ are related through
\begin{equation}
	n = At\,.
\end{equation}

The notions of mixing, dissipation time, etc.\ in discrete time are defined analogously to those in the continuous-time setting.
To differentiate from the continuous time versions, in the discrete-time setting we will use~$\ndis$ and~$\nmix$ to denote the dissipation and mixing times respectively. 
The main results in this section are the following.

\begin{proposition}\label{p:DTimeMix}
	\begin{enumerate}[(1)]\reqnomode
	  \item
			There exists a~$\mu$-measure preserving, exponentially mixing diffeomorphism~$\varPhi$ whose mixing rate is 
			\begin{equation}\label{e:DTimeMixRate}
				h(n) = D e^{-\gamma n}\,,
			\end{equation}
			where where $D = D(d, \kappa)$, but $\gamma$ is independent of both~$\kappa$ and~$d$.

		\item
			If further~$U$ is in the form
			\begin{equation}\label{e:seppot}
				U(x)=\sum_{i=1}^d \tilde{U} (x_i) \quad\text{for some}\quad \tilde{U}\in C^2(\T) \,,
			\end{equation}
			then $D$ can be bounded by
			\begin{equation}\label{e:DBound}
				D \leq \sqrt{d} e^{O(1/\kappa)}
				\quad\text{as } \kappa \to 0\,.
			\end{equation}
	\end{enumerate}
\end{proposition}

\begin{proposition}\label{p:NdisNmix}
	If~$\varPhi$ is mixing with rate~$h$, then the mixing time and dissipation time of~$Y$ are bounded above by
	\begin{equation}
		\nmix
			\leq C d\paren[\Big]{ 1+\frac{\norm{U}_{\osc}}{\kappa}
				-\ln\paren[\Big]{ \frac{\kappa\ndis}{A} } }
				\ndis\,,
		\qquad
		\ndis \leq \frac{CA}{H(A)}\,.
	\end{equation}
	Here
	\begin{align}\label{e:dicreteH}
		H(A)\defeq \sup \set[\Big]{\lambda \st h\paren[\Big]{ \frac{\sqrt A}{2\sqrt{\lambda}} } \leq \frac{\kappa}{2\lambda}}\,.
	\end{align}
\end{proposition}

The proof of Proposition~\ref{p:NdisNmix} follows the same method as Proposition~\ref{t:fconv}, with the continuous-time energy decay replaced with the time discrete analogs (see Lemmas~3.1 and~3.2 in~\cite{FengIyer19}).
For brevity, we do not present it here.
We will prove Proposition~\ref{p:DTimeMix} below.

The main idea behind the proof of Proposition~\ref{p:DTimeMix} is to construct~$\mu$-exponentially mixing diffeomorphisms as conjugates of Lebesgue exponentially mixing diffeomorphisms.
There are many examples of Lebesgue exponentially mixing diffeomorphisms, such as the baker's map or toral automorphisms~\cite{KatokHasselblatt95,SturmanOttinoEA06}.
In the time inhomogeneous case, they can also be constructed as flow maps of alternating shear flows as in Section~\ref{s:appendixmixing} or~\cite{BlumenthalCotiZelatiEA22}.
To prove Proposition~\ref{p:DTimeMix}, however, we will need a Lebesgue exponentially mixing map with mixing rate that is independent of the dimension.
While many of the examples mentioned above likely have a mixing rate that can be bounded independent of the dimension, it is easiest to prove this for an explicit toral automorphism.
Once this has been established, a direct calculation will show that the pre-factor~$D$ in~\eqref{e:DTimeMixRate} may depend on~$\kappa$, but the exponential rate~$\gamma$ does not.

\begin{lemma}\label{l:toralDDim}
	There exists a diffeomorphism~$\varPsi \colon \T^d \to \T^d$ which is Lebesgue exponentially mixing with a rate that is independent of the dimension.
\end{lemma}
\begin{proof}
	We will choose~$\varPsi$ to be a toral automorphism.
	Recall, given any $A \in \SL(d, \Z)$, a \emph{toral automorphism} with matrix~$A$ is the map~$\varPsi \colon \T^d \to \T^d$ defined by
	\begin{equation}
		\varPsi(x) = A x \pmod{ \Z^d}\,.
	\end{equation}
	The mixing properties of these maps are well known (see for instance~\cite{Lind82,FannjiangWoowski03,FengIyer19}).
	In particular, if all eigenvalues of~$A$ are irrational and at least one lies outside the unit disk, then~$\varPsi$ is Lebesgue exponentially mixing.
	To prove this, note that a Fourier series expansion immediately shows
	\begin{equation}\label{e:H2CorrelationDecay}
		\ip{f \circ \varPsi^n, g} \leq \sup_{k \in \Z^d - 0} \frac{1}{\abs{A^n k} \abs{k}^2} \norm{f}_{\dot H^1} \norm{ g }_{\dot H^2}\,,
	\end{equation}
	for all test functions~$f \in \dot H^1$, $g \in \dot H^2$ (see for instance equation~(4.7) in~\cite{FengIyer19}).
	Now using Diophantine approximation results one can show
	\begin{equation}\label{e:AnkLower}
		\abs{A^n k} \abs{k^{d-1}} \geq \frac{\abs{\lambda_1}^n}{C_d} \,,
	\end{equation}
	where~$\lambda_1$ is an eigenvalue of~$A$ with the largest modulus, and $C_d$ is a dimensional constant (see for instance the inequality immediately after (4.8) in~\cite{FengIyer19}).
	This can be used to show~$\varPsi$ is Lebesgue exponentially mixing, however, constants appearing in the mixing rate will depend on the dimension.

	We will now choose~$A$ in a specific form that will ensure the mixing rate is independent of the dimension.
	Let
	\begin{equation}
		A_1 = \begin{pmatrix} 2 & 1\\ 1 & 1 \end{pmatrix}
		\qquad\text{and}\qquad
		A_2 = \begin{pmatrix}
			2 & -1 & 0\\
			0 & 1 & 1\\
			1 & 0 & 1
		\end{pmatrix}
	\end{equation}
	and choose~$A \in \SL(d, \Z)$ to be any block diagonal matrix with only $2 \times 2$ blocks~$A_2$,  or $3 \times 3$ blocks $A_3$ on the diagonal.
	One can directly check that both~$A_2$ and~$A_3$ are ergodic toral automorphisms.
	Since the domain of~$A_m$ is~$\T^m$, for $m \in \set{2, 3}$, the lower bound~\eqref{e:AnkLower} becomes
	\begin{equation}
		\abs{A_m^n k}\abs{k}^2
		\geq
		\abs{ A_m^n k } \abs{k}^{m-1} \geq \frac{ \lambda^n }{C} \,,
		\quad\text{for all }
		m \in \set{2, 3},~
		k \in \Z^m - 0\,,
	\end{equation}
	where
	\begin{equation}
		\lambda
			= \min_{m \in \set{2, 3}} \max \set{ \abs{\mu} \st \mu \text{ is an eigenvalue of } A_m }
			> 1 \,,
	\end{equation}
	and~$C$ only depends on $A_2, A_3$ (and hence is independent of~$d$).

	Now for~$k \in \Z^d$, write $k = (k_1, \dots, k_{d'})$ where each $k_i \in \Z^{m_i}$, $m_i \in \set{2, 3}$, corresponds to the block diagonal structure of~$A$.
	If~$k \neq 0$, at least one~$k_i$ must be non-zero, and hence
	\begin{equation}
		\abs{A^n k} \abs{k}^2 \geq \abs{A_{m_i}^n k_i} \abs{k_i}^2
			\geq \frac{\lambda^n}{C} \,.
	\end{equation}
	Combined with~\eqref{e:H2CorrelationDecay} this immediately implies
	\begin{equation}
		\norm{f \circ \varPsi^n}_{H^{-2}} \leq C \lambda^{-n} \norm{f}_{\dot H^1}\,.
	\end{equation}
	By H\"older's inequality
	\begin{align}
		\ip{f \circ \varPsi, g}
			&\leq \norm{f \circ \varPsi}_{H^{-2}}^{1/2} \norm{f\circ \varPsi}_{L^2}^{1/2} \norm{g}_{\dot H^1}
			\leq \sqrt{C} \lambda^{-n/2} \norm{f}_{\dot H^1}^{1/2} \norm{f}^{1/2}_{L^2} \norm{g}_{\dot H^1}
		\\
			&\leq \sqrt{C} \lambda^{-n/2} \norm{f}_{\dot H^1} \norm{g}_{\dot H^1}\,,
	\end{align}
	showing~$\varPsi$ is Lebesgue exponentially mixing with rate~$\sqrt{C} \lambda^{-n/2}$.
	Since~$C$ and~$\lambda > 1$ are independent of~$d$, this concludes the proof.
\end{proof}

\begin{proof}[Proof of Proposition~\ref{p:DTimeMix}]
	Let~$\varPsi \colon \T^d \to \T^d$ be the Lebesgue exponentially mixing diffeomorphism from Lemma~\ref{l:toralDDim}.
	Notice that this is completely independent of~$\kappa$, and neither~$D'$, $\gamma$, nor~$\varPsi$ depend on~$\kappa$.

	Let~$\psi \colon \T^d \to \T^d$ be a diffeomorphism such that the push forward of the measure~$\mu$ under~$\psi$ is the Lebesgue measure (i.e.\ for all Borel sets~$A$ we have $\mu(\psi^{-1}(A)) = m(A)$, where~$m$ is the Lebesgue measure).
	One can, for instance, prove the existence of such a map using optimal transport.
	We claim that
	\begin{equation}
		\varPhi = \psi^{-1} \varPsi \psi\,,
	\end{equation}
	is a~$\mu$-exponentially mixing map with exponential rate~$\gamma$.
	To see this, we note first that clearly~$\varPhi$ preserves the measure~$\mu$.
	Moreover, for any pair of test functions~$f, g \in \dot H^1(\mu)$, we have
	\begin{align}
		\ip{ f\circ \varPhi^n, g}_\mu
			= \ip{ f\circ (\psi^{-1} \varPsi^n \psi), g }_\mu
			= \ip{ (f\circ \psi^{-1}) \varPsi^n, g\circ \psi^{-1} }\,.
	\end{align} 
	Since $f, g$ have $\mu$-mean zero, the functions~$f \circ \psi^{-1}$ and~$g\circ \psi^{-1}$ must be Lebesgue mean-zero.
	Since~$\varPsi$ is Lebesgue exponentially mixing, this implies
	\begin{align}
		\ip{ (f\circ \psi^{-1}) \varPsi^n, g\circ \psi^{-1} }
			&\leq D' e^{-\gamma n} \norm{f\circ \psi^{-1}}_{\dot H^{1}}\norm{g\circ \psi^{-1}}_{\dot H^1}
		\\
			&\leq D' \norm{ \grad \psi^{-1}}_{L^\infty}^2 e^{-\gamma n} \norm{f}_{\dot H^1(\mu)}\norm{g}_{\dot H^1(\mu)}\,,
	\end{align}
	where we clarify that
	\begin{equation}
		\norm{\grad \psi^{-1}}_{L^\infty} = \norm[\Big]{ \sum_{i,j} \abs{ \partial_i \psi_j^{-1}  }^2 }_{L^\infty}^{1/2}\,.
	\end{equation}
	Hence $\varPhi$ is $\mu$-exponentially mixing with rate
	\begin{equation}\label{e:MixRateGradPsi}
		h(n) = D' \norm{\grad \psi^{-1}}_{L^\infty}^2 e^{-\gamma n}\,.
	\end{equation}
	This proves the first assertion of Proposition~\ref{p:DTimeMix}.
	\smallskip

	To prove the second assertion, note from~\eqref{e:MixRateGradPsi}, that the pre-factor~$D$ is bounded by
	\begin{equation}
		D \leq D' \norm{\grad \psi^{-1}}_{L^\infty}\,,
	\end{equation}
	where~$\psi$ is any diffeomorphism that pushes forward~$\mu$ onto the Lebesgue measure.
	When~$d = 1$ such maps are characterized by
	\begin{equation}
		\partial_x \psi^{-1} = \frac{e^{-U / \kappa}}{Z}\,.
	\end{equation}
	When~$d > 1$ and~$U$ is in the form~\eqref{e:seppot}, we can construct~$\psi$ using the one dimensional maps described above.
	Explicitly, define~$\tilde \varphi\colon \R \to \R$ by
	\begin{equation}
		\tilde \varphi (x) = \frac{1}{\tilde Z} \int_0^x e^{-\tilde U(y) / \kappa} \, dy\,,
		\quad\text{where}\quad
		\tilde Z \defeq \int_0^1 e^{-U(y) / \kappa} \, dy\,.
	\end{equation}
	Since~$U$ is $1$-periodic  we note~$\tilde \varphi(x + 1) = 1 + \tilde \varphi(x)$, and hence~$\tilde \varphi$ can be viewed as a map on the one dimensional torus.
	We now define~$\psi^{-1} \colon \T^d \to \T^d$ by
	\begin{equation}
		\psi^{-1}(x) = ( \tilde \varphi(x_1), \tilde \varphi(x_2), \dots, \tilde \varphi(x_d) )\,.
	\end{equation}
	Clearly the push forward of~$\mu$ under~$\psi$ is the Lebesgue measure, and hence
	\begin{equation}
		D \leq D' \norm{\grad \psi^{-1}}_{L^\infty}
			\leq
			\frac{D' \sqrt{d}}{ \inf_{x \in [0, 1]} \int_0^1 \exp\paren[\Big]{\frac{1}{\kappa}(\tilde U(y) - \tilde U(x))} \, dy }\,.
	\end{equation}
	This proves~\eqref{e:DBound}, concluding the proof.
\end{proof}

\section{Proof of exponential mixing of sawtooth shears (Theorem~\ref{t:SinMix}).}\label{s:modifiedmixing}
The objective of this section is to show that the modified shears in~\eqref{e:vDefAltShear} are exponentially mixing with probability~$1$, as stated in Theorem~\ref{t:SinMix}.
The proof involves the analysis of geometric ergodicity of a pair of trajectories of the velocity field~$v$.
This study was initiated in~\cite{BaxendaleStroock88} and further developed in~\cite[Theorem 4]{DolgopyatKaloshinEA04},~\cite[Theorem 1.3]{BedrossianBlumenthalEA22}, and~\cite[Theorem 1.1]{BlumenthalCotiZelatiEA22}. Among these results our proof is closest to Theorem~1.3 in~\cite{BedrossianBlumenthalEA22} and differs from Theorem~1.3 in~\cite{BedrossianBlumenthalEA22} only in one aspect. The proof in ~\cite{BedrossianBlumenthalEA22} uses H\"{o}rmander's condition to obtain irreducibility and a positive Lyapunov exponent of underlying Markov processes.
We cannot use H\"{o}rmander's condition in our context.
Instead, we use the Rashevsky--Chow Theorem~\cite[Theroem 5]{Sachkov21} (see Theorem~\ref{T: chow}, below) to obtain the same results. 
\subsection{Modified Sawtooth Shears in Two Dimensions}\label{s:modifiedTentConstruction}
We will first prove that the modified, randomly shifted, sawtooth shears are almost surely exponentially mixing in two dimensions.
Following this we will prove the remaining conclusions of Theorem~\ref{t:SinMix}.

We begin by writing down the function~$F$ with the sawtooth shaped derivative (shown in Figure~\ref{f:sawTooth}).
Define
\begin{equation}\label{e:psi0}
	F(x) = F_0(x) \defeq
	\begin{dcases}
		2x^2 & x\in[0, \tfrac{1}{4}] \\
		-2 (x - \tfrac{1}{4}) (x - \tfrac{3}{4} ) + \tfrac{1}{8} & x\in[ \tfrac{1}{4}, \tfrac{3}{4}]\\
		2(x-1)^2 & x\in[ \tfrac{3}{4}, 1]\,,
	\end{dcases} 
\end{equation}
and extended periodically to $x \in \R$ (see Figure~\ref{f:sawTooth}, right).
For $\alpha\in[0, 1]$ we define $F_\alpha$ by 
\begin{equation}\label{e:psiDef}
	F_{\alpha}(x) \defeq F_0(x-\alpha)\,.
\end{equation}

Given~$F_\alpha$, we define the associated velocity fields~$v_\alpha$ using~\eqref{e:vDef} by replacing~$F$ with $F_\alpha$.
Explicitly, we define
\begin{gather}
	\label{e:valpha1}
	v_{\alpha, 1}=
	\frac{1}{p}\grad^\perp(p F_\alpha(x_1))
	=
	\frac{1}{\kappa}
	\begin{pmatrix}
		U_{x_2} F_\alpha(x_1)\\
		\kappa F'_\alpha(x_1) - U_{x_1} F_\alpha(x_1)
	\end{pmatrix}
	\,,
	\\
	\label{e:valpha2}
	v_{\alpha, 2}=
	\frac{1}{p} \grad^\perp (p F_\alpha(x_2) )
	= \frac{1}{\kappa}\begin{pmatrix}
		U_{x_2} F_\alpha(x_2)-\kappa F'_\alpha(x_2)\\
		-U_{x_1} F_\alpha(x_2)
	\end{pmatrix}
	\,,
\end{gather}
where~$\grad^\perp = \grad^\perp_{1, 2}$ is the skew gradient in two dimensions.
Stream plots of this flow (for $U$ given by~\eqref{e:Sinpotential}) are shown in Figure~\ref{f:stream}.

For notational convenience, define~$V_n = v_n$, where~$v$ defined by~\eqref{e:vDefAltShear}.
Note
\begin{equation}\label{e:VnDef}
	V_n = \beta_n v_{\alpha_n, i_n}\,,
\end{equation}
where $(\alpha_n, \beta_n, i_n)$ are i.i.d.\ random variables that are uniformly distributed on the parameter space $[0, 1]\times [0, 1] \times \set{1, 2}$.
We will show that the paths of the random flow obtained by composing the time $1$ flows of each of the vector fields~$V_1$, \dots, $V_n$ are almost surely exponentially mixing.

\subsection{Conditions Guaranteeing Exponential Mixing}
To prove that the paths of the random flow above is exponentially mixing, it will be convenient to use Theorem~1.4 from~\cite{BlumenthalCotiZelatiEA22}.
For clarity of presentation we introduce the required preliminaries and restate this result below.

Consider the Markov process~$X$ defined by
\begin{equation}\label{e:AlexsGreatestEquation}
	X_{n+1}=\varphi^{V_{n+1}}(X_{n}) \,.
\end{equation}
Here the notation $\varphi^v$ denotes time $1$ flow map of the vector field $v$, and the vector fields~$V_n$ were defined in~\eqref{e:VnDef}.
Define the random flow~$\phi_n$ by
\begin{equation}\label{e:phidef}
	\phi_n \defeq \varphi^{V_n} \circ \varphi^{V_{n-1}} \circ \cdots \varphi^{V_1}.
\end{equation}

Let $\Pro$ be the projective space of $\R^2$, that is the collection of lines in $\R^2$ that pass through the origin. 
For initial condition $(x, u)\in \T^2\times\Pro$ the projective process on $\T^2\times \Pro$ is defined by 
\begin{equation}\label{e: projpros}
	(X_n, U_n)= \paren[\Big]{ \phi_n(x), \frac{ D_x \phi_n u }{\abs{ D_x \phi_n u }} } \, .
\end{equation}
Given an initial condition $(x, g)\in \T^2\times \SL_2(\R)$ the rescaled derivative process on $\T^2\times \SL_2(\R)$ is defined by
\begin{equation}\label{e: derivpros}
	(X_n, A_n)=\left(\phi_n(x), \frac{D_x \phi_n g}{(\det D_x \phi_n)^{1/2}}\right).
\end{equation}
For initial condition $(x, y)\in \mathcal{D}^{\comp} \defeq \T^2\times\T^2-\{(x, x)\}$ the two point process $(X_n, Y_n)$ on $\mathcal{D}^{\comp}$ is defined by
\begin{equation}\label{def: 2pointp}
	(X_n, Y_n)=(\phi_n(x), \phi_n(y)).
\end{equation}
\begin{remark}\label{r:lip}
	Since each velocity field is Lipschitz, each flow is Lipschitz by Gr\"{o}nwall's inequality. That is,
	\begin{equation}\label{e:lip}
		\vert \varphi^v_t(x)-\varphi^v_t(y) \vert \leq \vert x - y \vert e^{t C_v}
	\end{equation}
	when $x$ and $y$ are sufficiently close. Thus, each flow is differentiable almost everywhere and the processes, ~\eqref{e: projpros} and ~\eqref{e: derivpros}, are well defined. 
\end{remark}

Theorem~1.4 in~\cite{BlumenthalCotiZelatiEA22} can now be stated as follows.
\begin{proposition}[Theorem $1.4$ in~\cite{BlumenthalCotiZelatiEA22}]\label{T:micheleMixing} 
	Assume the following conditions:
	\begin{enumerate}\reqnomode
		\item
			The one point process, two point process and the projective process are all aperiodic.
		\item
			The one point process has a positive Lyapunov exponent.
	  \item
			The one point process~\eqref{e:AlexsGreatestEquation} and the projective process~\eqref{e: projpros} are uniformly geometrically ergodic.
			Let $\mu$ be the unique invariant measure of the one point process.
		\item
			The two point process~\eqref{def: 2pointp} has a Lyapunov function~$\mathcal V \in L^1( \T^4, \mu \times \mu )$ of the form
			\begin{equation}\label{e: lyapunovfunctionform}
				\mathcal{V}(x, y)=d(x, y)^{-p}\chi(x, y) \hbox{ for some small } p>0 \,,
			\end{equation}
			where $\chi(x,y)$ is a continuous function which is bounded both from above and away from $0$.
		\item The two point process is $\mathcal{V}$-geometrically ergodic.
	\end{enumerate}
	Then there exists a deterministic constant $\gamma>0$ and a random constant $D(\omega)$ such that for all mean-zero functions $f, g\in H^1(\T^2, \mu)$ we have that
	\[
	\ip{ f\circ \phi_{n}, g }_\mu \leq D(\omega) e^{-\gamma n} \norm{f}_{\dot H^1} \norm{g}_{\dot H^1}
	\]
	almost surely.
\end{proposition}

We now briefly recall the notions used in Proposition~\ref{T:micheleMixing}.
A \emph{Lyapunov exponent} is the asymptotic rate at which tangent vectors are stretched or shrunk under the iterates of a dynamical system.
Precisely, they are limits of the form
\[
  \lim_{n \to \infty} \frac{ \ln \norm{ D_x \phi_n w } }{n}
\]
for~$x \in \T^2$ and a tangent vector.
A general theory of Lyapunov exponents for random dynamical systems can be found in e.g.~\cite[Chapter 3]{Arnold98}.

A \emph{Lyapunov function} for a Markov process with state space $E$ is a function 
\begin{equation}\label{e: lyapunovfunction}
	\mathcal{V}:E\to[1, \infty)  \hbox{ such that } \EX^x [\mathcal{V}(X_1)]\leq \lambda V(x)+C
\end{equation}
for some constants $0<\lambda<1$, $C>0$.
A Markov process is~$\mathcal V$-geometrically ergodic if there exists $\gamma>0$ and a unique invariant distribution, $\mu$, such that
\begin{equation}\label{e: Vgeomerg}
	\norm{P^n(x, \cdot)-\mu}_{TV}\leq \mathcal{V}(x) e^{-n \gamma} \hbox{ for all } x\in E.
\end{equation}
The process is said to be \emph{uniformly geometrically ergodic} if the function~$\mathcal V$ in~\eqref{e: Vgeomerg} is constant,
It is well known (see for instance~\cite[Theorem~15.0.1]{MeynTweedie09}) that if the process is aperiodic, irreducible and there exists a Lyapunov function $\mathcal{V}$ with compact sub-level sets then he process is $\mathcal{V}$-geometrically ergodic.

\begin{remark}
The state space of the two point process is not compact. Thus, showing geometric ergodicity is not immediate. We take the Lyapunov function of the two point process to be a specific perturbation of the principle eigenfunction of the projective process, see~\cite[Section 5]{BedrossianBlumenthalEA22}. In this case equation~\eqref{e: lyapunovfunction} states that two particles which are close together move away from each other on average. For more information on geometric ergodicity of Markov processes with non-compact state spaces see~\cite[Chapter 15]{MeynTweedie09}.
\end{remark}	

Note that the processes~\eqref{e:AlexsGreatestEquation},~\eqref{e: projpros}, and~\eqref{def: 2pointp} are all aperiodic.
Indeed, for all $z$ in the processes' state space and every $\epsilon > 0$, the flow of any vector field~$v$ with sufficiently small amplitude will stay inside an~$\epsilon$-ball centered at its initial position.
This will ensure $\P(z, B(z, \epsilon))>0$, showing that the processes~\eqref{e:AlexsGreatestEquation},~\eqref{e: projpros}, and~\eqref{def: 2pointp} are all aperiodic.
Therefore to prove a sequence of randomly shifted modified sawtooth shears are mixing we must show the following.
	\begin{enumerate}
	\item The processes~\eqref{e:AlexsGreatestEquation} and~\eqref{e: projpros} are uniformly geometrically ergodic.
	\item The existence of a positive Lyapunov exponent for~\eqref{e:AlexsGreatestEquation}.
	\item Existence of a Lyapunov function as in equation \eqref{e: lyapunovfunctionform}.
	\item The process~\eqref{def: 2pointp} is $\mathcal{V}$-geometrically ergodic with respect to a function of the form~\eqref{e: lyapunovfunctionform}.
\end{enumerate}   

To prove the first item it suffices to prove the processes,~\eqref{e:AlexsGreatestEquation} and~\eqref{e: projpros}, are irreducible and Feller (see for example Theorem~15.0.1 in~\cite{MeynTweedie09}).
We prove irreducibility in Lemma~\ref{kappa_deriv} using the Rachevsky--Chow theorem, and prove the Feller condition in Lemma~\ref{L: mufeller}.
The second item follows from irreducibility, and is shown in Lemma~\ref{l: poslyaexpon} below.
The third item follows from~\cite{BedrossianBlumenthalEA22} Section $5$, since the projective process is irreducible and uniformly geometrically ergodic.
Finally, the fourth item follows from section~$6$ in~\cite{BedrossianBlumenthalEA22},  and the fact that both the two point process~\eqref{e:AlexsGreatestEquation} and the projective process~\eqref{e: projpros} are irreducible and Feller.
\begin{remark}
	In \cite{BlumenthalCotiZelatiEA22}, the authors needed one more condition to show $\mathcal{V}$-geometric ergodicity. The condition was on the existence of an open so called small set \cite[Chapter 5]{MeynTweedie09}. Theorem~5.5.7 and Proposition~6.2.8  in~\cite{MeynTweedie09} imply that compact sets are small when the process is irreducible, aperiodic, and Feller. Thus, any small enough open ball is an open small set. 
\end{remark}


\subsection{Irreducibility}
	In order to prove irreducibility, we will use a controllability result of Rachevsky~\cite{Rashevskii38} and Chow~\cite{Chow40}, which shows that if a collection of vector fields satisfies a H\"ormander condition, then any two points are connected by a composition of flows.
	\begin{theorem}[Rashevsky--Chow~(Theorem 5 in~\cite{Sachkov21})]\label{T: chow} 
		Let $M$ be a smooth connected manifold, and $\mathcal F$ a collection of vector fields on $M$. Suppose that for every $x\in M$, the Lie algebra~$\operatorname{Lie}_x(\mathcal F)$ spans the tangent space $T_x M$.
		Then for every $x, y\in M$ there exists a sequence of times $t_1, t_2,\dots t_n$ and vector fields $v_1, v_2,\dots v_n\in\mathcal{F}$ with flows $\varphi^i$ such that 
		\[
		\varphi_{t_n}^n(\varphi_{t_{n-1}}^{n-1}(\cdots(\varphi^1_{t_1}(x))\cdots))=y.
		\]
	\end{theorem}
	Recall that Lie algebra~$\operatorname{Lie}_x \mathcal F$ is defined by
	\begin{equation*}
		\operatorname{Lie}_x \mathcal F=\operatorname{span} \paren[\Big]{ \bigcup\nolimits_{n}\set{v(x) \st v\in\mathcal{F}_n}}\,,
	\end{equation*}
	where $\mathcal{F}_0= \set{v}_{v\in \mathcal{F}}$, and~$\mathcal F_{n}$ is defined inductively by
	\begin{equation}
		\mathcal{F}_{n+1} \defeq \mathcal{F}_{n}\cup\set{[v_1, v_2] \st v_1\in\mathcal{F}_{n},~v_2\in \mathcal{F}}\,.
	\end{equation}
	Here the notation $[v_1, v_2]$ denotes the Lie bracket of two vector fields $v_1$, $v_2$ on a smooth manifold $M$.
	Recall the Lie bracket is the derivative of $v_2$ along the flow of~$v_1$, and can be computed by
	\[
	[v_1, v_2]= \operatorname{D}_{v_1} v_2 - \operatorname{D}_{v_2}v_1\,,
	\]
	where $\operatorname{D}_v$ is the directional derivative in the direction of $v$.

	We will now use Theorem~\ref{T: chow} to prove irreducibility.

	\begin{lemma}\label{kappa_deriv}
		The one point process~\eqref{e:AlexsGreatestEquation}, the projective process~\eqref{e: projpros}, the rescaled derivative process~\eqref{e: derivpros}, and the two point process~\eqref{def: 2pointp} are all irreducible. 
	\end{lemma}
	\begin{proof}
		To show the irreducibility of the one point process we need to demonstrate that the Lie algebra generated by vector fields $v_{\alpha, 1}$ and $v_{\alpha, 2}$ (equations~\eqref{e:valpha1} and~\eqref{e:valpha2}, respectively), is two-dimensional for every point $x\in\T^2$. It is straightforward to verify directly, and we do not explicitly do that. Instead, we conclude it from the irreducibility of the two-point process. The irreducibility of the projective process follows from that of the rescaled derivative process. Thus we only need to show irreducibility of the two-point process and the rescaled derivative process.

		The proofs of irreducibility of the rescaled derivative process and the two-point process are similar, and we consider the two-point process first. 
		We need to show that for a dense connected subset of $\mathcal{D}^c$, the corresponding Lie algebra generated by the vector fields
		\begin{equation}
			\set{(v_{\alpha, i}^T, v_{\alpha, i}^T)^T}
		\end{equation}
		has dimension $4$.
		We will choose this dense connected subset to be the set~$M$ defined by
		\begin{equation}\label{e:2pirredspace}
			M\defeq\mathcal{D}^c-\set{(x, y)\in\T^2 \times \T^2 \st x_1= y_1 \hbox{ or } x_2= y_2}.
		\end{equation}
		Fix a pair $(x,y) \in M$, 
		$x=(x_1, x_2)$, $y=(y_1, y_2)$. For our $x_1$ and $y_1$ we can choose $\alpha\in\left[-\frac{1}{4}, \frac{3}{4}\right]$ so that the stream function~\eqref{e:psiDef} satisfies
		\begin{align*}
			F_{\alpha}(x_1) &=-2 (x_1 - \alpha) \paren[\Big]{ x_1 - \paren[\Big]{ \alpha + \frac{1}{2}}} + \frac{1}{8}\,,\\
			F_{\alpha}(y_1) &=2 (y_2 - \alpha)^2 \,,
		\end{align*}
		respectively. This gives that $(v_{\alpha, 1}^T(x), v_{\alpha, 1}^T(y))$ is a quadratic polynomial in $\alpha$. That is,
		\begin{equation}\label{e:quadpoly}
			\begin{pmatrix}
				v_{\alpha, 1}(x)\\
				v_{\alpha, 1}(y)
			\end{pmatrix}=w_0(x, y)+\alpha w_1(x, y)+\alpha^2 w_2(x, y)\,,
		\end{equation}
		for some vector fields  $w_0$, $w_1$, $w_2$ on $\T^2 \times \T^2$.
		This implies that the span of $(v_{\alpha, 1}(x), v_{\alpha, 1}(y))^T$ is contained in 
		\[
		\operatorname{Span}\set{w_0(x, y),~w_1(x, y),~w_2(x, y)}.
		\]
		By taking $\tilde{\alpha}\equiv \alpha+1/2 \pmod{1}$, we have that $(v_{\alpha, 1}(x), v_{\alpha, 1}(y))^T$ is also a quadratic polynomial in $\tilde{\alpha}$.
		Thus we write
		\[
		\begin{pmatrix}
			v_{\alpha, 1}(x)\\
			v_{\alpha, 1}(y)
		\end{pmatrix}=\tilde{w}_0(x, y)+\tilde{a} \tilde{w}_1(x, y)+\tilde{a}^2 \tilde{w}_2(x, y)\,,
		\]
		for some vector fields~$\tilde w_0$, $\tilde w_1$ and~$\tilde w_2$ on $\T^2 \times \T^2$.
		Thus, the $6$ vector fields
		\[
		w_0(x, y),~w_1(x, y)~,w_2(x, y),~\tilde{w}_0(x, y)~\tilde{w}_1(x, y),~\tilde{w}_2(x, y)
		\]
		are all contained in the Lie algebra at $(x, y)$.
		Similarly, for our $x_2$ and $y_2$ we can choose another $\alpha\in[0, 1]$, and obtain another $6$ vector fields from $v_{\alpha, 2}$.
		We can directly compute and check that the span of these $12$ vector fields has dimension $4$ if $x_1\neq y_1$ and $x_2\neq y_2$.
		Since the calculation is somewhat tedious to verify, we do this symbolically and the code verifying this can be found~\cite{ChristieFengEA23}.

		Also note that the set $M$ (defined in~\eqref{e:2pirredspace}) is connected.
		Thus we can apply the Rachevsky-Chow Theorem~\ref{T: chow} and conclude irreducibility of the two-point process. 


		We now show the rescaled derivative process is irreducible. The corresponding Lie algebra is in the tangent space at 
		$(x, A) \in \T^2\times \SL_2\R$. Let $A_t(v, x)$ be the derivative matrix of the flow $v$ at $x$ and time $t$ rescaled by its determinant, and $a_1$, \dots, $a_4$ be the individual entries of this matrix.
		That is, define
		\begin{equation}\label{e:matrixAirred}
		A_t(v, x) \defeq \frac{D_x\varphi_t^v}{\sqrt{\det D_x \varphi_t^v}}=
		\begin{pmatrix}
			a_{1} & a_{2}\\
			a_{3} & a_{4}
		\end{pmatrix}\,.
		\end{equation}
		We consider the vector fields of the form
		\begin{equation}\label{e:derivprocessvectors}
			\tilde{v}_{\alpha,1}=\begin{pmatrix}
				v_{\alpha, 1}\\
				\partial_t a_{1}\\
				\partial_t a_{2}\\
				\partial_t a_{3}\\
				\partial_t a_{4}
			\end{pmatrix}
			\,,
			\quad
			\tilde{v}_{\alpha,2}=\begin{pmatrix}
				v_{\alpha, 2}\\
				\partial_t a_{1}\\
				\partial_t a_{2}\\
				\partial_t a_{3}\\
				\partial_t a_{4}
			\end{pmatrix}\,.
		\end{equation}
		That is, the vector fields have the first two coordinates from the vectors~$v_{\alpha, 1}$ or $v_{\alpha, 2}$ and last four coordinates as the time derivatives of the corresponding matrix $A_t$.

		We now show the collection of these vector fields satisfies the conditions of Theorem~\ref{T: chow}.
		For this, it suffices to show that for every~$x \in \T^2$, the Lie algebra at $(x, \mathit{Id})$ has dimension $5$.
		Since the rescaled derivative process is a matrix flow, given any other $g\in \SL_2(\R)$, the Lie algebra also has dimension~$5$. 

		Since each $v$ is piecewise-smooth,
		we compute the Lie bracket on each smooth piece.
		By selecting $\alpha_1$ and $\alpha_2$ such that
		\begin{align*}
			F_{\alpha_1}(x_1) &=-2 (x_1 - \alpha_1) \paren[\Big]{ x_1 - \paren[\Big]{ \alpha_1 + \frac{1}{2}}} + \frac{1}{8}\,,\\
			F_{\alpha_2}(x_2) &=-2 (x_2 - \alpha_2) \paren[\Big]{ x_2 - \paren[\Big]{ \alpha_2 + \frac{1}{2}}} + \frac{1}{8}\,,
		\end{align*}
		we get $12$ vectors that are contained in the Lie algebra at $(x, \mathit{Id})$ from the same methodology as that done in the case of the two point process. 

		It turns out that the span of these 12 vector fields  is contained in the span of
		\begin{equation}\label{e:derivvectors}
			\begin{pmatrix}
				\kappa\\
				0\\
				0\\
				0\\
				\partial_1 U\\
				\partial_1 U
			\end{pmatrix},
			\quad
			\begin{pmatrix}
				0\\
				\kappa\\
				0\\
				0\\
				-\partial_2 U\\
				-\partial_2 U\\
			\end{pmatrix},
			\quad
			\begin{pmatrix}
				0\\
				0\\
				1\\
				0\\
				2\\
				1
			\end{pmatrix},
			\quad
			\begin{pmatrix}
				0\\    
				0\\
				0\\
				1\\
				0\\
				0
			\end{pmatrix},
			\begin{pmatrix}
				0\\    
				0\\
				0\\
				0\\
				\partial_{1}U\partial_2U+\kappa \partial_{1, 1}U\\
				\partial_{1}U\partial_2U+\kappa \partial_{1, 1}U
			\end{pmatrix}\,.
		\end{equation}
		While this can also be directly checked, the computation is somewhat tedious.
		Thus we include code verifying this symbolically~\cite{ChristieFengEA23}.

		The span of the five vectors shown in~\eqref{e:derivvectors} may be less than five-dimensional. If, however, the $5^{th}$ vector is non-zero, the span is  five-dimensional. It is not easy to verify that the the $5^{th}$ vector is non-zero
		on a connected dense set. Therefore, we use different vectors to obtain the full rank.  

		Computing the Lie bracket of 
		$\tilde{v}_ {\alpha_1, 1}$ and $\tilde{v}_{\alpha_2, 2}$, defined in~\eqref{e:derivprocessvectors}, we obtain a quadratic polynomial in $\alpha_1$ and $\alpha_2$. Similar to~\eqref{e:quadpoly}, we now obtain a collection of 9 vectors that lie in our Lie Algebra.
		Using a computer to perform symbolic Gauss-Jordan elimination, we can use these 9 vectors and~\eqref{e:derivvectors} to obtain another set of 9 vectors, such that their first $4$ coordinates are $0$.
		(The code for this can also be found~\cite{ChristieFengEA23}.)
		The resulting vectors, and the last vector in~\eqref{e:derivvectors} yields the following non-degeneracy condition: the dimension of the Lie algebra at $(x, \mathit{Id})$ is $5$ if at least one of the following five inequalities hold
		\begin{equation}\label{e:derivirred}
			\begin{aligned}
				0 &\neq \partial_{1, 2} U\\
				0 &\neq \partial_{1}U\partial_2U+\kappa \partial_{1, 2}U \\
				0 &\neq \partial_{1}U\partial_2 U+\left(\partial_1 U\right)^2+3\kappa \partial_{1, 2} U\\
				0 &\neq 16 \kappa+\partial_{2, 2}U+\partial_{1, 1}U-3 \partial_{1, 2} U\\
				0 &\neq \kappa \partial_{2, 2} U +\kappa\partial_{1, 1} U +4\left(\partial_1 U\right)^2+\kappa \partial_{1, 2} U\,.	
			\end{aligned}
		\end{equation}
		We claim that for any $\kappa \neq 0$, at least one of the above non-degeneracy conditions must hold.
		Indeed if all the above non-degeneracy conditions fail, then we must have
		\[
		\Delta U(x) =0 \quad \text{and}\quad  \Delta U(x)=-16\kappa\,,
		\]
		which is impossible. Thus, the Lie algebra at $(x, \mathit{Id})$ has dimension $5$ and we can apply the Rachevsky--Chow Theorem.
		This completes the proof of irreducibility of the rescaled derivative process~\eqref{e: derivpros}.
	\end{proof}
	\begin{remark}\label{r:derivcontrol}
		Note that we demonstrated controllability of the two point process only on a connected dense set of $\mathcal{D}^c$, but controllability of the rescaled derivative process was shown on the entire set $\T^2\times\SL_2(\R)$.
	\end{remark}

\subsection{Positivity of top Lyapunov Exponent}
	Controllability of the rescaled derivative process implies that the top Lyapunov exponent is positive.
	This follows from the following version of Furstenberg's criterion.

	\begin{theorem}[Furstenberg's criterion]\label{T: Furst}
		For a sequence of elementary events $\underline{\omega}=(\omega_1, \omega_2, \dots)\in\Omega^\N$, continuous random dynamical system $f=f_{\omega}(x)$, $x \in X$ and measurable $A: \Omega\times X\to \operatorname{SL}_d(\R)$ consider a linear cocycle
		\begin{equation}\label{e:rdsnotation}
			f^n_{\underline{\omega}}=f_{\omega_n}\circ\dots\circ f_{\omega_1} \hbox{ and } A^n_{\underline{\omega}, x}=A_{w_n, f^{n-1}_{\underline{\omega}}(x)}\circ\dots A_{\omega_1, x}.
		\end{equation}
		Suppose $A^n_{\underline{\omega}, x}$ is integrable, and $\pi$ is the invariant measure of the Markov process corresponding to $f$. Then $A^n_{\underline{\omega}, x}$ has two Lyapunov exponents. Let $\lambda^+ \geq 0$ be the larger of the two.
		Then $\lambda^+$ can only be $0$ if there exists a family of probability measures $\{\nu_x\}_{x\in\supp\{\pi\}}$ on $\Pro$ varying measurably in $x$ such that for every $x\in \operatorname{Support}(\pi)$, $n\geq1$, and all pairs
		\begin{equation}\label{e:supplaw}
			(y, A)\in \operatorname{Support}(\operatorname{Law}( f^n_{\underline{\omega}}(x), A^n_{\underline{\omega}, x}))
		\end{equation}
		we have that the  pushforward of $\nu_y$ satisfies
		\begin{equation}\label{e:projectivemeasures}
			\left(A^{T}\right)_* \nu_y= \nu_x.
		\end{equation}
	\end{theorem}

	\begin{remark}
		There are many versions of the Furstenberg's criterion. We rescale the derivative process to use this version of Furstenberg's criterion. The version of Furstenberg's criterion given by Theorem~\ref{T: Furst} is a combination of Proposition~2 and Theorem~3 in~\cite{Ledrappier86} with one difference. This version has an extra conclusion that Equation~\eqref{e:projectivemeasures} holds for all elements in the support of the $n$-step transition probabilities. The conclusion follows from arguments done in Bougerol’s paper \cite{Bougerol88}, and was explicitly proven by A. Blumenthal et al. \cite[Proposition 2.9]{BlumenthalCotiZelatiEA22}. 
	\end{remark}	

	Let us recall definitions of notions that arise in Theorem~\ref{T: Furst} (refer to~\cite[Chapter 3]{Arnold98}). Briefly recall that a {\it continuous random dynamical system} (with independent increments in our case) on a metric space $X$, over a probability space $(\Omega, \mathcal{G}, \Prob)$, is a mapping $(\omega, x) \mapsto f_\omega(x)$ such that for every fixed $x\in X$ the 
	function $f_\omega(x)$ is a random variable in $\omega$, and for every fixed $\omega \in \Omega$ the function $f_\omega(x)$ is continuous in $x$. Such random dynamical systems {\it correspond} to Markov processes with transition probabilities 
	\[
	P(x, A)=\Prob(f_\omega(x)\in A) \,.
	\]
	A cocycle is {\it integrable} if 
	\begin{equation}\label{e:cointcond}
	\int\left(\ln^+ \norm{ A_{\omega, x}}+\ln^+ \norm{A^{-1}_{\omega, x}}\right) \, d\Prob(\omega) \, d\pi(x) \,.
	\end{equation}

	Note that the rescaled derivative process~\eqref{e: derivpros} is a continuous random dynamical system with integrable linear cocycle, and so we can apply Theorem~\ref{T: Furst} to show that the one point process~\eqref{e:AlexsGreatestEquation} has a positive Lyapunov exponent.
	\begin{lemma}\label{l: poslyaexpon}
		The top Lyapunov exponent of the one point process~\eqref{e:AlexsGreatestEquation} is positive.
	\end{lemma}
	\begin{proof}
		Suppose towards contradiction that the top Lyapunov exponent of the one point process~\eqref{e:AlexsGreatestEquation} is $0$. This implies that the top Lyapunov exponent of the linear cocycle $A_n$ (defined in~\eqref{e: derivpros}) is $0$.
		Indeed, since~$U$ is bounded
		the measure $\mu$ is equivalent to the Lebesgue measure on $\T^2$. So the $\mu$-measure preserving map $\phi_n$ satisfies
		\begin{equation*}
			0<C_1\leq\det D_x \phi_n(x)\leq C_2 < \infty\,,
		\end{equation*}
		uniformly in $n\in\N$, $x \in \T^2$ and the parameters $(\alpha, \beta, i)\in[0, 1]\times[0, 1]\times\{0, 1\}$. This implies that the top Lyapunov exponent of the one point process is the same as the top Lyapunov exponent of the rescaled derivative process~\eqref{e: derivpros}. Indeed, for any $v\in\R^2$ we have
		\[
		\lim_{n\to\infty}\frac{1}{n}\ln \norm{A_n v}=\lim_{n\to\infty}\frac{1}{n}\left(\ln \norm{D_x\phi_n v}-\ln\left(\det D_x \phi_n\right)\right)=\lim_{n\to\infty}\frac{1}{n}\ln \norm{D_x\phi_n v},
		\]
		almost surely with respect to $\Prob\times \mu$.

		Now assume that the top Lyapunov exponent of the linear cocycle $A_n$ is $0$.
		Suppose $\nu_x$ is a family of probability measures produced by Theorem~\ref{T: Furst}. Identify $\Pro$ with $S^1$ by considering the elements of elements of $\Pro$ as unit vectors in $\R^2$. Let $\epsilon>0$ be arbitrary and fix $x\in\T^2$. Further suppose that $\mathcal{U}\subset\Pro$, $\mathcal{U}=(a-\epsilon, a+\epsilon)$ such that $\nu_x (\mathcal{U})>0$. We can select $B_k \in SL_2(\R)$ such that $B_k$ contracts $\mathcal{U}$ into $\mathcal{U}$ by mapping $a-\epsilon$ to $a-\epsilon/k$, and  $a+\epsilon$  to $a+\epsilon/k$. In other words, we claim we can find $B_k\in SL_2(\R)$ such that
		\begin{equation}\label{e:matrixconds}
			B_k \mathcal{U}\subset \mathcal{U}\,,
			\quad
			a\in B_k\mathcal{U}\,,
			\quad\text{and}\quad
			\operatorname{Leb}
			\left(B_k\mathcal{U}\right)<\frac{\epsilon}{k}\,,
		\end{equation}
		for an arbitrary $k\in\mathbb{N}$. Set $A_k=\left(B_k^T\right)^{-1}$.
		By controllability of the rescaled derivative process (see Remark~\ref{r:derivcontrol}), we can reach $\left(x, A_k\right)$ 
		from $\left(x, \mathit{Id} \right)$
		in finite time. Therefore, $(x, A_k)$ satisfies condition~\eqref{e:supplaw} and so equation~\eqref{e:projectivemeasures} for each $B_k$ becomes
		\[
		\nu_x \left(\mathcal{U}\right)=\nu_x \left(B_k \mathcal{U}\right).
		\]
		Since $k$ was arbitrary 
		\[
		\nu_x(\{a\})=\nu_x(\mathcal{U})>0.
		\] Select any $c\in\mathcal{U}$, $c\neq a$. Analogous to \eqref{e:matrixconds}, we can select $C_k$ such that
		\begin{equation}
			C_k \mathcal{U}\subset \mathcal{U}\,,
			\quad c\in C_k\mathcal{U} \,,
			\quad\text{and}\quad
			\operatorname{Leb}\left(C_k\mathcal{U}\right)<\frac{\epsilon}{k}\,.
		\end{equation}
		Thus, again $\nu_x(\{c\})=\nu_x(\mathcal{U})>0$, but $\nu_x(\mathcal{U})\geq\nu_x(\{a\})+\nu_x(\{c\})$.
		So we have reached a contradiction and the top Lyapunov exponent of the rescaled derivative process is not~$0$.
		The top Lyapunov exponent cannot be negative due to the Kingman Subadditive Ergodic Theorem~\cite{Kingman1973}.
		This implies the top Lyapunov exponent of the one point process is strictly positive.
	\end{proof}

\subsection{Feller Property}
	To prove that the modified, randomly shifted, sawtooth shears are exponentially mixing it remains to show that that the one point, projective and two point processes are Feller. Recall, a process is Feller if the function $x\mapsto \EX^x[f(X_1)]$ is continuous for every bounded continuous function~$f$. 

	\begin{lemma}\label{L: mufeller}
		The one point process~\eqref{e:AlexsGreatestEquation}, projective process~\eqref{e: projpros}, and two point process~\eqref{def: 2pointp} are Feller.
	\end{lemma}

	\begin{remark}
		No randomness is needed to show that the one point and two-point processes are Feller. We will, however, use randomness to show that the projective process is Feller. 
		We will capitalize on the fact that the parameter $\alpha$ in~\eqref{e:psiDef} is uniformly distributed on $[0, 1]$. 
	\end{remark}

	\begin{proof}[Proof of Lemma~\ref{L: mufeller}]
		The velocity fields given in~\eqref{e:valpha1} and~\eqref{e:valpha2} are uniformly Lipschitz. Thus, the two point process satisfies a bound similar to equation~\eqref{e:lip} in Remark~\ref{r:lip}. Therefore, it is Feller. Similarly the one point process is Feller.

		We will now show that the projective process is Feller. Note that $v_{\alpha, i}(x)$, and, therefore, the time $1$ flows $\varphi^{\beta v_{\alpha, i}}(x)$ and its derivatives $D_x \varphi^{\beta v_{\alpha, i}}(x)$ are smooth  on $\T^2$ except for three lines: $x_i=\alpha$, $x_i=\alpha+1/4$, and $x_i=\alpha+3/4$. 
		In order to compute the expectation
		\begin{equation}\label{e:Fel}
			\EX^{(x, u)}[f(\phi_1(x), D_x\phi_1 u)]= \EX^{(x, u)}[f(\varphi^{V_1}(x), D_x\varphi^{V_1} u)],
		\end{equation}
		we need to fix $(x, u)\in\T^2\times\Pro$ 
		and integrate over the 
		parameter space, 
		$(\alpha, \beta, i) \in [0, 1]\times[0, 1]\times \{1, 2\}$. The function 
		$f(\phi_1(x), D_x\phi_1(x) u)$ is uniformly bounded. It is also
		continuous everywhere except 
		$x_i=\alpha$, $x_i=\alpha+1/4$, and $x_i=\alpha+3/4$. Therefore its integral 
		with respect to the parameter $\alpha$ is continuous. Thus the 
		expectation~\eqref{e:Fel} is continuous with respect to $x$ and $u$.

		We now do the formal proof that the projective process is Feller. Fix $f\in\mathcal{C}(\T^2\times\Pro)$ and let $\epsilon>0$ be arbitrary. Let $d(\cdot, \cdot)$ be the product distance on $\T^2\times\Pro$. First select $\delta_1>0$ such that $6 C\delta_1<\epsilon/2$, where $C$ is such that $f(z)\leq C$ for all $z\in\T^2 \times\Pro$.
		Now select $0<\delta_2<\delta_1$ so that if $(x, u), (y, u')\in\T^2\times\Pro$ are such that $d((x, u), (y, u'))<\delta_2$, and $x$ and $y$ are in the same smooth piece of $\beta v_{\alpha, i}$ we have that
		\begin{equation}\label{e: feller}
			\lvert f(\phi^{\beta v_{\alpha, i}}_t (x), D_x \phi^{\beta v_{\alpha, i}}_t u)-f(\phi^{\beta v_{\alpha, i}}_t(y), D_x \phi^{\beta v_{\alpha, i}}_t u')\rvert<\frac{\epsilon}{2}
		\end{equation}
		for all $\alpha, \beta\in[0, 1]$ and $i\in\{1, 2\}$. Let $(x, u),~(y, u')\in \T^2\times\Pro$ be such that $d((x, u), (y, u'))<\delta_2$ and 
		\[
		A=\{(\alpha, \beta, i)\in[0, 1]\times[0, 1]\times\{0, 1\} \st \hbox{ equation~\eqref{e: feller} does not hold}\}\,.
		\]
		Note that $\Prob(A)< 6 \delta_1$ since there are both vertical and horizontal modified shear flows, and each modified shear flow has $3$ different smooth pieces. This gives
		\begin{align*}
			\MoveEqLeft
			\left\lvert \EX^{(x, v)} [f(X_1, V_1)]- \EX^{(y, v')} [f(X_1, V_1)]\right\rvert \\
			&\leq \int_{A} \left\lvert f(\phi_1(x), D_x\phi_1(x)v)-f(\phi_1(y), D_x\phi_1(y)v')\right\rvert d\Prob\\
			&\qquad+\int_{A^{\comp}} \left\lvert f(\phi_1(x), D_x\phi_1(x)v)-f(\phi_1(y), D_x\phi_1(y)v')\right\rvert d\Prob\\
			&< \frac{\epsilon}{2} +6 C\delta_1<\epsilon.
		\end{align*}
		Thus, the projective process is Feller.
	\end{proof}

\subsection{Proof of Theorem~\ref{t:SinMix}}
\begin{proof}[Proof of Theorem~\ref{t:SinMix} with a sawtooth shear profile in two dimensions]
The proces\-ses~\eqref{e:AlexsGreatestEquation}, \eqref{e: projpros}, and \eqref{def: 2pointp} are all irreducible by Lemma~\ref{kappa_deriv}. The three processes are Feller by Lemma~\ref{L: mufeller}. Thus, both the derivative and one point processes are uniformly geometrically ergodic. Furthermore, the one point process has a positive Lyapunov exponent by Lemma~\ref{l: poslyaexpon}. Finally, the two point process is $\mathcal{V}$-geometrically ergodic with respect to a function of the form~\eqref{e: lyapunovfunctionform}. This follows from section~$6$ in~\cite{BedrossianBlumenthalEA22}, and the fact that the two point process~\eqref{e:AlexsGreatestEquation} is irreducible and Feller. Thus, by Proposition~\ref{T:micheleMixing} the modified, randomly shifted sawtooth shears are almost surely mixing.
\end{proof}
We now consider modified, randomly shifted, sine shears. In this case we take the stream function~$F_0$ in~\eqref{e:psi0} as
\[
F_0(x)=\sin(2\pi x),
\]
and define 
\begin{equation}\label{e:sinshearpotential}
F_\alpha(x)=F_0(x-\alpha) \hbox{ for } \alpha\in[0, 1].
\end{equation}
The vector fields~\eqref{e:valpha1} and~\eqref{e:valpha2} are defined using the stream functions given by~\eqref{e:sinshearpotential}. All other definitions are the same.

\begin{proof}[Proof of Theorem~\ref{t:SinMix} with a sine shear profile in two dimensions]
After we prove that both the two point process and rescaled derivative process are irreducible, the proof with a sine stream function is identical to the proof with a tent stream function.
We will now show that the two point process is irreducible.
As in Lemma~\ref{kappa_deriv} we must show that the Lie algebra generated by the vector fields
		\begin{equation}
			\set{(v_{\alpha, i}^T(x), v_{\alpha, i}^T(y))^T}
		\end{equation}
		has dimension $4$ on a dense connected set. We will show that the Lie algebra has dimension $4$ on the set
\begin{equation*}\label{e:sinshearconnectedset}
			M\defeq\mathcal{D}^c- S,
\end{equation*}
where 
\begin{equation*}
S=\set{ x_1= y_1 \hbox{ or } x_2= y_2}\cup\set{x \hbox{ or } y \hbox{ is a critical point of } U}\subset\T^2\times\T^2.
\end{equation*}
By selecting $\alpha=x_2$ and $\alpha=x_1$ we obtain the $2$ linearly independent vectors,
\begin{equation}\label{e:span1}
s_2\defeq -\kappa\begin{pmatrix}
v_{x_2, 2}(x)\\
 v_{x_2, 2}(y)
 \end{pmatrix}=
 \begin{pmatrix}
 2\pi\kappa\\
 0\\
 2\pi\kappa\cos(2\pi(x_2-y_2))+\sin(2\pi(x_2-y_2))\partial_2U(y)\\
 -\sin(2\pi(x_2-y_2))\partial_1U(y)
 \end{pmatrix} 
 \end{equation}
 and
 \begin{equation}\label{e:span2}
s_1\defeq\kappa \begin{pmatrix}
v_{x_1, 1}(x)\\
 v_{x_1, 1}(y)
 \end{pmatrix}=
 \begin{pmatrix}
 0\\
 2\pi \kappa \\
 \sin(2\pi(y_1-x_1))\partial_2U(y)\\
 2\pi\kappa\cos(2\pi(y_1-x_1))-\sin(2\pi(y_1-x_1))\partial_1U(y)
 \end{pmatrix}.
\end{equation}
Using a computer to perform symbolic Gauss-Jordan elimination with $s_2$ and $s_1$ from~\eqref{e:span2} and~\eqref{e:span1} we can reduce the vector
\begin{equation*}
\left[\begin{pmatrix}
v_{\alpha, 1}\\
v_{\alpha, 1}
\end{pmatrix},
\begin{pmatrix}
v_{\beta, 2}\\
v_{\beta, 2}
\end{pmatrix}\right](x,y)
\end{equation*} 
to a vector $L_1$ with $0$ in the first two coordinates. Using trigonometric  identities we can write $L_1$ as a linear combination of trigonometric products in the form
\begin{align}
	\nonumber
	\MoveEqLeft
	L_1=\sin(2\pi(\alpha-x_1))\sin(2\pi(\beta-x_2)) \ell_1+\sin(2\pi(\alpha-x_1))\cos(2\pi(\beta-x_2)) \ell_2
	\\
	\nonumber
	&+\cos(2\pi(\alpha-x_1))\sin(2\pi(\beta-x_2)) \ell_3+\cos(2\pi(\alpha-x_1))\cos(2\pi(\beta-x_2)) \ell_4
	\\
	\nonumber
	&+\sin(2\pi(\alpha-y_1))\sin(2\pi(\beta-y_2)) \ell_5+\sin(2\pi(\alpha-y_1))\cos(2\pi(\beta-y_2)) \ell_6
	\\
	\label{e:trigprods}
	&+\cos(2\pi(\alpha-y_1)\sin(2\pi(\beta-y_2)) \ell_7+\cos(2\pi(\alpha-y_1))\cos(2\pi(\beta-y_2) )\ell_8 \,,
\end{align}
for some explicit vectors $\ell_i$, $i=1, 2,\ldots, 8$.
Similarly we can reduce 
\[
\begin{pmatrix}
v_{\alpha, 1}\\
v_{\alpha, 1}
\end{pmatrix},
\begin{pmatrix}
v_{\alpha, 2}\\
v_{\alpha, 2}
\end{pmatrix},
\left[\begin{pmatrix}
v_{\alpha, 1}\\
v_{\alpha, 1}
\end{pmatrix},
\begin{pmatrix}
v_{\beta, 1}\\
v_{\beta, 1}
\end{pmatrix}\right](x,y),
\hbox{ and }
\left[\begin{pmatrix}
v_{\alpha, 2}\\
v_{\alpha, 2}
\end{pmatrix},
\begin{pmatrix}
v_{\beta, 2}\\
v_{\beta, 2}
\end{pmatrix}\right](x,y)
\]
to the vectors
\begin{gather*}
	L_2=\sin(2\pi(\alpha-x_1))\ell_9\,,
	\quad L_3=\sin(2\pi(\alpha-x_2))\ell_{10}\,,
	\\
	L_4=\sin(2\pi(\alpha-\beta)\ell_{11}\,,
	\quad
	L_5=\sin(2\pi(\alpha-\beta))\ell_{12}\,,
\end{gather*}
respectively, for some explicit vectors~$\ell_9$, \dots, $\ell_{12}$, which all have a $0$ in their first two coordinates.

On the set $M$ since $x_1\neq y_1$ and $x_2\neq y_2$, the trigonometric products in~\eqref{e:trigprods} are linearly independent unless 
\[
x_1=y_1+\frac{1}{2} \quad\text{and}\quad x_2=y_2+ \frac{1}{2}  \,.
\]
We first consider the case when $x_1\neq y_1+1/2$ or $x_2\neq y_2+1/2$.
In this case we can use different choices of~$\alpha, \beta$ in~\eqref{e:trigprods} and take linear combinations to ensure each $\ell_i$ is contained in the Lie algebra at $(x, y)$.
In particular, the vectors~$\ell_5$ and~$\ell_8$ are contained in the Lie algebra, and we symbolically compute
		\begin{equation}\label{e:sin2ptvectors}
			\frac{\ell_5}{4\pi^2 \kappa^2 }=\begin{pmatrix}
				0\\
				0\\
				\partial_1 U(y)\\
				-\partial_2 U(y)
			\end{pmatrix},
			\quad
			\frac{\ell_8}{4\pi^2 \kappa^2}=\begin{pmatrix}
				0\\
				0\\
				-\partial_2 U(y)\\
				\partial_1 U(y)
			\end{pmatrix}\,.
		\end{equation}
Thus the Lie algebra has dimension $4$ unless
\[
\partial_1 U(y)=-\partial_2 U(y) \quad\text{or}\quad \partial_1 U(y)=\partial_2 U(y) \,.
\]
Symbolic computations (see~\cite{ChristieFengEA23}) show that in either case, the vector
\[
\begin{pmatrix}
	0\\
	0\\
	0\\
	1
\end{pmatrix}\,
\]
is contained in the span of the remaining $\ell_i$. 
Since in~$M$ both $\partial_1 U(y)$ and $\partial_2 U(y)$ can not vanish, using~$\ell_5$ or~$\ell_8$ from~\eqref{e:sin2ptvectors} shows that the Lie algebra has dimension~$4$ as desired.

Now consider the case where $x_1= y_1+1/2$ and $x_2= y_2+1/2$.
We rewrite~\eqref{e:trigprods} as 
\begin{align*}
L_1&=\sin(2\pi(\alpha-x_1))\sin(2\pi(\beta-x_2)) \ell'_1+\sin(2\pi(\alpha-x_1))\cos(2\pi(\beta-x_2)) \ell'_2
\\
	&\quad+\cos(2\pi(\alpha-x_1))\sin(2\pi(\beta-x_2)) \ell'_3+\cos(2\pi(\alpha-x_1))\cos(2\pi(\beta-x_2)) \ell'_4 \,,
\end{align*}
for some explicit vectors~$\ell_1'$, \dots, $\ell_4'$.
We symbolically compute and verify
\[
\begin{pmatrix}
0\\
0\\
16 \pi^3 \kappa^2\\
0
\end{pmatrix}=\ell_{11}\kappa^2 + \frac{\ell'_3}{\kappa}
\quad\text{and}\quad
\begin{pmatrix}
0\\
0\\
0\\
16 \pi^3 \kappa^2\\
\end{pmatrix}=\ell_{12}\kappa^2 -\frac{\ell'_2}{\kappa} \,,
\]
which implies that the Lie algebra is $4$ dimensional on $M$ if  $x_1= y_1+1/2$ and $x_2= y_2+1/2$. Thus, in either case the Lie algebra is~$4$ dimensional, and hence the two point process is irreducible.
\smallskip

We now show that the rescaled derivative process~\eqref{e: derivpros} is irreducible by showing that the Lie algebra at $(x, \mathit{Id})$ is $5$ dimensional on the set
\[
	M' \defeq \set{ x\in\T^2 \st \grad U(x) \neq 0 } \,.
\]
Since~$M'$ is connected, this will imply irreducibility of the rescaled derivative process exactly as in the proof of Lemma~\ref{kappa_deriv}. 

We consider the vector fields~\eqref{e:derivprocessvectors} with a sine stream function. That is, vectors with first two coordinates coming from the coordinates of~\eqref{e:valpha1} or~\eqref{e:valpha2} with a sine stream function, and last four coordinates coming from the corresponding matrix, $A_t(v, x)$ (defined in~\eqref{e:matrixAirred}).
By choosing $\alpha=x_2$, and $\alpha=x_1$ in~\eqref{e:derivprocessvectors} we compute
\begin{equation}\label{e:derivspan}
\tilde{s}_2\defeq -2\kappa \tilde{v}_{x_2, 2}=
 \begin{pmatrix}
 4 \pi \kappa\\
 0\\
 -2\pi \partial_1 U(x)\\
  -4\pi\partial_2 U(x)\\
  0\\
  2\pi\partial_1 U(x)
 \end{pmatrix} 
 \hbox{ and }
 \tilde{s}_1\defeq2\kappa \tilde{v}_{x_1, 1}=
 \begin{pmatrix}
 0\\
 2 \pi \kappa \\
 \pi\left(\partial_2U(x)-\partial_1 U(x)\right)\\
 0\\
 -4 \pi\partial_1 U(x)\\
 -2\pi\left(\partial_2U(x)-\partial_1 U(x)\right)
 \end{pmatrix}.
\end{equation}
By performing symbolic Gauss-Jordan elimination on the vectors $\tilde{v}_{\alpha, 1}$, $\tilde{v}_{\alpha, 2}$, $\tilde{s}_1$ and $\tilde{s}_2$, yields the vectors 
\[
\tilde{L}_1=\sin(2\pi(\alpha-x_1))\tilde{\ell}_1
\quad\text{and}\quad
\tilde{L}_2=\sin(2\pi(\alpha-x_1))\tilde{\ell}_2
\]
for some explicit vectors~$\tilde \ell_1$ and~$\tilde \ell_2$ whose first two coordinates are both~$0$.
We symbolically compute
\[
	\tilde s_3 \defeq 2(\tilde{\ell}_1-\tilde{\ell}_2)
	=
\begin{pmatrix}
0\\
0\\
1\\
2\\
2\\
1
\end{pmatrix} \,.
\]
Using the $\tilde{s}_i$ we symbolically compute and reduce the vectors 
\[ 
[\tilde{v}_{\alpha, 1}, \tilde{v}_{\beta, 2}], [\tilde{v}_{\alpha, 1}, \tilde{v}_{\beta, 1}], \hbox{ and } [\tilde{v}_{\alpha, 2}, \tilde{v}_{\beta, 2}]
\]
to the vectors 
\begin{align*}
	\tilde{L}_3
		&=\sin(2\pi(\alpha-x_1))\sin(2\pi(\beta-x_2)) \tilde{\ell}_3+\sin(2\pi(\alpha-x_1))\cos(2\pi(\beta-x_2)) \tilde{\ell}_4
\\
		&\qquad +\cos(2\pi(\alpha-x_1))\sin(2\pi(\beta-x_2)) \tilde{\ell}_5+\cos(2\pi(\alpha-x_1))\cos(2\pi(\beta-x_2)) \tilde{\ell}_6 \,,
	\\
	\span \tilde{L}_4 =\sin(2\pi(\alpha-\beta))\tilde{\ell}_7\,,
	\qquad\text{and}\qquad
	\tilde{L}_5=\sin(2\pi(\alpha-\beta))\tilde{\ell}_8 \,,
\end{align*}
respectively, for some explicit vectors $\tilde \ell_3$, \dots, $\tilde \ell_8$, which all have zeros in the first three coordinates.
By choosing $\alpha = x_1 + 1/4, x_1 - 1/4, x_2, x_1, x_2 - 1/4$, and taking Lie brackets, we obtain three more vectors
\begin{equation}\label{e:finalsindegen}
\tilde{v}_{x_1-1/4, 1}\,,
\quad [\tilde{v}_{x_1-1/4, 1}, \tilde{v}_{x_2, 2}]\,,
\quad [\tilde{v}_{x_1-1/4, 1}, \tilde{v}_{x_2-1/4, 2}]
\quad\text{and}\quad
[\tilde{v}_{x_1, 1}, \tilde{v}_{x_2, 2}] \,.
\end{equation}
Again using the $\tilde{s}_i$ and performing symbolic Gauss-Jordan elimination we obtain three new vectors, $\tilde{\ell_9}$, $\tilde{\ell}_{10}$, $\tilde{\ell}_{11}$, and $\tilde{\ell}_{12}$ which have zeros in the first three coordinates.
We now symbolically compute and check that the span of~$\operatorname{span}\set{\tilde \ell_1, \dots, \tilde \ell_{12} }$ also contains the vectors
\[
	\tilde \ell_{13} \defeq \begin{pmatrix}
0\\
0\\
0\\
\frac{1}{2}\left(8\pi^2 \kappa^2+\partial_2 U(x)^2-6\partial_1 U(x)\partial_2 U(x)+5\partial_1 U(x)^2\right)\\
2\partial_1 U(x)\left(\partial_1 U(x)-\partial_2 U(x)\right)\\
\partial_1 U(x)\left(\partial_1 U(x)-\partial_2 U(x)\right)
\end{pmatrix}
\]
and
\[
	\tilde \ell_{14} \defeq \begin{pmatrix}
0\\
0\\
0\\
\tilde \ell_{14,4}
\\
\partial_1 U(x)\partial_2 U(x)+\kappa\partial_{1,2} U(x)\\
\partial_1 U(x)\partial_2 U(x)+\kappa\partial_{1,2} U(x)
\end{pmatrix}\,,
\]
where
\begin{equation}
	\tilde \ell_{14,4} = \frac{1}{2}\bigl(
		-4\pi^2\kappa^2-\partial_2 U(x)^2-\kappa\lap U(x) +
			2\partial_2 U(x)\partial_1 U(x)-\partial_1 U(x)^2+2\kappa\partial_{1, 2} U(x)
		\bigr) \,.
\end{equation}
Clearly, the two vectors~$\tilde \ell_{13}$, $\tilde \ell_{14}$ are linearly independent unless
\[
\partial_1 U(x)\partial_2 U(x)+\kappa\partial_{1,2} U(x)=0~\hbox{ or }~\partial_1 U(x)\left(\partial_1 U(x)-\partial_2 U(x)\right)=0 \,.
\]
In either case we symbolically compute and explicitly verify that the span of the $\tilde{\ell}_i$ is $2$ dimensional (see~\cite{ChristieFengEA23} for details).
This shows that the Lie algebra is $5$ dimensional, and hence the rescaled derivative process is irreducible.
\end{proof}
\begin{remark}
If the potential is constant then the vectors~\eqref{e:sin2ptvectors} do not prove irreducibility of the two point process. However, similar symbolic computations can still be used to show that the two point process is irreducible.
\end{remark}

When $d>2$, we perform the $2$-dimensional dynamics on pairs of coordinates. That is, for a pair of indices $(i, j)$, $i<j$ we consider the velocity fields 
 \begin{gather}
		\label{e:highdvalphai}
		v_{\alpha, (i, j), 1}=
		\frac{1}{p}\grad^\perp_{i,j}(p F_\alpha(x_i))\,,
		\\
		\label{e:highdvalphaj}
		v_{\alpha, (i, j), 2}=
		\frac{1}{p}\grad^\perp_{i,j}(p F_\alpha(x_j))
		\,,
\end{gather}
where $\grad^\perp_{i,j}$ is the $2$-dimensional perpendicular gradient on the coordinate pair $(i, j)$ and $0$ in the other coordinates. We define
\[
V_n=\beta_n v_{\alpha_n, \xi_n, i_n}\,,
\] where  $\xi_n$ are i.i.d.\ uniform random variables on the set of pairs of coordinates, and $(\alpha_n, \beta_n, i_n)$ are again i.i.d.\ random variables uniformly distributed on the parameter space $[0, 1]\times [0, 1] \times \set{1, 2}$. All of the associated processes are defined identically to their $2$-dimension counterparts (see~\eqref{e:AlexsGreatestEquation}, \eqref{e: projpros}, and~\eqref{e: derivpros}).

\begin{proof}[Proof of Theorem~\ref{t:SinMix} in dimension $d \geq 3$]
The proof in high dimensions follows from the proof done in $2$ dimensions. Both the Feller property (Lemma~\ref{L: mufeller}) and positivity of the top Lyapunov exponent (Lemma~\ref{l: poslyaexpon}) are exactly the same as the proofs done in dimension $2$. Irreducibility of the one point and two point processes also follows immediately from Lemma~\ref{kappa_deriv}. For each pair of coordinates we can perform the same computations done in Lemma~\ref{kappa_deriv}. This implies that the span of the vector fields that generate the two point process has dimension $2d$, and so the one point and two point processes are irreducible.  

The added assumption in Theorem~\ref{t:SinMix} for high dimensions that~\eqref{e:highddegeneracycond} holds on a small ball, $B(\hat{x}, \hat{\epsilon})$ allows us to generalize the proof of irreducibility of the derivative and projective processes. Irreducibility follows if we show irreducibility on the subset $B(\hat{x}, \hat{\epsilon})$. Indeed, since the one point process is irreducible we can first move the position component into $B(\hat{x}, \hat{\epsilon})$, control the derivative without moving the position component out of $B(\hat{x}, \hat{\epsilon})$, and then move the position and derivative pair to the end condition. Thus, irreducibility on the set $B(\hat{x}, \hat{\epsilon})$ implies irreducibility of both processes.

Irreducibility of the processes on the set $B(\hat{x}, \hat{\epsilon})$ follows from counting the dimension of the Lie algebra of the rescaled derivative process. Consider the $d$-dimension counterparts to~\eqref{e:matrixAirred} and~\eqref{e:derivprocessvectors}. We will refer to the coordinate that corresponds to the $(i, j)$ entry of $A_t$ as coordinate $(i, j)$. Since the potential is separable on $B(\hat{x}, \hat{\epsilon})$, the Lie bracket of the $d^2+d$-dimensional counterparts to~\eqref{e:derivprocessvectors} can only be non-zero in 6 positions. That is, for a pair of coordinates $i<j$
\[[\tilde{v}_{\alpha, (i, j), 1}, \tilde{v}_{\alpha, (i, j), 2}]\]
can be non-zero in the $i$, $j$, $(i, i)$, $(j, j)$, $(i, j)$, and  $(j, i)$ coordinates. Thus, for each pair of coordinates we can perform the same computations done in the $2$ dimensional cases. 

First, consider the case of tent shears. Clearly, each coordinate gives rise to a linearly independent vector (vectors $1$ and $2$ in~\eqref{e:derivvectors}). For each pair $i<j$ of coordinates we obtain $2$ vectors with nonzero entries in the $(i, j)$ and $(j, i)$ coordinates (vectors $3$ and $4$ in~\eqref{e:derivvectors}) that are not in the span of the first $d$ vectors. This gives $d^2$ linearly independent vectors in total. The computations in Lemma~\ref{kappa_deriv} for the coordinate pair $i, i-1$ add another $d-1$ vectors to the collection of linearly independent vectors (5\textsuperscript{th} vector found in Lemma~\ref{kappa_deriv}). Thus the dimension of the Lie algebra of the rescaled derivative process is $d^2+d-1$ on $B(\hat{x}, \hat{\epsilon})\times\mathit{Id}$ and so both the derivative and projective processes are irreducible (see Lemma~\ref{kappa_deriv} for details). The conclusion of Theorem~\ref{t:SinMix} now follows from the same exact argument done in the earlier parts of this section. Counting the dimension of the Lie algebra in the case of sine shears is similar, and omitted for brevity.
\end{proof}

	\appendix

	\section{Lebesgue Exponential Mixing of Randomly Shifted Tent Flows}\label{s:appendixmixing}
	Studying mixing rates of incompressible flows is an area of active research~\cite{
		Bressan06,
		CrippaDeLellis08a,
		Thiffeault12,
		Seis13,
		IyerKiselevEA14},
	and examples of exponentially mixing flows were only constructed recently~\cite{
		YaoZlatos17,
		AlbertiCrippaEA19,
		ElgindiZlatos19,
		BedrossianBlumenthalEA22,
		BlumenthalCotiZelatiEA22,
		Cooperman22
	}.
	We note that when~$F$ is the sawtooth shear~\eqref{e:psi0} (shown in Figure~\ref{f:sawTooth}), then Theorem~\ref{t:SinMix} still holds when~$U = 0$.
	In this case, $\mu$ is simply the Lebesgue measure, and hence the flow~$v$ (defined in~\eqref{e:vDefAltShear}, with~$F$ as in~\eqref{e:psi0}) is almost surely (Lebesgue) exponentially mixing.
	The proof of Theorem~\ref{t:SinMix}, however, involves technical calculations to check irreducibility.
	If instead of using the sawtooth profiles~\eqref{e:psi0}, we use a localized tent function, then the proof of irreducibility becomes extremely simple.
	We devote this appendix to presenting this, and hence obtain a short, simple, proof showing a family of incompressible flows is almost surely (Lebesgue) exponentially mixing.


	\begin{figure}[htb]
		\centering
		\includegraphics[width=.5\linewidth]{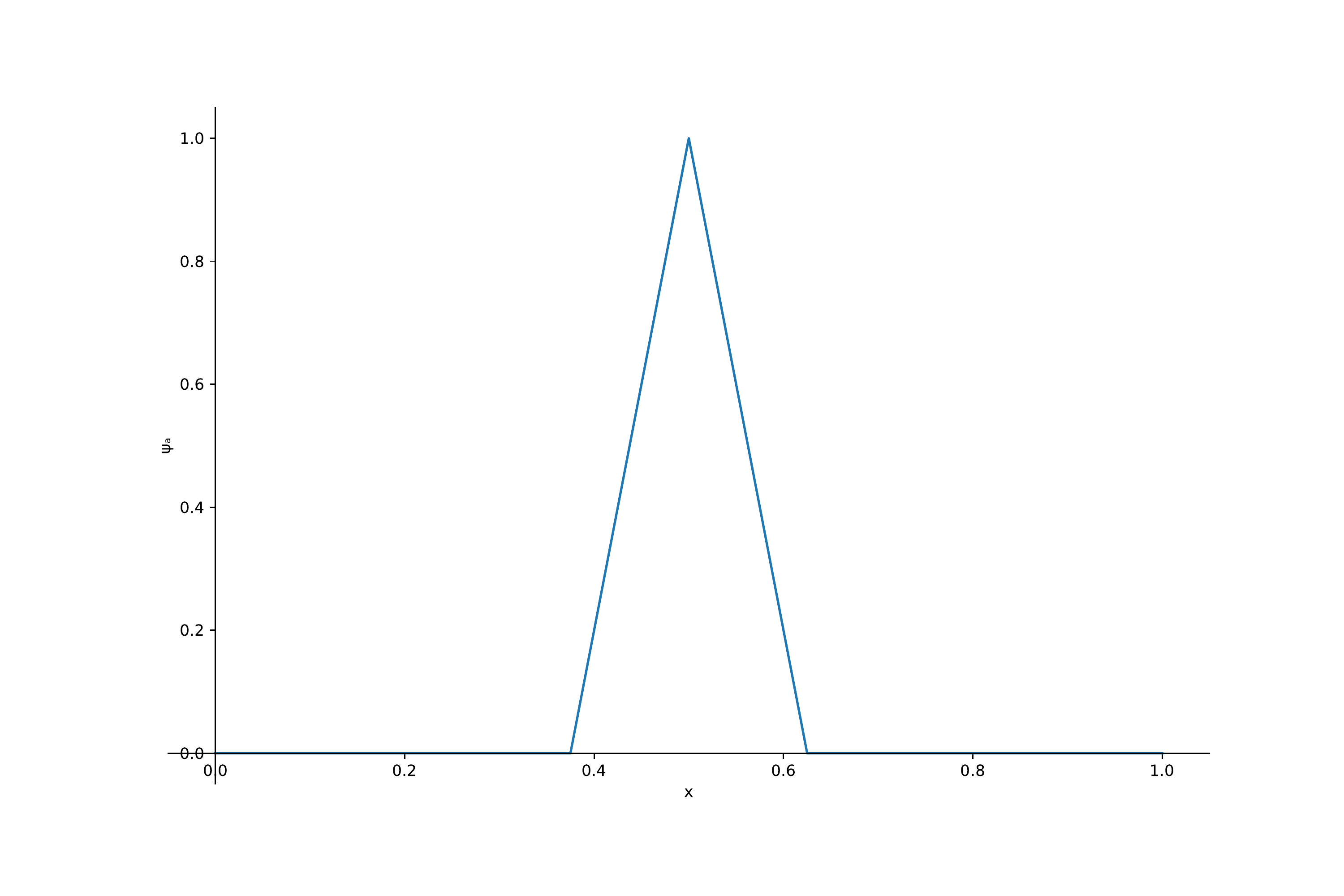}
		\caption{Profile of the localized tent function~$\psi_0$.}
		\label{fig:M1}
	\end{figure}
	Let $F_0$ be the localized tent function shown in Figure~\ref{fig:M1}.
	Explicitly, $F_0$ a piece-wise linear periodic function such that
	\begin{equation}\label{e:localTent}
		F_{0}\left(\frac{1}{2}\pm \frac{1}{8}\right)=0 \,,
		\quad
		F_{0}\left(\frac{1}{2}\right)=1 \,,
		\quad\text{and}\quad
		F_{0}(x)=0 \hbox{ for } x\in\Bigl[\frac{3}{8}, \frac{5}{8}\Bigr]^{\comp} \,.
	\end{equation}
	For $\alpha\in[0, 1]$ define 
	\begin{equation}\label{e:psilebesgue}
		F_{\alpha}(x)=F_0(x-\alpha)\,,
	\end{equation}
	so that $F_{\alpha}$ is a localized tent function with peak at $1/2 + \alpha$.
	\GI[2023-03-28]{Peak is at $1/2 + \alpha$ not~$\alpha$}
	Define the associated horizontal and vertical shear flows by
	\begin{equation}\label{e:localtentvectors}
		v_{\alpha, 1}(x)=
		\begin{pmatrix}
			F(x_2 - \alpha)\\
			0
		\end{pmatrix}\,,
		\quad\text{and}\quad
		v_{\alpha, 2}(x)=
		\begin{pmatrix}
			0\\
			F(x_1 - \alpha)
		\end{pmatrix}\,.
	\end{equation}
	respectively.
	We will now show that if we randomly shift (and modulate) the localized tent shears flows~$v_{\alpha, i}$ then we are (Lebesgue) exponentially mixing almost surely.

	\begin{theorem}\label{t: lebmixingflow}
		Define the velocity field~$v$ by
		\begin{equation}
			v_t = \beta_n v_{\alpha_n, i_n}
			\qquad\text{when } t \in [n, n+1)\,,
		\end{equation}
		where $(\alpha_n, \beta_n, i_n)$ are i.i.d.\ random variables that are uniformly distributed on parameter space $[0, 1]\times [0, 1] \times \set{1, 2}$.
		The flow of~$v$ is almost surely (Lebesgue) exponentially mixing with rate~\eqref{e:expMixRate}.
	\end{theorem}
	\begin{remark}
		If instead of localized tent shears~\eqref{e:localTent}, we use sawtooth shears, then Theorem~\ref{t:SinMix} already guarantees~$v$ is almost surely (Lebesgue) exponentially mixing.
		The reason we choose localized tent shears here is because the proof of irreducibility is simpler, and does not rely on the Rachevsky--Chow Theorem (Theorem~\ref{T: chow}).
		The local property makes the argument direct and highlights the fact that for a collection of vector fields to be almost surely mixing you must be able to control $2$ points independently.
	\end{remark} 
	\begin{remark}
		In dimensions~$d > 2$, Lebesgue exponentially mixing flows can be constructed by setting
		\begin{equation}
			v_t = \grad^\perp_{i_n,j_n} F_{\alpha_{i_n}}(x_{j_n})\,,
		\end{equation}
		where~$\grad^\perp_{i,j}$ is the skew gradient in the $x_i$-$x_j$ plane (see~\eqref{e:GradPerpij}).
	\end{remark}

	\begin{proof}[Proof of Theorem~\ref{t: lebmixingflow}]
		As in Section~\ref{s:modifiedmixing}, define~$V_n$ by~\eqref{e:VnDef}, with~$v_{\alpha, i}$ as in~\eqref{e:localtentvectors},
		and consider the Markov process~$X_n$ defined by~\eqref{e:AlexsGreatestEquation}.
		Using the same notation as Section~\ref{s:modifiedmixing}, define the random flow~$\phi_n$ by~\eqref{e:phidef}.

		We will prove Theorem~\ref{t: lebmixingflow} by showing the conditions of Proposition~\ref{T:micheleMixing} hold.
		As in Section~\ref{s:modifiedmixing} we need to show that the one point process~\eqref{e:AlexsGreatestEquation} has a positive Lyapunov exponent, and that the one point process~\eqref{e:AlexsGreatestEquation}, two point process~\eqref{def: 2pointp}, and projective process~\eqref{e: projpros} are all irreducible and Feller. 

		Since each velocity field is a shear the one point process is clearly irreducible. The existence of a positive Lyapunov exponent is immediate due to the classical result of Furstenberg~\cite{Furstenberg63}, which we quote below.
		The version we state is Theorem 4.1 in~\cite{BougerolLacroix85}.
		\begin{theorem}\label{T: ogfurst}
			Let $G_n$ be a sequence of IID random matrices in $\SL_2(\R)$ with probability measure $\mu$. If the smallest closed subgroup containing the support of $\mu$ is not compact and $G_n$ does not leave a set of one or two lines invariant then there exists a constant $\lambda^+>0$ such that for any $u\in\R^2$,
			\[
			\lim_{n\to +\infty} \frac{1}{n}\ln \norm{ G_n G_{n-1}\cdots G_1 u}=\lambda^+,
			\]
			almost surely.
		\end{theorem}
		Since $D_x\phi_n$ is independent of $x$ it is a random product of shear matrices and so Theorem~\ref{T: ogfurst} implies that the Lyapunov exponent is positive.
		(We remark that this argument is not specific to the localized tent shears, and the randomly shifted sawtooth shears also have a positive Lyapunov exponent for the same reason.)

		Both processes are Feller by the same argument done in Lemma~\ref{L: mufeller}.
		We show irreducibility of the two point process and of the projective process in the following two Lemmas.
		\begin{lemma}\label{l:leb2pointirr}
			The two point process~\eqref{def: 2pointp} is irreducible. 
		\end{lemma}
	\begin{lemma}\label{l:lebprojirr}
		The projective process~\eqref{e: projpros} is irreducible.
	\end{lemma}

		This concludes the proof of Theorem~\ref{t: lebmixingflow}, modulo the above lemmas.
	\end{proof}

	It remains to prove Lemmas~\ref{l:leb2pointirr} and~\ref{l:lebprojirr}.
	We recall that the proof of irreducibility for Theorem~\ref{t:SinMix} (Lemma~\ref{kappa_deriv}), involved the Rachevsky--Chow theorem, and technical calculations that were checked symbolically.
	For localized tent flows and the Lebesgue measure, the proofs are short and simple.

	\begin{proof}[Proof of Lemma~\ref{l:leb2pointirr}]
		We show that given $(x, y), (w, z)\in\mathcal{D}^c$ we can find a sequence parameters 
		\[
		\alpha_n, \beta_n\in[0, 1], \hbox{ }i_n\in\{1, 2\}
		\] so that the composition of the time $1$ flows of the vector fields $\beta_n v_{\alpha_n, i_n}$, defined in~\eqref{e:localtentvectors}, map $x$ to $w$ and $y$ to $z$. Continuity in $\alpha$ and $\beta$ then implies irreducibility of the two point process.
		The key observation of the proof is that the vector fields~\eqref{e:localtentvectors} can move the coordinates of two points $(x, y)\in\mathcal{D}^c$ independently.

		Fix position $(x, y), (w, z)\in\mathcal{D}^c$ and let $d(\cdot, \cdot)$ be the distance on $S^1$.
		Suppose without loss of generality that $x_2\neq y_2$. Then we can fine an open interval $I=(a, b)\subset [0, 1]$ so that for all $\alpha\in I$ 
		\[
		  F_\alpha(x_2)=8(x_2-\alpha) \quad\text{and}\quad F_\alpha(y_2)=0\,.
		\]
		Furthermore $F_\alpha(x_2)$ at one of the endpoints of $I$ is $0$. That is, $F_{a}(x_2)=0$ or $F_{b}(x_2)=0$. Thus, we can find a vector field from~\eqref{e:localtentvectors} whose time one flow translates $x_1$ without changing $x_2$, $y_1$, and $y_2$. By repeatedly selecting $\alpha\in I$, we can map $x_1$ to $w_1$ without changing any other coordinate. By repeating this process with each coordinate we can map $(x, y)$ to $(w, z)$ as desired.
	\end{proof}
	\begin{proof}[Proof of Lemma~\ref{l:lebprojirr}]
		We show that for any two elements, $(x, u), (y, u')\in\T^2\times\Pro$, there is a sequence of $\alpha_n, \beta_n\in[0, 1]$, $i_n\in\{1, 2\}$ such that the composition of the time $1$ flows of the vector fields, $\beta_n v_{\alpha_n, i_n}$, map $x$ to $y$. Furthermore, the derivative of the composition of time one flows map $u$ to $u'$. Continuity in $\alpha$ and $\beta$ implies irreducibility of the projective process.
		The lemma follows by observing that for any $x\in\T^2$ we can find $v_n$, defined in~\eqref{e:localtentvectors}, such that the derivative matrix at $x$ is given by 
		\begin{equation}\label{e:lebderivmatrices}
			A_{\beta} = \begin{pmatrix}
				1 & 8\beta\\
				0  & 1
			\end{pmatrix}
			\quad\text{or}\quad
			B_{\beta} = \begin{pmatrix}
				1 & -8\beta\\
				0  & 1
			\end{pmatrix},
		\end{equation}
		which is notably independent $\alpha$. 
		Consider the projective process as elements of $\R^2 - \{0\}$ under the equivalence relation $u\sim u'$ if and only if $u=cu'$ for some constant $c$.
		For $u\in\Pro$, $u\neq (1, 0)$ we can rescale $u$ so that the second coordinate is $1$. Fix $(x, u), (y, u')\in \T^2\times\Pro$. Without loss of generality assume that neither $u$ nor $u'$ point in the direction $\bm e_1\defeq (1, 0)$.
		Indeed, suppose that they they do. Then we can select a vertical shear so that $(x, u)$ is mapped to say, $(x_0, u_0)\in \T^2\times\Pro$ where $u_0\neq \bm e_1$. Similarly, we can select a vertical shear that maps $(y_0, u_0')\in \T^2\times\Pro$, $u_0'\neq \bm e_1$, to $(y, u')$. By mapping $(x_0, u_0)$ to $(y_0, u_0')$ we have a sequence of shears that maps $(x, u)$ to $(y, u')$. We may further assume without loss of generality that $u_1+8<u'_1$. This is possible by repeatedly selecting vector fields whose derivative at $x$ is $B_\beta$, defined in~\eqref{e:lebderivmatrices}.

		We now give the proof. First map $x_2$ to $y_2$ without mapping $u$ to $\bm e_1$. This is possible since for any $\beta\in[0, 1]$ there exists $\alpha$ so that the time $1$ flow of $\beta v_{\alpha, 2}$ maps $x_2$ to $y_2$. Select a vector field $v=\beta v_{\alpha, 2}$ such that \[D_x \varphi^v(x) u=A_{\beta} u= d=\begin{pmatrix}
			u_1+8\beta\\
			1
		\end{pmatrix}\] and $d_1-u'_1\equiv 0 \mod 8$. By repeatedly selecting vector fields so that the derivative is always $A_1$ we can map $(x, u)$ to \[\left(z, \begin{pmatrix}
			u'_1-8\\
			1
		\end{pmatrix}\right)\] for some $z\in\T^2$ such that $z_2=y_2$. Now select $\alpha$ such that $\varphi^{ v_{\alpha, 1}}(z)=y$ and $D_x \varphi^{ v_{\alpha, 1}} =A_1$. The sequence of parameters, $\alpha$, $\beta$, $i$, provide a realization of the projective process which maps $(x, u)$ to $(y, u')$.
	\end{proof}

	\bibliographystyle{halpha-abbrv}
	\bibliography{refs,preprints}

\end{document}